\documentclass[a4paper,12pt,reqno]{amsart}
\usepackage[left=3cm,right=3cm,top=3cm,bottom=3cm]{geometry}

\usepackage{amsmath,amssymb,mathrsfs,amsthm,stmaryrd}
\usepackage{mathtools} 
\usepackage{empheq}
\usepackage{enumerate,enumitem}
\usepackage{dsfont}
\usepackage[hidelinks]{hyperref} 

\usepackage{esint} 

\usepackage{todonotes}

\theoremstyle{plain}
\newtheorem{lemma}{Lemma}[section]
\newtheorem{proposition}{Proposition}[section]
\newtheorem{corollary}{Corollary}[section]
\newtheorem{theorem}{Theorem}[section] 
\newtheorem{assumption}{Assumption}[section]

\theoremstyle{remark}
\newtheorem{remark}{Remark}[section]

\theoremstyle{definition}
\newtheorem{notation}{Notation}[section]
\newtheorem{definition}{Definition}[section]
\newtheorem{claim}{Claim}[section]

\makeatletter

\@addtoreset{equation}{section}
\makeatother
\providecommand{\msc}[1]{{\small \textit{Mathematics Subject Classification ---} #1}}

\usepackage{enumitem,amssymb}
\newlist{todolist}{itemize}{2}
\setlist[todolist]{label=$\square$}
\usepackage{pifont}
\title[Stochastic Ericksen-Leslie equation]{Existence, uniqueness and regularity for the stochastic Ericksen-Leslie equation}
\date{\today}

\author{Anne De Bouard}
\address[A.~De Bouard]{CMAP, Ecole polytechnique, CNRS, Institut Polytechnique de Paris, 91128 Palaiseau, France}
\email{anne.debouard@polytechnique.edu}

\author{Antoine Hocquet}
\address[A. Hocquet]{Institut f\"ur Mathematik,  Technische Universit\"at Berlin, Stra\ss e des 17. Juni 136, D-10623 Berlin, Germany}
\email{antoine.hocquet@wanadoo.fr}

\author{Andreas Prohl}
\address[A.~Prohl]{Mathematisches Institut, Universit\"at T\"ubingen, Auf der Morgenstelle 10, D-72076  T\"ubingen, Germany}
\email{prohl@na.uni-tuebingen.de}

\thanks{Financial support by the DFG via Research Unit FOR 2402 is gratefully acknowledged.}

\DeclareMathOperator*{\esssup}{ess\,sup}
\DeclareMathOperator*{\spt}{Supp}
\DeclareMathOperator{\id}{id}

\renewcommand{\div}{\ensuremath{\mathrm{div}}}
\renewcommand{\d}{\ensuremath{\mathrm{\hspace{0.035em}d}}} 

\newcommand{\N}{\ensuremath{\mathbb N}} 
\newcommand{\R}{\ensuremath{\mathbb R}} 

\renewcommand{\O}{\ensuremath{{\mathbb T^2}}}
\renewcommand{\P}{\ensuremath{\mathbb P}}
\newcommand{\E}{\ensuremath{\mathbb E}}
\newcommand{\C}{\ensuremath{\psi}}
\newcommand{\CC}{\ensuremath{{\boldsymbol\psi}}}
\newcommand{\ee}{\ensuremath{{\mathbf e}}}
\newcommand{\vs}{\ensuremath{{\Hat v}}}
\newcommand{\Bus}{\ensuremath{\mathbf{\Hat v}}}
\newcommand{\us}{\ensuremath{{\Hat u}}}
\newcommand{\Bzs}{\ensuremath{\mathbf{\Hat u}}}
\newcommand{\zetas}{\ensuremath{{\Hat\zeta}}}
\newcommand{\Bzetas}{\ensuremath{\boldsymbol{\Hat\zeta}}}
\newcommand{\Ws}{\ensuremath{{\Hat W}}}
\renewcommand{\tt}{\ensuremath{\mathscr T}} 
\newcommand{\ttt}{\ensuremath{\mathscr {\hat T}}} 
\newcommand{\X}{\ensuremath{\mathfrak X}} 

\newcommand{\T}{\ensuremath{{I}}}
\newcommand{\V}{\ensuremath{\mathscr V}}
\newcommand{\U}{\ensuremath{\mathscr U}}
\renewcommand{\L}{\ensuremath{\mathbb L}} 
\newcommand{\W}{\ensuremath{\mathbb W}\hspace{0.035em}} 
\newcommand{\LL}{\ensuremath{\mathbf L}} 
\newcommand{\WW}{\ensuremath{\mathbf W}} 
\renewcommand{\H}{\ensuremath{\mathbb H}}
\newcommand{\HH}{\ensuremath{\mathbf H}}
\newcommand{\loc}{\ensuremath{{\mathrm{loc}}}}
\newcommand{\hs}[1]{\ensuremath{{\mathscr {L}_2 (L^2,H^{#1})}}}
\newcommand{\vvs}{\ensuremath{{\Hat{\Hat{v}}}}}
\newcommand{\uus}{\ensuremath{{\Hat{\Hat{u}}}}}

\begin{document}

\begin{abstract}
We investigate existence and uniqueness for the liquid crystal flow driven by colored noise on the two-dimensional torus.
After giving a natural uniqueness criterion, we prove local solvability in $L^p$-based spaces, for every $p>2.$
Thanks to a bootstrap principle together with a Gy\"ongy-Krylov-type compactness argument, this will ultimately lead us to prove the existence of a particular class of global solutions which are partially regular, strong in the probabilistic sense, and taking values in the ``critical space'' $L^2\times H^1.$
\end{abstract}

\maketitle


\msc{60H15,76A15,35K55,58E20}


\tableofcontents
  
\section{Introduction}
\label{sec:intro}
\subsection{Motivations}
Let $T>0$ and ${\O = (0,1)^2}$.
The simplified incompressible Ericksen-Leslie equations (EL) model the combined dynamics of a director field $u: [0,T]\times \O \rightarrow {\mathbb S}^2$ (also refered to as ``molecular orientation'') and a velocity field $v: [0,T]\times \O \rightarrow {\mathbb R}^2$ of a thermotropic nematic liquid crystal which is contained in the domain $\O$. The director field gives an averaged orientation of the constituent molecules (``mesogens") of the liquid crystal phase to e.g.~predict the evolution of a texture that exhibits several defects, and which could possibly annihilate or nucleate due to a nontrivial fluid flow dynamics \cite{RD1}. In this work, we include a stochastic forcing
term to the simplified Ericksen-Leslie equations to account for thermal fluctuation effects \cite{LM1,BMA1,BHR1}. The noise acts on the molecular orientation so that the system is written as
\begin{equation}
\label{SEL}
\tag{SEL}
\left\{
\begin{aligned}
&\d v - \bigl(\Delta v-v\cdot \nabla v - \div(\nabla u \odot\nabla u ) -\nabla \pi \bigr)\d t = 0
\\
&\quad \div v=0
\\
&\d u - \bigl(\Delta u - v\cdot \nabla u + |\nabla u |^2u \bigr)\d t = \nu u \times\circ \d  W _t
\\
&\quad |u|_{\R^3}=1\enskip \text{a.e.}
\end{aligned}
\right.
\qquad \text{in}\enskip (0,T]\times \O,
\end{equation}
where $\pi$ is the hydrostatic pressure while $\nu \geq 0$ denotes a constant, and we assume that $(v(0),u(0))$ is given with finite energy.
In the above, $ W(\omega,t,x)$ stands for an $\R^3$-valued, spatially correlated Wiener process which models random forces on the mesogens in the liquid crystal. The noise term appears inside a Stratonovitch integral, which allows the process $u \equiv u(\omega,t,x)$ to be ${\mathbb S}^2$-valued. Throughout the paper, we will denote by 
$(\nabla u \odot \nabla u)$ the $2\times 2$ matrix whose entries are given by
\[
(\nabla u \odot \nabla u)_{i,j}= \langle \frac{\partial u}{\partial x_i}, \frac{\partial u}{\partial x_j}\rangle_{{\mathbb R}^3},\quad \text{for}\enskip 1 \leq i,j \leq 2
\,.
\]
The velocity process $v \equiv v(\omega,t,x)$ is divergence-free, and is driven by a non-trivial director field in the random Navier-Stokes equation. 
The system \eqref{SEL} is supposed to model (sall-scale) fluctuating rotations of elongated molecules which are embedded into a liquid; 
a major motivation here is to study the stability of texture in the nematic phase, and mechanisms which trigger changes in the orientation of the director field.

The system (EL) of nonlinear PDEs (i.e., $\nu  = 0$ in (\ref{SEL})) couples the convected harmonic map flow
with the incompressible Navier-Stokes equation through the term $-{\rm div}(\nabla u \odot \nabla u)$. Its non-trivial interplay is essentially due to the latter quadratic term, since weak solutions of the harmonic map heat equation (i.e., $v\equiv 0$ and $\nu = 0$  in (\ref{SEL})$_3$) might exhibit a singular behavior at finite time $T_1 >0$. According to \cite{St1,St2}, such a singularity manifests as concentration of the ``local energy"
$\sup_{x \in {\O}} \frac12\int_{B_{\rho }(x)}(|v(t,y)|^2+|\nabla u(t, y)|^2)\, {\rm d}y$ of the regular local solution as $t$ gets close to $T_1$; 
this local solution may, however, be continued  to a partially regular global weak solution, by choosing the weak limit 
$w-\lim_{t \uparrow T_1}  u(t, \cdot)$ in $H^1(\O,{\mathbb S}^2)$ as new initial value for times $t \geq T_1$. This construction was used by Lin and co-authors in \cite{LLW1,LW2} and Hong in \cite{Ho1}. A partially regular weak solution of (EL) was constructed, for which there exist at most finitely many times $0 < T_1 < \ldots <  T_L < \infty$ such that the local energy concentrates, i.e.,
\begin{equation}\label{el-1}
\liminf_{t \uparrow T_j} \sup_{x \in {\O}} \frac12\int_{B_\rho (x)} \bigl[ \vert v(t, y)\vert^2 + \vert \nabla u(t, y)\vert^2\bigr]\, {\rm d}y \geq \varepsilon _1, \qquad \forall\, \rho  > 0\, ,
\end{equation}
for some geometric quantity $\varepsilon _1>0$ (depending only on the domain).
We also mention the works of Lin and Liu \cite{LL1,LL2} (see also \cite{W1}) where (EL) is approximated via a Ginzburg-Landau penalization term, allowing to relax the constaint $\vert {u}\vert=1.$
In this setting, the third equation in (\ref{SEL}) changes, for some positive $\epsilon$, to
\begin{equation}
\label{el-pen1}
{\rm d}u^{\epsilon} = \Bigl(  \Delta u^\epsilon -[v^{\epsilon} \cdot\nabla] u^{\epsilon} - \frac{1}{\epsilon} f(u^{\epsilon})\Bigr) {\rm d}t + \nu\, u^{\epsilon}  \times \circ {\rm d} W,
\qquad \mbox{\rm in } \O \times (0,\infty)
\end{equation}
where $f(x) = \nabla F(x)$ and $F(x) = \frac{1}{4} \int_{{\O}} \bigl\vert \vert x\vert^2 - 1\bigr\vert^2\, {\rm d}x$. For $\nu = 0$, the director field $u^{\epsilon}$ is more regular, which allows to construct a weak solution by a  Galerkin method through uniform bounds for the related energy \cite{LL1,LL2}; however, passing to the limit $\epsilon \downarrow 0$ is an open problem, mostly due to the quadratic term in the first equation in (\ref{SEL}) that was discussed above; see also \cite[Thm.~7.1]{LL2} for the deterministic case. \\

Solvability of the SPDE given by a stochastic perturbation of the penalized (EL) equation is established in \cite{BHR1}; the authors use the Lyapunov structure of the problem
to construct a global strong solution by continuation of a local in time mild solution, and It\^o's formula
to verify ${\P}$-a.s.
non-negativity of $1- \vert u\vert$ in space and time; a weak martingale solution to the 3D
system is constructed by a Galerkin method in \cite{BHR2}, and it is is shown to be the unique strong
solution in the two-dimensional setting. All these results were obtained for $\epsilon >0$ fixed. \\

In this work, we construct a partially regular, global solution for \eqref{SEL},
which is strong in the usual probabilistic sense, and which for any time lies in the critical space $L^2(\O;\R^2)\times H^1(\O;\R^3),$
see Definition \ref{def:mart_sol} below. In addition, we will see that the solutions of \eqref{SEL} are unique under some integrability property which, to the best of our knowledge, is new even in the deterministic context (it suffices to let $\nu =0$ in \eqref{SEL}).
Thanks to a classical interpolation inequality (see \eqref{ineq:interp}), the partially regular solutions constructed above fulfill (correspondingly) the integrability condition \eqref{int_property}, which implies in turn that these solutions are unique in their class (as in the case of the stochastic harmonic map flow in \cite{H1}).

\bigskip

An interesting extension of the model \eqref{SEL} could include an additive noise term $\d V_t$ in the right hand side of the velocity equation.
In fact if $V$ is a divergence-free, $L^2(\O;\R^2)$-Wiener process, it is easy to see that the proofs below adapt mutatis mutandis (modulo notational difficulties).
Since the paper is already rather technical, we here stuck to \eqref{SEL} and leave the proof of such an extension to the reader.
We nevertheless point out that, due to the important scale differences between the velocity and the molecular orientation, our model \eqref{SEL} is physically relevant, as for instance discussed in \cite{LM1}. Similarly, it should be possible to deal with more general domains and boundary conditions, but we chose to restrict ourselves to the torus for simplicity.

\bigskip

\textbf{Organization of the paper.}
Section \ref{sec:notation} introduces the used notations, while in Section \ref{sec:prelim} we define two different notions of solution and provide our main solvability results (theorems \ref{thm:uniqueness} and \ref{thm:main}).
For the reader's convenience, in Section \ref{sec:strategy} we will explain the main ideas in the proof of Theorem \ref{thm:main} and provide a ``sketch of proof".
The proof of Theorem \ref{thm:uniqueness} will be given in Section \ref{sec:uniqueness}.
Section \ref{sec:solvability} shows local solvability for \eqref{SEL}, using a fixed point argument in the case where the data are ``subcritical".
In Section \ref{sec:apriori}, we derive a priori estimates, while a bootstrap principle will be shown in Section \ref{sec:bootstrap}. The proof of Theorem \ref{thm:main} will be addressed in Section \ref{sec:convergence}.
Finally, computational details related to It\^o-Stratonovitch corrections of the form \eqref{ito_strato} will be given in Appendix \ref{app:trace}.

\subsection{Notation}
\label{sec:notation}

Throughout the paper the symbol $T\in(0,\infty)$ refers to a fixed, deterministic time horizon.
If $u,v$ are measurable maps, the notation ``$u\times w$'' refers to the pointwise vector product, namely for each $i\in \{1,2,3\}\simeq\mathbb Z/3\mathbb Z$ we let
$(u\times w)^i:= u^{i +1}w^{i +2}-u^{i +2}w^{i +1}$.

For $p\in[1,\infty],$ we will denote by $L^p$ the usual Lebesgue spaces.
We will make use of the Sobolev-Slobodeckij spaces $W^{\alpha ,p}$ defined as usual for $\alpha \geq 0$ and by duality as $W^{\alpha ,p}=(W^{-\alpha ,\frac{p}{p-1}})^*$ whenever $p\in(1,\infty]$ and $\alpha <0.$ For notational simplicity, we also denote $H^\alpha :=W^{\alpha ,2},$ for $\alpha \in\R.$
Throughout the paper, we denote by $\O\equiv (0,1)^2$ the two-dimensional torus,
and for $d\geq 1$ the notations
\[
L^p(\O;\R^d),\quad W^{\alpha ,p}(\O;\R^d), \quad H^\alpha (\O;\R^d)\,,
\]
will be used as shorthands for the spaces
$L^p_{\mathrm{per}}(\R^2;\R^d),$ $W^{\alpha ,p}_{\mathrm{per}}(\R^2;\R^d),$ $H^\alpha _{\mathrm{per}}(\R^2;\R^d),$
consisting of $1$-periodic elements $f:\R^2\to\R^d$ that belong to the corresponding local Sobolev space (endowed with the appropriate topology).
For any square-integrable $f,g:\O\to\R^d,$ we denote
\[
\langle f,g\rangle:=\int_{\O} f(x)\cdot g(x)\d x\equiv\sum\nolimits_{i=1}^d\int_{(0,1)^2} f^i(x)g^i(x)\d x,
\]
and we denote by the same bracket the bilinear mapping
\[
\langle \cdot ,\cdot \rangle : (W^{\alpha ,p}(\O;\R^d))^*\times W^{\alpha ,p}(\O;\R^d)\longrightarrow \R,\quad  (f,g)\longmapsto \langle f,g\rangle:= f(g).
\]
Given a Banach space $E,$ the space of continuous paths $u:[0,T]\to E,$ endowed with the supremum norm will be denoted by $C(0,T;E).$ If $p\in[1,\infty)$ we similarly denote
$L^p(0,T;E)$ the space of Bochner $p$-integrable functions with values in $E.$

\bigskip

Following \cite{temam1984navier}, we introduce the linear space 
\[
\mathcal V:= \{f\in C^\infty(\O;\R^2),\enskip  \div f=0\}\,.
\]
For $p\in(1,\infty]$ and $\alpha \in\R,$ we denote by $\W^{\alpha ,p}$ the completion of $\mathcal V$ with respect to
the norm of $W^{\alpha ,p}(\O;\R^2),$ and we further let $\H^\alpha :=\W^{\alpha ,2}.$
For $\alpha =0$ we will also use the notation $\L^p:=\W^{0,p}.$ 
Similarly, and in order to distinguish between the velocity component $v$ and the molecular orientation $u,$ we adopt the notations
$\WW^{\alpha ,p}:=W^{\alpha ,p}(\O;\R^3),$
$\HH^{\alpha }:=W^{\alpha ,2}(\O;\R^3),$
and $\LL^p:=L^p(\O;\R^3).$

Given a two-dimensional vector field $f\in L^2(\O;\R^2),$ we denote by
$
\mathscr P f
$
its linear projection onto the space $\L^2$ and the same notation is used for its unique extension as a linear map $\mathscr P:W^{\alpha ,p}(\O;\R^2)\to \W^{\alpha ,p}$ for every $\alpha\in\R$ and $p\in (1,\infty].$
For convenience, note that $\mathscr Pf$ is explicitly given on the torus by the formula
\begin{equation}
\label{fourier}
\hat f(k)= \left(\id-\frac{kk^T}{k_1^2+k_2^2}\right)\hat f(k)\,,\quad k\equiv(k_1,k_2)\in\mathbb Z^2\,,
\end{equation}
where $\hat f(k):= \int_{\O}f(x)\exp(2i\pi k\cdot x)\d x$ denotes the spatial Fourier Transform, and we further recall that $\mathscr P$ is continuous (see for instance \cite{lemarie2002recent}). 

When $p\in(1,\infty]$ we further denote by $A: D(A)\subset \W^{-1,p}\to \W^{-1,p}$ the Stokes operator, that is the operator defined as 
\[
Av=-\mathscr P \Delta v,\quad \text{for}\enskip 
v\in D(A):= \W^{1,p}\,.
\]

If in addition $I$ denotes a fixed time-interval we define the Banach spaces
\begin{equation}
\label{nota:V_U}
\V^{\alpha }_I:= C(I;\H^{\alpha })\cap L^2(I;\H^{\alpha +1}),
\quad \text{and}\quad 
\U^{\alpha }_I:= C(I;\HH^{\alpha })\cap L^2(I;\HH^{\alpha +1}).
\end{equation}
The notation ``$v\in\V_{\mathrm{loc};I}^{\alpha }$'' means that $v$ belongs to $\V_{J}^{\alpha }$ for every compact interval $J\subset I.$ Whenever $I=[0,T]$ for some deterministic $T>0,$ we also use the abbreviation $\V_T^{\alpha }:=\V_{[0,T]}^{\alpha }.$ We will use similar notations for $\U.$

\bigskip

Given a Banach space $B,$ the set of random variables on a probability space $(\Omega ,\mathcal A,\P)$ (i.e.\ measurable maps from $\Omega \to B$ w.r.t.\ $\mathcal A$) will be denoted by $L^0(\Omega;B).$
If $0\leq \tau_1<\tau _2 \leq T$ are stopping times with respect to some filtration on $\Omega$, we use the notation
\[
v\in L^0(\Omega ;\V^\alpha _{\mathrm{loc};[\tau _1,\tau _2)})
\]
if and only if $v(\cdot \wedge\sigma )\in L^0(\Omega ;\V^\alpha _{[\tau _1,T]})$ for every stopping time $\tau _1\leq \sigma <\tau _2,$ and we define the space $\U^\alpha_{\mathrm{loc};[\tau _1,\tau _2)}$ in a similar fashion.
Similarly, if $z:\Omega \times[0,T]\to X$ is a stochastic process with values in some Banach space $X,$
$\tau >0$ is a stopping time, and $q\in[1,\infty],$ we write $z\in L^q(\Omega ; C([0,\tau );X))$ to indicate that the stopped process $z(\cdot \wedge \sigma )$ belongs to $L^q(\Omega;C(0,T;X))$ for any stopping time $\sigma <\tau .$

Given a Hilbert space $K,$ we denote by 
$\mathscr L_2(L^2,K)$ the class of Hilbert-Schmidt linear maps $\Theta :L^2\to K,$ i.e.\ such that
\[
|\Theta |_{\mathscr L_2(L^2,K)}^2:=\sum_{l\in\N}|\Theta f_l|_{K}^2<\infty\,,
\]
where in the sequel $(f_l)_{l\in\N}$ is a given orthonormal basis of the Hilbert space $L^2(\O;\R).$
Similarly, given a Banach space $Y,$
we denote by $\gamma (Y)$ the space of $\gamma $-radonifying operators from $L^2$ to $Y,$ namely $\Theta \in\gamma (Y)$ if and only if
\[
|\Theta |^2_{\gamma (Y)}:= \int_{\tilde\Omega }\left|\sum_{l\in\N}\gamma _l(\tilde \omega )\Theta f_l \right|^2_Y\mathbb{\Hat P}(\d \hat\omega )<\infty\,,
\]
for every independent, identically distributed normal family $(\gamma _l)_{l\in\N}$ on some probability space $(\Hat\Omega ,\mathcal{\Hat A},\mathbb{\Hat P}).$

\subsection{Assumption on the noise and main results}
\label{sec:prelim}

Before we introduce an appropriate notion of solution for \eqref{SEL}, we need to make some assumptions on the noise term.
In what follows, we will denote by $(\Omega,\mathcal A,\P,(\mathcal F_t)_{t\in[0,T]})$ a filtered probability space satisfying the usual assumptions.

\begin{assumption}
\label{ass:W}
Assume that we are given $\C\in\hs{1}$, and let $W:\Omega \times[0,T]\to \HH^1$
be the $(\mathcal F_t)$-adapted Wiener process given by the infinite series
\begin{equation}
\label{nota:eta}
 W(\omega ,t,x)=\sum_{l\in\N}\mathbf B_l(\omega ,t)\C f_l(x)\,,
\end{equation} 
where $(f_l)_{l\in\N}$ denotes a fixed complete orthonormal system for the Hilbert space $L^2(\O;\R),$
while $(\mathbf B_l)_{l\in\mathbb N}$ is a given family of independent and identically distributed Brownian motions in $\R^3.$
\end{assumption}

In the sequel, we will make use of the notation
\begin{equation}
\label{nota:psi_k}
\C_l:=\C f_l\in L^2(\O;\R),
\end{equation}
while for $j=1,2,3,$ we set
\begin{equation}
\label{psi_l_j}
\CC _l^j:=\C_l\ee _j \in \LL^2,
\end{equation} 
where 
\begin{equation}
\label{canonical_R3}
(\ee_1,\ee_2,\ee_3)\enskip 
\text{is the canonical basis of $\R^3.$}
\end{equation} 
Note that Assumption \ref{ass:W} leads to the following expression for the covariance of $W$:
\begin{equation}
\label{covariance_law}
\E\left[\langle  W(t), a \rangle\langle  W(s), b\rangle \right]
= \min(t,s)\sum\limits_{i,j=1}^3\langle\C^* a^i,\C^* b^j \rangle_{L^2(\O;\R)},
\end{equation} 
for all $ a,b\in \LL^2$,
and every $s,t\in [0,T]^2.$

In order to define a solution $(v,u)$ of \eqref{SEL}, it is convenient to switch to an It\^o equation.
For $\Phi \in C^1(L^2;\mathscr {L}_2(L^2,L^2)),$ we have the It\^o-Stratonovitch conversion formula:
\begin{equation}
\label{ito_strato}
\int\Phi (u)\circ \d  W 
=\int\Phi (u)\d  W  + \frac{1}{2}\int\sum_{l\in \N}\sum_{j=1}^3 [\Phi '(u)\cdot (u\times \CC_l^j)\big](\CC_l^j)\d t,
\end{equation}
where $\CC_l^j$ is as in \eqref{psi_l_j} and here ``$\circ$" denotes the usual Stratonovitch product.
In particular, an immediate computation using \eqref{geometric_fact} gives the formula $\int u\times \circ \d W= \int u\times \d W - \int F_\C u \d t$ where we denote
\begin{equation}
\label{ito_correction}
F_\C(x):=-\sum_{l\in\N}\C_l(x)^2,\quad x\in \O.
\end{equation}

\begin{definition}
\label{def:mart_sol}
Let $\mathfrak P$ be a filtered probability space $(\Omega ,\mathcal A,\P,(\mathcal F_t)_{t\in[0,T]}),$ and denote by
$W(\omega ,t,x)$ an $\LL^2$-Wiener process whose covariance verifies \eqref{covariance_law} for some $\C\in\mathscr {L}_2(L^2,L^2).$

We will call $(v,u)$ a \emph{solution} with respect to $(\mathfrak P,W)$ if the following properties are fulfilled:
\begin{enumerate}[label=(\roman*)]
 \item\label{sol_i} $(v,u):\Omega \times [0,T]\to \H^{-1}\times \LL^2$ is a progressively measurable process with respect to the filtration $(\mathcal F_t)_{t\in[0,T]},$ such that in addition
 \[
 v\in L^0(\Omega ;\V^{-1}_{\mathrm{loc};[\tau _1,\tau _2)}),\quad \quad 
 u\in L^0(\Omega;\U^{0}_{\mathrm{loc};[\tau _1,\tau _2)}),\quad \quad 
 \]
 (in particular we have $\div v=0$ by definition of the spaces $\V^{-1}$ and $\H^{-1}$);
 \item\label{sol_ii} the following constraint holds on $u$:
 \begin{equation}
 \label{sphere_constraint}
 |u(\omega ,t,x)|=1\quad \text{for almost every}\quad \omega ,t,x\in \Omega \times[0,T]\times \O;
 \end{equation} 
 \item\label{sol_iii} we have 
 \begin{equation}
\label{moment_drift}
 \E\left[\int_0^T\big(|v\cdot \nabla v|^2_{H^{-2}}+|\nabla u\odot\nabla u|^2_{H^{-1}} + |\Delta u+u|\nabla u|^2|^2_{H^{-1}}\big)\d s\right]
 <\infty\enskip ;
 \end{equation} 
\item\label{sol_iv} for every $t\in[0,T],$ $\P$-a.s.:
\begin{align}
\label{variational_identity_1}
& v(t) - v(0) + \int_{0}^t \big[Av + \mathscr P(v \cdot \nabla v + \div\nabla u \odot \nabla u) \big]\d s = 0\,
\\
 \intertext{and}
 \label{variational_identity_2}
& u(t)-u(0) - \int_{0}^t \big[ \Delta u +  \vert \nabla u\vert^2 u- v \cdot \nabla u -F_\C u\big]\d s
=\int_{0}^t  u\times \d W(s)\,,
\end{align}
where \eqref{variational_identity_1} is to be understood in the Bochner sense in $H^{-2},$ while the integrals in \eqref{variational_identity_2} 
are taken in the Bochner sense (respectively It\^o sense), in $H^{-1}.$
\end{enumerate}

Furthermore, let  $\tau _1$ and $\tau _2$ be two stopping times with respect to $(\mathcal F_t)_{t\in[0,T]}$ such that $\P$-as $0\leq \tau _1<\tau _2\leq T.$
We will say that $(v,u)$ is a \emph{local solution} on $[\tau _1,\tau _2),$ provided \ref{sol_i}, \ref{sol_ii} and \ref{sol_iii} hold, but \ref{sol_iv} is replaced by the relations \eqref{variational_identity_1}-\eqref{variational_identity_2} for $t\in[\tau _1(\omega ),\tau _2(\omega )),$ for $\P$-almost every $\omega \in\Omega .$
\end{definition}

We start by giving a uniqueness result.
\begin{theorem}[Uniqueness]
\label{thm:uniqueness}
Fix a stochastic basis $(\mathfrak P, W)$ as in Definition \ref{def:mart_sol}.
Let $(v_j,u_j),$ $j=1,2,$ be solutions with respect to $(\mathfrak P,W),$ starting from the same initial datum $(v(0),u(0))\in \H^{-1}\times \LL^2.$ Assume that $\C\in\mathscr {L}_2(L^2,L^2)$ and suppose furthermore that for $j=1,2,$ it holds
\begin{equation}
\label{int_property}
v_j \in L^4(\Omega ; L^4(0,T;\L^4)),\quad 
u_j \in L^4(\Omega ; L^4(0,T;\WW^{1,4}))\,.
\end{equation} 

Then, we have $(v_1,u_1)=(v_2,u_2).$
\end{theorem}

Before we state our second result, we need another (stronger) notion of solution.
\begin{definition}[strong solution]
\label{def:strong_sol}
Fix a stochastic basis $(\mathfrak P,W),$ $\mathfrak P=(\Omega ,\mathcal A,(\mathcal F_t)_{t\in[0,T]}),$
and let $\tau _1$ and $\tau _2$ be two stopping times with respect to $(\mathcal F_t)_{t\in[0,T]}$ such that  $0\leq \tau _1<\tau _2\leq T,$ $\P$-a.s.
We say that $(v,u)$ is a \emph{strong solution} of \eqref{SEL} on $[\tau _1,\tau _2)$ if the following holds
\begin{enumerate}[label=(\roman*)]
 \item $(v,u)$ is a local solution on $[\tau _1,\tau _2)$ with respect to $(\mathfrak P,W),$ in the sense of Definition \ref{def:mart_sol};
 \item it has the additional regularity
 \[
 v\in L^0(\Omega ;\V^{0}_{\mathrm{loc};[\tau _1,\tau _2)}),\quad \quad 
 u\in L^0(\Omega;\U^{1}_{\mathrm{loc};[\tau _1,\tau _2)})\,.\]
\end{enumerate}
\end{definition}

We are interested in a special class of local strong solutions, which is uniquely caracterized by the ``forward bubbling'' property. In \cite{H1}, these were referred to as``Struwe solutions'', in analogy with the deterministic theory in \cite{St1}.

\begin{theorem}[Existence and uniqueness of strong solutions]
\label{thm:main}
Let $(v_0, u_0) \in \L^2 \times \HH^{1},$ $\C\in \hs{1},$ and assume that $|u_0|_{\R^3}=1$ almost everywhere.
Fix a stochastic basis $(\mathfrak P,W),$ where $\mathfrak P=(\Omega ,\mathcal A,\P)$ and $(W,\C)$ satisfies Assumption \ref{ass:W}.

There exists a solution $(v,u)$ of \eqref{SEL}, starting from the initial datum $(v_0,u_0),$ and a random variable $J\in L^0(\Omega ;\N)$
such that the following properties hold:
\begin{enumerate}[label=(P\arabic*)]
\item\label{prop:1}
There is a sequence of stopping times $0\equiv\tau _0<\tau _1<\dots <\tau_J\equiv T $ with the property that for each $j\in\{0,\dots,J-1\}$,
\[
(v,u)|_{[\tau _j,\tau _{j+1})}\enskip  \text{is supported in}\enskip  \V_{\loc,[\tau _j,\tau _{j+1})}^{0}\times\U^{1}_{\loc,[\tau _j,\tau _{j+1})},
\]
and has moments of all order in this space.
In particular the solution is \emph{strong} on each subinterval $[\tau _j,\tau _{j+1})$ for $j\in\{0,\dots,J-1\}.$

\item\label{prop:2} 
For each $j\in\{0,\dots,J-1\}$, the random variable $(v(\tau _{j+1}),u(\tau _{j+1}))$ belongs to the space $L^2(\Omega ,\L^2\times \HH^1).$ 
Moreover, for every sequence of stopping times $\{\sigma _j^k,k\in\mathbb{N}\}$ such that $\P(\sigma ^k_j\nearrow \tau _{j+1})=1,$
we have
\[
(v(\sigma ^k_j),u(\sigma ^k_j))
\underset{k\to\infty}{\longrightarrow}
(v(\tau _{j+1}),u(\tau _{j+1}))\enskip \text{weakly in}\enskip L^2(\Omega ;\L^2\times \HH^1)\,.
\]

\item\label{prop:3} 
The solution $(v,u)$ may become singular at $t=\tau _{j+1}-0,$ for each $j\in\{0,\dots,J-1\}.$ Namely, we have the alternative that either
$\P(J=1)=1,$ or $\,\P(J>1)>0$ in which case the solution ``bubbles forward'' in the following sense:
\begin{multline}
\inf_{\rho >0}\sup_{t\nearrow \tau _{j+1}}\sup_{x\in \O}\int_{B(x,\rho )}(|v(t,y)|^2+|\nabla u(t,y)|^2)\d y>0\,,
\\
\P\text{-a.s.\ on}\enskip \{\tau _{j+1}<T\}.
\end{multline}

\item\label{prop:4}
The above solution is unique in the class described by \ref{prop:1},\ref{prop:2} and \ref{prop:3}.
\end{enumerate}
\end{theorem}

\subsection{Strategy of the proof of Theorem \ref{thm:main}}
\label{sec:strategy}

The existence part generalizes arguments for the deterministic case from \cite{LLW1,Ho1,LW2} to the SPDE  (\ref{SEL}), and benefits from \cite{H1}. As it turns out, solutions are arbitrary regular in the space-like variable (as permitted by the data), as could be easily seen by a higher order generalization of Theorem \ref{thm:bootstrap}. As this fact does not play any specific role in the proof of Theorem \ref{thm:main}, we will restrict ourselves to show that, under the assumptions that $(v(0),u(0))\in \H^2\times\HH^3$ and $\C\in\hs{3},$ the full trajectory takes values in $\H^2\times\HH^3,$ up to $t=\zeta (\rho ),$ where $\zeta (\rho )$ denotes the concentration time defined in \eqref{nota:zeta} for any fixed $\rho >0.$ This property will be obtained as a consequence of the bootstrap argument shown in Section \ref{sec:bootstrap}.

Before proving Theorem \ref{thm:main}, we will first need local solvability results in the ``subcritical case'' i.e.\ when
\[
(v_0,u_0)\in \L^p\times \WW^{1,p},\quad \text{and}\quad \C\in\gamma (W^{\alpha ,p})\enskip \text{with}\enskip \alpha >2/p,
\]
for some $p>2.$
For $\C\in \hs{\alpha} ,$ $\alpha $ being sufficiently large, $(v_0,u_0) \in \W^{s-1,p}\times\WW^{s,p}$ with $s\in [1,3]$ and $p>2,$ we show that the problem \eqref{SEL} is locally well-posed, i.e.\ up to some stopping time $\tau ^{s,p}$, and that the trajectories of $(v,u)$ belong to $C([0,\tau ^{s,p});\W^{s-1,p}\times\WW^{s,p}).$
The proof of these facts will be done in Section \ref{sec:solvability} via a fixed point argument and for an appropriate mild formulation for \eqref{SEL}.

The proof of Theorem \ref{thm:main} will be addressed progressively in sections \ref{sec:apriori}, \ref{sec:bootstrap} and \ref{sec:convergence}.
These steps are summarized below.

~\paragraph{\itshape Step 1: A priori estimates.}
In a first step, we shall collect a priori estimates on the ``energy'' $E(t):=\frac12(|v(t)|_{L^2}^2+ |\nabla u(t)|^2_{L^2})$
associated to a strong solution $(v,u).$
To prove this, we use the fact that the quantity
\begin{equation}
\label{conserved_flow}
\frac{1}{2}\left(|v(t)|_{L^2}^2+|\nabla u(t) |^2_{L^2}\right) +\int_0^t(|\nabla v(s)|_{L^2}^2+ |\Delta u+u|\nabla u|^2|_{L^2}^2(s))\d s\,,
\end{equation} 
modulo correction by a suitable semi-martingale term, is preserved along the flow.
In particular, the energy estimate comes with an $L^2_t(L^2_x)$ estimate on the ``tension'' 
\begin{equation}
\label{tension_intro}
\tt:=\Delta u+u|\nabla u|^2\,.
\end{equation} 
We will also derive ``local estimates'' for this energy, where the locality is to be understood in the spatial sense, i.e.\ on small balls $B(x,\rho )$ for $x\in\O$ and $\rho >0.$ These estimate will also prove useful in  Section \ref{sec:positivity}, when we will show that the concentration time is non-trivial for a limit of approximate solutions, i.e.\ $\P(\zeta(\lim v_n,\lim u_n;\rho )= 0)=0.$

~\paragraph{\itshape Step 2: Bootstrap}
In a second step, we will prove a priori estimates that are local \emph{in time}, i.e.\ up to some energy-concentration time $\zeta(v,u;\rho )>0.$
Given a local solution $(v,u,\tau )$ and $\rho >0,$ the energy-concentration time $\zeta (v,u;\rho )$
is defined as the first time $t\in[0,\tau )$ such that there exists $x\in \O$ for which the local energy
\[
\varepsilon(t,\rho ,x):= \frac12\int_{B(x,\rho )}(|v(t,y)|^2+|\nabla u(t,y)|^2)\d y
\]
attains a threshold $\varepsilon_1>0$ (the exact value of $\varepsilon _1$ is related to the optimal constant in \eqref{struwe_intro}).
The core of the argument is then to show that, for this value of $\varepsilon_1$ and for every $t\leq \zeta (v,u;\rho ),$
the $L^2(0,T;\HH^2)$-norm of $u(\cdot \wedge \zeta)$ has moments of arbitrary order, estimated above by a constant that only depends on $E(0),|\C|_{\hs{1}}$ and $\rho$.
This fact will follow from the very definition of $\zeta (v,u;\rho ),$ 
and the following inequality due to Struwe \cite{St1}: if $u\in L^2(0,T;\HH^2),$ it holds
\begin{equation}
\label{struwe_intro}
\iint_{[0,T]\times\O}|\nabla u|^2\d y\d t
\leq \mu _1\left(\sup_{\substack{t\in [0,T]\\ x\in\O}}\int\limits_{y\in B(x,\rho )}|\nabla u(t,y)|^4\d y\right)
\iint\limits_{[0,T]\times\O}\Bigg(|\Delta u|^2+\frac{|\nabla u|^2}{\rho^2}\Bigg)\d y\d t
\end{equation} 
The $L^2_t(\HH^2_x)$-estimate then follows rather easily from \eqref{conserved_flow}, \eqref{struwe_intro} and the fact that, since $\tt\perp u$ (as can be seen from the norm constraint $|u|_{\R^3}=1$ a.e.), it holds 
\[
|\tt|_{\R^3}^2=\tt\cdot \Delta u = |\Delta u|_{\R^3}^2- |\nabla u|_{\R^{3\times 2}}^4\,.
\]

Together with the estimate on the energy $E,$ this will allow  us to obtain suitable bounds in the ``critical space'' $\V_T^0\times\U_T^1,$ locally in time.

~\paragraph{\itshape Step 3: Construction of the local strong solution by approximation}
We will approximate the equation \eqref{SEL} by considering the local solution $(v_n,u_n,\tau _n)$  to a problem with smooth data $(v_n(0),u_n(0);\C_n).$ The existence and uniqueness of such local solutions for each $n\in\N$ is ensured by the above local solvability results.
By the above steps, the image laws of the sequence $\{(v_n,u_n),n\in\N\}$ is weakly compact in the spaces $\V^{\delta -1}_{\loc}\times\U^\delta_{\loc} $ for any $\delta <1,$ and therefore there exists -- up to a change of probability space in the use of the Skorohod embedding theorem -- a limiting process $(\vs,\us)$ which will be shown to provide a local strong solution.
This solution is regular enough to satisfy the assumptions of the uniqueness theorem, and therefore, any jointly converging subsequence $((v_{n_\ell },u_{n_\ell}),(v_{m_\ell },u_{m_\ell}))$ should, in the limit, be supported on the diagonal of $\V^0_\loc\times\U^1_\loc.$ By Gyongy-Krylov Theorem, this will show that the solution is in fact probabilistically strong, up to the singular time $\zeta (v,u;\rho )=\lim_{n\to \infty}\zeta (v_n,u_n;\rho ).$
The local solution $(v,u;\tau )$ will be then constructed by letting $\tau _1:=\lim_{\rho \to0}\zeta (v,u;\rho ),$
and then by induction on each $[\tau _j,\tau _{j+1}),$ by taking as initial datum the weak limit of $(v,u)$ as $t\nearrow\tau _j$ (see the details in Section \ref{sec:end_proof}).

\section{Proof of uniqueness }
\label{sec:uniqueness}
\subsection{Main interpolation inequality}

The following well-known interpolation inequality will be crucial in the sequel:
there exists a constant $\mu _0>0$, such that for every $\phi \in H^1,$
\begin{equation}\label{ineq:interp}
\int_\O|\phi |^4\d x\leq \mu _0\left(\int_{\O}|\phi |^2\d x\right)\left(\int_\O\big[|\phi |^2+|\nabla \phi |^2\big]\d x\right)\,.
\end{equation}
For a proof, we refer e.g.\ to \cite[II Thm.\ 2.2]{ladyzhenskaya1968linear}.

\subsection{Proof of Theorem \ref{thm:uniqueness}}
Let $(v_1,u_1),$ $(v_2,u_2)$ be two solutions with respect to $(\mathfrak P,W),$ starting from the same initial datum.
Denote further by $g=v_1-v_2$, $f=u_1-u_2,$ and $U=u_1+u_2.$
We have the system
\[
\begin{cases}
 \d g +Ag\d t
 = \mathscr P\big[-g\cdot \nabla v_2 - v_1\cdot \nabla g - \div(\nabla f\odot \nabla u_2+\nabla u_1\odot\nabla f)\big]\d t
 \\[0.5em]
 \d f -\Delta f\d t 
 = \big[ f|\nabla u_1|^2 + u_2\nabla f\cdot \nabla U -g\cdot \nabla u_1 - v_2\cdot \nabla f\big]\d t + f\times\circ\d W
 \\[0.5em]
 g(0)=f(0)=0\,.
\end{cases}
\]
Our strategy is to apply the It\^o Formula \cite[Theorem 4.2.5]{prevot2007concise}, for a suitable ``Gelfand triple'' $V\subset H\subset V^*.$

By Definition \ref{def:mart_sol}, the stochastic processes 
\[
\begin{aligned}
Y
&:=\mathscr P\big[-g\cdot \nabla v_2- v_1\cdot \nabla g - \div(\nabla f\odot \nabla u_2+\nabla u_1\odot\nabla f)\big]
\\
\tilde Y
&:=f|\nabla u_1|^2 + u_2\nabla f\cdot \nabla U -g\cdot \nabla u_1 - v_2\cdot \nabla f
\\
Z
&:=u\times(\C (\cdot)) \,,
\end{aligned}
\]
are progressively measurable. Furthermore, using \eqref{moment_drift} we see that
$Z\in L^2(\Omega \times[0,T];\mathscr L_2(L^2,L^2)),$ and thus 
\[
(0,Z)\in L^2(\Omega \times[0,T];\mathscr L_2(U;\LL^2)),\quad \text{where}\enskip 
U=\{0\}\times\LL^2\,.
\]

Next, observe that the shifted operator
\begin{equation}
 \Lambda :=I+A,
\end{equation} 
defines a continuous isomorphism from $\H^{-1}$ to $\H^1$ (as is classical, see e.g.\ \cite{giga1981analyticity}) and introduce the spaces 
$H:=\H^{-1}\times\LL^2,$ $V:=\L^2\times \HH^1,$
where $H$ is endowed with the norm 
\begin{equation}
|(g,f)|_{H}:=|\Lambda ^{-1/2}g|^2_{L^{2}} + |f|^2_{L^2}\,.
\end{equation}
With this definition, the Riesz isomorphism $i:H\to H^*$ allows for the identification
$V^*\simeq\H^{-2}\times \HH^{-1}.$
As a consequence, using \eqref{int_property}, it is easily seen that
$(Y,\tilde Y)\in L^2(\Omega \times[0,T];V^*)$ and therefore we have all in hand to apply
\cite[Theorem 4.2.5]{prevot2007concise}. This yields the relation:
\[
\begin{aligned}
\frac12(|\Lambda ^{-1/2}&
g(t)|_{L^2}^2+|f(t)|_{L^2}^2 )+ \int_0^t(|g|_{L^2}^2+|\nabla f|^2_{L^2})\d s 
\\
&=\iint\nolimits_{[0,t]\times\O}\Big[ (-g\cdot \nabla v_2)\cdot \Lambda ^{-1}g - (v_1\cdot \nabla g)\cdot \Lambda ^{-1}g + (\nabla f\odot\nabla u_2)\cdot \nabla \Lambda ^{-1}g 
\\
&\quad \quad \quad \quad \quad 
\quad 
+ (\nabla u_1\odot\nabla f)\cdot \nabla \Lambda ^{-1}g +|f|^2|\nabla u_1|^2 
\\
&\quad \quad \quad 
+ u_2\cdot f(\nabla f\cdot \nabla U) -f\cdot (g\cdot \nabla u_1) - f \cdot (v_2\cdot \nabla f)
\Big]\d x\d s
\\
&=\sum\nolimits_{\gamma =1}^8\T_\gamma\,,
\end{aligned}
\]
and we now evaluate each term separately.
Concerning the first term, we have using that $\div g=0$
\begin{equation}
\label{pre_I_1}
\begin{aligned}
\T_1
&=
\iint_{[0,t]\times\O} \big(\Lambda ^{-1/2}\partial _i(g^iv_2)\big)(\Lambda ^{-1/2}g)\d x\d s
\\
&\leq C\int_0^t|\partial _i(g^iv_2)|_{W^{-1,4/3}}|g|_{W^{-1,4}}\d s
\\
&\leq C \int_0^t|g|_{L^2}|v_2|_{L^{4}}|g|_{W^{-1,4}}\d s\,.
\end{aligned}
\end{equation} 
Now, observe that an immediate generalization of \eqref{ineq:interp} yields the existence of $\widetilde \mu_0>0$ such that for any $\phi \in W^{-1,4}:$
\begin{equation}
\label{immediate_gen}
|\phi |_{W^{-1,4}}^2\leq \widetilde \mu_0|\phi |_{H^{-1}}|\phi |_{L^2}\,.
\end{equation} 
Whence, making use of Young inequality in \eqref{pre_I_1} together with \eqref{immediate_gen}, one obtains
\[
\begin{aligned}
\T_1&
\leq \epsilon  \int_0^t|g|^2_{L^2}\d s + C(\epsilon ,\widetilde \mu_0)\int_0^t|v_2|_{L^4}^2|g|_{H^{-1}}|g|_{L^2}\d s
\\
&\leq 2\epsilon \int_0^t|g|^2_{L^2}\d s + \tilde C(\epsilon ,\widetilde \mu_0)\int_0^t|v_2|_{L^4}^4|g|_{H^{-1}}^2\d s\,,
\end{aligned}
\]
for any $\epsilon >0.$
Similar computations give for the second term
\[
\begin{aligned}
\T_2
&=-\iint_{[0,t]\times\O}\big(\Lambda ^{-1/2}\partial _i(v_1^ig)\big)(\Lambda ^{-1/2}g)\d x\d s
\\
&\leq C\int_0^t|v_1|_{L^4}|g|_{L^2}|g|_{W^{-1,4}}\d s
\\
&\leq 
\epsilon \int_0^t|g|^2_{L^2}\d s + C(\epsilon ,\widetilde\mu _0)\int_0^t|v_1|_{L^4}^2|g|_{H^{-1}}|g|_{L^2}\d s\,.
\\
&\leq
2\epsilon \int_0^t|g|^2_{L^2}\d s + \tilde C(\epsilon ,\widetilde\mu _0)\int_0^t|v_1|_{L^4}^4|g|_{H^{-1}}^2\d s\,,
\end{aligned}
\]
$\epsilon >0$ being arbitrary.
For the third term, we have using \eqref{immediate_gen}:
\[
\begin{aligned}
\T_3
&\leq 
\int_0^t |\nabla f|_{L^2}|\nabla u_2|_{L^4}|\nabla \Lambda ^{-1}g|_{L^4}\d s
\\
&\leq \epsilon \int_0^t|\nabla f|^2_{L^2}\d s+C(\epsilon )\int_0^t|\nabla u_2|^2_{L^4}|g|_{W^{-1,4}}^2\d s
\\
&\leq \epsilon \int_0^t|\nabla f|^2_{L^2}\d s+C(\epsilon ,\widetilde\mu _0)\int_0^t|\nabla u_2|^2_{L^4}|g|_{H^{-1}}|g|_{L^2}\d s
\\
&\leq \epsilon \int_0^t(|\nabla f|^2_{L^2}+|g|^2 )\d s+\tilde C(\epsilon ,\widetilde\mu _0)\int_0^t|\nabla u_2|^4_{L^4}|g|_{H^{-1}}^2\d s\,,
\end{aligned}
\]
for any $\epsilon >0,$
and the same computations as above yield for the fourth term:
\[
\begin{aligned}
\T_4
&\leq C\int_0^t|\nabla u_1|_{L^4}|\nabla f|_{L^2}|g|_{W^{-1,4}}\d s
\\
&\leq \epsilon \int_0^t|\nabla f|^2_{L^2}\d s + C(\epsilon ,\widetilde\mu_0)\int_0^t|\nabla u_1|_{L^4}^2|g|_{H^{-1}}|g|_{L^2}\d s
\\
&\leq  \epsilon \int_0^t(|\nabla f|^2_{L^2}+|g|^2 )\d s+C(\epsilon ,\widetilde\mu _0)\int_0^t|\nabla u_1|^4_{L^4}|g|_{H^{-1}}^2\d s\,.
\end{aligned}
\]

Next, using this time \eqref{ineq:interp}, we have for any $\epsilon >0:$
\[
\begin{aligned}
\T_5
&\leq \int_0^t|f|_{L^4}^2|\nabla u_1|_{L^4}^2\d s
\\
&\leq \mu _0 \int_0^t|f|_{L^2}(|f|_{L^2}+ |\nabla f|_{L^2})|\nabla u_1|^2_{L^4}\d s
\\
&\leq \epsilon \int_0^t|\nabla f|_{L^2}^2\d s+ C(\epsilon ,\mu _0)\int_0^t(|\nabla u_1|^2_{L^4}+|\nabla u_1|^4_{L^4})|f|_{L^2}^2\d s
\\
&\leq \epsilon \int_0^t|\nabla f|_{L^2}^2\d s+ \tilde C(\epsilon ,\mu _0)\int_0^t(1+|\nabla u_1|^4_{L^4})|f|_{L^2}^2\d s
\end{aligned}
\]
Similarly, using $|u_2|_{\R^3}=1$ and \eqref{ineq:interp}, we have
\[
\begin{aligned}
\T_6
&\leq\int_0^t|f|_{L^4}|\nabla f|_{L^2}|\nabla U|_{L^4}\d s
\\
&\leq \epsilon \int_0^t|\nabla f|^2_{L^2}\d s + C(\epsilon )\int_0^t|f|^2_{L^4}|\nabla U|_{L^4}^2
\\
&\leq \epsilon \int_0^t|\nabla f|^2_{L^2}\d s + C(\epsilon ,\mu _0)\int_0^t|f|_{L^2}(|f|_{L^2}+ |\nabla f|_{L^2})|\nabla U|_{L^4}^2
\\
&\leq 2\epsilon \int_0^t|\nabla f|^2_{L^2}\d s + \tilde C(\epsilon ,\mu _0)\int_0^t(1+|\nabla U|_{L^4}^4)|f|^2_{L^2}\d s
\end{aligned}
\]
while
\[
\begin{aligned}
\T_7
&\leq \int_0^t|f|_{L^4}|g|_{L^2}|\nabla u_1|_{L^4}\d s
\\
&\leq \epsilon \int_0^t|g|^2_{L^2}\d s + C(\epsilon ,\mu_0)\int_0^t|f|_{L^2}(|f|_{L^2}+ |\nabla f|_{L^2})|\nabla u_1|^2_{L^4}\d s
\\
&\leq \epsilon \int_0^t(|g|^2_{L^2}+|\nabla f|_{L^2}^2)\d s + \tilde C(\epsilon ,\mu_0)\int_0^t(1+|\nabla u_1|^4_{L^4})|f|_{L^2}^2\d s
\end{aligned}
\]

Using again \eqref{ineq:interp} with H\"older and Young Inequalities, we have finally
\[
\begin{aligned}
\T_8
&\leq \int_0^t|f|_{L^4}|v_2|_{L^4}|\nabla f|_{L^2}\d s
\\
&\leq \epsilon \int_0^t|\nabla f|^2_{L^2}\d s
+C(\epsilon )\mu _0\int_0^t|f|_{L^2}(|f|_{L^2} +|\nabla f|_{L^2})|v_2|^2_{L^4}\d s
\\
&\leq 2\epsilon \int_0^t|\nabla f|^2_{L^2}\d s
+C(\epsilon ,\mu _0)\int_0^t(1 +|v_2|^4_{L^4})|f|_{L^2}^2\d s\,.
\end{aligned}
\]

\textit{Conclusion.}
Since $|\Lambda ^{-1/2}g|_{L^2}$ and $|g|_{H^{-1}}$ are equivalent quantities, we see that provided $\epsilon >0$ is chosen sufficiently small, the summation of all the above contributions leads to the relation
\[
\Psi(t)\leq C\int_0^t\sum_{j=1,2}(1+|v_j(s)|^4_{L^4}+|\nabla u_j(s)|_{L^4}^4)\Psi(s)\d s
\]
where 
$\Psi(t):=\sup_{s\in[0,t]}\frac12(|\Lambda ^{-1/2}g(s)|^2_{L^2} +|f(s)|^2_{H^1})$ and $C>0$ is a universal constant.
Applying Gronwall Lemma for $\P$-a.e.\ $\omega\in\Omega ,$ we find that $f=g=0.$ This ends the proof of Theorem \ref{thm:uniqueness}.
\hfill\qed

\section{Local solvability} 
\label{sec:solvability}
This section is devoted to the proof of existence and uniqueness of local solutions.
In order to express this as a fixed point problem, it will be convenient to switch to a mild form for \eqref{SEL}. As seen for instance in \cite[Section 6]{DPZ}, mild solutions that are sufficiently regular (in a sense made precise below) are also strong solutions in the usual sense.

We shall say that a triplet $(v,u;\tau )$ is a local mild solution to \eqref{SEL} if $\tau >0$ denotes a stopping time, such that for some $p>2,$ the following holds:
\begin{enumerate}[label=(M\arabic*)]
 \item \label{cond:I}
  for any stopping time $0<\zeta <\tau $ the couple
 $(v(\cdot \wedge\zeta ),u(\cdot \wedge\zeta ))$ is progressively measurable as a process with values in $\L^p\times\WW^{1,p}$;
 \item \label{cond:II}
$\P$-almost surely on $\{t<\tau \}$:
\begin{align}
\label{mild_v}
v(t)&=e^{-tA}v_0+\int_0^te^{A(s-t)}\mathscr P(-v\cdot \nabla v-\div(\nabla u\odot\nabla u))\d s
\\
\label{mild_u}
u(t)&=e^{t\Delta }u_0+\int_0^te^{\Delta (t-s)}(u|\nabla u|^2-v\cdot \nabla u-F_\C u)\d s+\int_0^t e^{\Delta (t-s)}\left[u\times \d  W\right],
\end{align}
where the above correspond to Bochner integral in $\L^p,$ respectively in $\WW^{1,p},$ and It\^o integral in $\W^{1,p}.$
\end{enumerate}
As will be seen below, the condition \ref{cond:I} ensures the summability of the integrals in \ref{cond:II}.

Our main result in this section is the following.

\begin{theorem}
\label{thm:localSolv}
Let $p\in(2,\infty)$, $s\in[1,3],$ fix 
$\alpha >2/p,$ and $q>2/\alpha .$
For every $\C\in \gamma (W^{s+\alpha ,p})$ and $(v_0,u_0)\in \W^{s-1,p}\times \WW^{s,p}$ with $|u_0|_{\R^3}=1$ almost everywhere, there exists a local mild solution $(v,u;\tau^{s,p} )$ to \eqref{SEL}, unique in the space
\[
L^q\left(\Omega;C\big([0,\tau^{s,p});\W^{s-1,p} \big)\right)\times L^q\left(\Omega ;C\big([0,\tau ^{s,p});\WW^{s,p}\big)\right)\,.
\]
The existence time $\tau ^{s,p}>0$ is maximal in the sense that
\[
\text{on}\enskip \{\tau ^{s,p}<T\}\text{:}\quad 
\limsup_{t\to\tau ^{s,p}}\max\left(|v(t)|_{W^{s-1,p}},|u(t)|_{W^{s,p}}\right)=\infty.
\]

Moreover, if $p>2$ and $s\geq 2,$ then $u$ satisfies the spherical constraint \eqref{sphere_constraint} for a.e.\ $(\omega ,t,x)\in\Omega \times[0,T]\times \O$ such that $0\leq t< \tau^{p,s}(\omega ).$
\end{theorem}

\begin{remark}
\label{rem:sphere}
In particular, by Theorem \ref{thm:localSolv} if $p>2$ and $s\geq 2,$ then $(v,u)$ takes values in the domain of the linear part, and thus $(v,u)$ is a strong solution in the usual analytic (and probabilistic) sense, see \cite[Chapter 6]{DPZ}.
It is also a solution in the sense of Definition \ref{def:mart_sol} (the properties \ref{sol_i}, \ref{sol_iii} and \ref{sol_iv} are immediate in this case).
\end{remark}

Prior to proving Theorem \ref{thm:localSolv}, we need some preparatory steps.

\subsection{Hypercontractivity bounds}
We will make use of the following well-known inequality. Let $k\in\mathbb{Z}$ and assume that $\Lambda :D(\Lambda )\subset H^k\to H^k,$ is a negative self-adjoint operator such that for any $1<p<\infty,$ the semigroup $e^{t\Lambda }$ extends canonically to a strongly continuous semigroup on $W^{k,p},1<p<\infty.$
Suppose in addition that for some constant $M>0,$ the following estimate holds:
\begin{equation}
\label{resolvent}
|(\Lambda -\lambda )^{-1}|_{\mathscr L(H^k;H^k)}\leq \frac{M}{\lambda },\quad \text{for}\enskip \lambda >0.
\end{equation}
Then, for any $t\in(0,1],$ $k-2\leq \alpha ,\beta \leq k+2$ and $1<p,q<\infty$ it holds the estimate
\begin{equation}
\label{hypercontractivity}
|e^{t\Lambda }|_{\mathscr L(W^{\alpha ,p},W^{\beta ,q})}\leq \frac{C}{t^{\frac{\beta -\alpha}{2} +\frac1p-\frac1q}},
\end{equation}
where the constant $C$ depends only on the largest spectral value of $\Lambda .$
Note that such semigroups are in some references referred to as ``hypercontractive'' or ``$L^p$-contractive'' -- see, e.g., \cite[Theorem X.55]{reed1975methods}.

\begin{proof}[Proof of \eqref{hypercontractivity}.]
The inequality \eqref{hypercontractivity} is well-known in principle, hence we only provide references.
The case $p=q=2$ is treated for instance in \cite[Theorem 5.2]{pazy1983semigroups}.
The general case follows by interpolation and duality. We refer the reader e.g. to the proof given in \cite[p.\ 25]{rothe1984global}.
\end{proof}

Note that $-A$ (resp.\ $\Delta $) with $k=1$ (resp.\ $k=2$) satisfies the above assumptions (the property \eqref{resolvent} for the Laplacian is standard, while for the stokes operator we refer to \cite{giga1981analyticity}).

\subsection{Stochastic parabolic estimates}
We now recall some well-known stochastic parabolic estimates: consider the solution, of the equation
\begin{equation}
\label{eq:parabolic_sto}
  \begin{cases}
   \d Z-\Delta Z\d t=\Psi(t)\d \xi
   \\
   Z(0)=0,
  \end{cases}
\end{equation} 
(It\^o sense)
where $\xi \equiv \sum_{k\in\N}\mathbf B_k(t)f_k$ is a cylindrical Wiener process 
and the unknown is a continuous process $Z:\Omega \times[0,T]\to H.$
Under suitable assumptions on $\Psi$ (see the proposition below) the solution of \eqref{eq:parabolic_sto} is written as the stochastic convolution process:
\begin{equation}
 \label{sto_conv}
Z(t):=\int_0^te^{(t-s)\Delta} \Psi(s)\d \xi (s),\quad t\in[0,T].
\end{equation} 
The following result is proven in \cite{brzezniak1997stochastic}, see also 
\cite{krylov1996theory}.

\begin{proposition}
\label{pro:parabolic_sto}
  Let $\alpha \geq0,p\in[2,\infty),r\geq1,$ and assume we are given a progressively
  measurable process $\Psi:\Omega\times[0,T]\to \mathscr L_2(L^2,L^2)$ such that $\Psi$ belongs to
  $L^r\big (\Omega;L^r(0,T;\gamma (W^{\alpha ,p}))\big ).$
The following holds:
\begin{enumerate}[label=(\roman*)]
 \item for each $p>2,$ every $\delta \in[0,1-2/r)$ and $\lambda \in[0,1-1/r-\delta /2),$
 the stochastic convolution \eqref{sto_conv} is well-defined and belongs to 
 $L^r(\Omega;C^\lambda(0,T;W^{\alpha +\delta ,p}) ).$
 Moreover, it holds
 \[
 \E\left[\|Z\|^r_{C^\lambda (0,T;W^{\alpha +\delta ,p})}\right]
 \leq C\E\left[\|\Psi\|^r_{L^r(0,T;\gamma (W^{\alpha ,p}))}\right].
 \]
 
 \item For each $p\geq2,\delta\in(0,1),$ $Z$ is well-defined and belongs to
 $L^r(\Omega;L^r(0,T;W^{\alpha +\delta,p })).$ It holds as well:
 \[
 \E\left[\|Z\|_{L^r(0,T;W^{\alpha +\delta,p })}^r\right]\leq C\E\left[\|\Psi\|^r_{L^r(0,T;\gamma (W^{\alpha ,p}))}\right].
 \]
\end{enumerate}
The above constants do not depend on $\Psi$ in the indicated classes. 
\end{proposition}

We can now proceed to the proof of the main theorem.

\subsection{Proof of Theorem \ref{thm:localSolv}}
Let $p>2.$ We shall first let $s=1$,
and show that the conclusions of the above theorem hold in this particular case.
The proof is based on a contraction mapping principle for a truncated version of \eqref{SEL},
into the Banach space
\begin{equation}
\label{nota:X}
\X_{q,T}:= L^q\left(\Omega;C\big(0,T;\L^p\big)\right)\times L^q\left(\Omega ;C\big(0,T;\WW^{1,p}\big)\right)\,,
\end{equation}
where $T>0$ and $q\geq 2$ are parameters to be fixed later. It is endowed with the product norm $\interleave(X_1,X_2) \interleave_{\X_{q,T}}:= \interleave X_1 \interleave_{q,T,\L^p}+\interleave X_2 \interleave_{q,T,\WW^{1,p}}$ where for convenience, whenever $E$ is a Banach space and $X:\Omega \times[0,T]\to E$ is a stochastic process, we shall denote
\[
\interleave X\interleave_{q,T,E}:= \E\left[\sup_{t\in[0,T]}\big|X(t)\big|^q_{E}\right]^{1/q}.
\]

Because the noise term cannot be estimated pathwise, we define a cut-off function $\theta \in C_c^\infty((0,\infty),\R)$, such that
\begin{equation}\label{nota:theta}
\mathrm{Supp}\theta\subset(0,2),\quad
0\leq \theta \leq 1\quad 
\text{and}\quad \theta(x)=1\quad\text{for all}\enskip  0\leq x\leq1\,,
\end{equation}
and for $R>0$, $x\in\R^+$, we denote
\[
\theta_R(x)=\theta\Big(\frac{x}{R}\Big)\,.
\]
Next, we fix $X_0\equiv(v_0,u_0)\in \L^p\times\WW^{1,p}$ and
solve a problem where the non-linearity is truncated,
that is: given $R>0$ for any $(w,y)\in \X_{q,T},$ we define
the map 
\[
\Gamma_{X_0,R}:\X_{q,T}\to \X_{q,T}\,,\quad (w,y)\mapsto \Gamma_{X_0,R} (w,y):=(v,u)\,,
\]
where for every $0\leq t\leq T,$ a.s.:
\begin{equation}
\label{equation:v_u}
\begin{aligned}
v(t)
&=e^{-tA}v_0 -\int_0^te^{(s-t)A}\mathscr P\big[\Theta  _R\big(w\cdot \nabla w +\div(\nabla y\odot\nabla y)\big)\big]\d s
\\
u(t)
&=e^{t\Delta }u_0+
\int_0^te^{(t-s)\Delta }\big[\Theta _R\big(y|\nabla y|^2+w\cdot\nabla  y\big)+F_\C y\big]\d s
+\int_0^te^{(t-s)\Delta }\big(y\times \d  W\big)\,,
\end{aligned}
\end{equation}
and we make use of the following abbreviation
\[
\begin{aligned}
&\Theta _R(t):=\theta _R\big(\max(|w(t)|_{L^p},|y(t)|_{W^{1,p}})\big)\,.
\end{aligned}
\]
In the sequel, we shall consider $R>0$ as fixed,
and show that provided $T$ is sufficiently small,
\emph{depending only on $R$}, then:
\begin{enumerate}[label=(P\arabic*)]
 \item\label{point_1} for $q>2/\alpha $, $\Gamma_{X_0,R}$ maps $\mathfrak X_{q,T}$ into itself;
 \item\label{point_2} $\Gamma_{X_0,R}$ is a contraction in $\mathfrak X_{q,T}$.
\end{enumerate}
Then, Picard Theorem yields existence and uniqueness of a fixed point $(v_R,u_R),$ solution to \eqref{mild_v}-\eqref{mild_u} up to the stopping time $\tau _R:=\inf\{t\in[0,T],\enskip |w(t)|_{L^p}=R\enskip \text{or}\enskip |y(t)|_{W^{1,p}}=R\}.$

~\paragraph{\itshape Step 1: proof of \ref{point_1}}
For the velocity component, we have for every $t\leq T:$
\[
\begin{aligned}
|v(t)|_{L^p} 
&\leq C(p,T)|v_0|_{L^p}
+\Big|\int_0^t\Theta _R(s)e^{(s-t)A}\mathscr P[w(s)\cdot \nabla w(s)]\d s\Big|_{L^p}
\\
&\quad \quad \quad \quad \quad 
+\Big|\int_0^t\Theta _R(s)e^{(s-t)A}\mathscr P\div(\nabla u\odot\nabla u(s))\d s\Big|_{L^p}
\\
&
=:C(p,T)|v_0|_{L^p}+|\T_1|_{L^p}+|\T_2|_{L^p}\,.
\end{aligned}
\]
The first term is estimated as follows:
since $w$ is divergence-free, we have for each $i=1,2,$ the relation
$[w\cdot \nabla ]w^i \equiv \sum_{j=1}^{2}w^j\partial _jw^i = \sum_{j=1}^{2}\partial _j(w^j w^i).$
Using in addition \eqref{hypercontractivity}, we have for the first term
\[
\begin{aligned}
|\T_1|_{L^p}
&\leq \int_0^t\Theta _R(s)\frac{|\mathscr P\sum\nolimits_{j=1}^2\partial _j(w^j(s)w(s))|_{W^{-1,p/2}}}{(t-s)^{1/2+1/p}}\d s
\\
&\leq C(\mathscr P,p)T^{1/2-1/p} \sup_{s\in[0,t]}\big(\Theta _R(s)|w(s)|_{L^p}^2\big)
\\
&\leq C(\mathscr P,p)T^{1/2-1/p} R^2,
\end{aligned}
\]
by continuity of $\mathscr P:W^{-1,p}\to \W^{-1,p}$ since $p>1.$
Whence, for some universal constant $C>0,$ it holds:
\begin{equation}
\label{contrib:1}
 \interleave\T_1\interleave_{q,T,\L^p}\leq CT^{1/2-1/p} R^2\,.
\end{equation} 

Similarly, we have for the second term
\[
\begin{aligned}
|\T_2|_{L^p}
&
\leq \int_0^t\Theta _R(s)\frac{|\mathscr P\div(\nabla y(s)\odot\nabla y(s))|_{W^{-1,p/2}}}{(t-s)^{1/2+1/p}}\d s
\\
&\leq C(\mathscr P,p)T^{1/2-1/p}\sup_{s\in[0,t]}\big(\Theta _R(s)|\nabla y(s)\odot\nabla y(s)|_{L^{p/2}}\big)
\\
&\leq C(\mathscr P,p)T^{1/2-1/p} \sup_{s\in[0,t]}\big(\Theta _R(s)|\nabla y(s)|_{L^p}^2\big)
\\
&\leq C(\mathscr P,p)T^{1/2-1/p} R^2\,,
\end{aligned}
\]
and therefore:
\begin{equation}
\label{contrib:2}
\interleave\T_2\interleave_{q,T,\L^p}
\leq CT^{1/2-1/p}R^2\,.
\end{equation}

Next, for $t\in[0,T],$ almost surely, we write
\[
\begin{aligned}
|u(t)|_{W^{1,p}}
&\leq C_p|u_0|_{W^{1,p}}
+\Big|\int_0^t\Theta _R(s)e^{(t-s)\Delta }\big(y(s)|\nabla y(s)|^2\big)\d s\Big|_{W^{1,p}}
\\
&\quad \quad \quad \quad \quad 
+\Big|\int_0^t\Theta _R(s)e^{(t-s)\Delta }\big(w\cdot\nabla  y(s)\big)\d s\Big|_{W^{1,p}}
\\
&\quad \quad 
+\Big|\int_0^te^{(t-s)\Delta }\big(F_\C y(s)\big)\d s\Big|_{W^{1,p}}
+\Big|\int_0^te^{(t-s)\Delta }y(s)\times\d  W\Big|_{W^{1,p}}
\\
&=: C_p|u_0|_{W^{1,p}}+\sum\nolimits_{i=3}^6|\T_i|_{W^{1,p}}.
\end{aligned}
\]

Making use of \eqref{hypercontractivity} and the Sobolev inequality $|f|_{L^\infty}\leq C_p|f|_{W^{1,p}}$ for $p>2,$
the first term above is estimated as follows:
\[\begin{aligned}
|\T_3|_{W^{1,p}}
&\leq C\int_0^t\Theta _R(s)\frac{|y(s)|\nabla y(s)|^2|_{L^{p/2}}}{(t-s)^{1/2+1/p}}\d s
\\
&\leq CT^{1/2-1/p}\sup_{s\leq t}(\Theta _R(s)|y|_{L^\infty}|\nabla y|_{L^p}^2)
\\
&\leq CT^{1/2-1/p}\sup_{0\leq s\leq T}\big(\Theta _R(s)|y(s)|_{W^{1,p}}^3\big).
\end{aligned}
\]
Consequently it holds true that
\begin{equation}
\label{contrib:3}
\interleave\T_3\interleave_{q,T,\WW^{1,p}}\leq C T^{1/2-1/p}R^3\,.
\end{equation} 

Similarly, we have 
\[
\begin{aligned}
|\T_5|_{W^{1,p}}
&\leq \int_0^t\frac{C}{(t-s)^{1/2+1/p}}|y(s)|_{L^\infty}\d s\Big|\sum\nolimits_{l\geq 1}(\C f_l)^ 2\Big|_{L^{p/2}}
\\
&\leq C' T^{1/2-1/p}|\C|_{\gamma (L^p)}^2\sup_{s\in[0,T]}|y(s)|_{W^{1,p}}\,,
\end{aligned}
\]
and hence,
\begin{equation}
\label{contrib:5}
\interleave\T_5\interleave_{q,T,\WW^{1,p}}\leq C T^{1/2-1/p}|\C|^2_{\gamma (L^p)}R.
\end{equation} 

Concerning the transport term, we have
\[
\begin{aligned}
|\T_4|_{W^{1,p}}
&\leq
C\int_0^t\frac{\Theta  _R(s)\left|(w(s)\cdot \nabla)  y(s)\right|^q_{L^{p/2}}}{(t-s)^{1/2+1/p}}\d s
\\
&\leq C \int_0^T\frac{\Theta  _R(s)\left|w(s)\right|_{L^p}^q\left|  y(s)\right|_{W^{1,p}}^q}{(t-s)^{1/2+1/p}}\d s\,,
\end{aligned}
\]
and thus:
\begin{equation}
 \label{contrib:4}
\interleave\T_4\interleave_{q,T,\WW^{1,p}}\leq C T^{1/2-1/p}R^2.
\end{equation}
The treatment of the stochastic convolution $\T_6 \equiv\int_0^\cdot e^{(\cdot -s)\Delta }y(s)\times\d  W$ works as follows:
letting $\epsilon >0$ such that
\begin{equation}
\label{parameters}
\min(\alpha,1) >\epsilon >\frac 2p\,,
\end{equation} 
then the stochastic estimates (Proposition \ref{pro:parabolic_sto}) with $\lambda =0,\delta =1-\epsilon $ imply
\begin{multline*}
\E \left[\|\T_6 \|_{C(0,T;\WW^{1,p})}^q\right]\leq C T\E\left[\|y\times\C \|_{C(0,T;\gamma (\WW^{\epsilon  ,p}))}^q\right]
\\
\leq C T|\C|_{\gamma (W^{\epsilon ,p})}^q\E\left[\|y\|_{C(0,T;\WW^{\epsilon ,p})}^q\right],
\end{multline*}
since the space $W^{\epsilon ,p},$ where $\epsilon ,p$ are subject to conditions \eqref{parameters}, is an algebra.
Hence, we end up with the estimate:
\begin{equation}
\label{contrib:6}
\interleave\T_6\interleave_{q,T,\WW^{1,p}}
\leq C^{1/q} T^{1/q}|\C|_{\gamma (W^{\alpha ,p})}\interleave y\interleave_{q,T,\WW^{1,p}}\,.
\end{equation} 

Summing the estimates \eqref{contrib:1}, \eqref{contrib:2}, \eqref{contrib:3}, \eqref{contrib:4}, \eqref{contrib:5} and \eqref{contrib:6}, we obtain \ref{point_1}.

~\paragraph{\itshape Step 2: proof of \ref{point_2}}
For $j=1,2,$ take $(w_j,y_j)\in \X_{q,T},$ let $(v_j,u_j):=\Gamma_{X_0,R}(w_j,y_j),$
and denote by 
\[\Theta _R^j(t):=\theta _R\big(\max(|w_j(t)|_{L^p},|y_j(t)|_{W^{1,p}})\big).
\]
We also define
\begin{equation}
\label{nota:tau1_tau2}
\tau _j:=\inf\{s\in[0,T]:\max(|w_j(s)|_{L^p},|y_j(s)|_{W^{1,p}})=R\}\,.
\end{equation}

For the velocity component, we write
\[
\begin{aligned}
v_1(t)-v_2(t)
&=
\int_0^t e^{(s-t)A}\mathscr P\Big(\Theta ^1_Rw_1\cdot \nabla w_1-\Theta ^2_Rw_2\cdot \nabla w_2\Big)\d s
\\
&\quad \quad 
+\int_0^te^{(s-t)A}\mathscr P\div \Big(\Theta ^1_R\nabla y_1\odot\nabla y_1 -\Theta ^2_R\nabla y_2\odot\nabla y_2\Big)\d s
\\
&=:\varDelta_1+\varDelta_2.
\end{aligned}
\]
To estimate the first term, we can assume without loss of generality that $\tau_1\leq \tau _2.$ This gives
\[
\begin{aligned}
|\varDelta_1|_{L^p}
&\leq C\int_0^{\tau _2}(\Theta ^1_R-\Theta ^2_R)\frac{\left|\mathscr P\sum\nolimits_{j=1}^2\partial _j(w_2^j w_2)\right|_{W^{-1,p/2}}}{(t-s)^{\frac12+\frac1p}}\d s
\\
&\quad \quad \quad \quad 
+ C\int^{\tau _2}_0\Theta ^1_R\frac{\left|\mathscr P\sum\nolimits_{j=1}^2\partial _j(w_1^jw_1-w_2^jw_2)\right|_{W^{-1,p/2}}}{(t-s)^{\frac12+\frac1p}}\d s
\\
&=\mathrm{I}+\mathrm{II}\,.
\end{aligned}
\]
Thanks to the continuity of $\mathscr P,$ and the fact that $x,y\mapsto \max(x,y)$ is Lipshitz, we have
\[
\begin{aligned}
|\mathrm{I}|_{L^q(\Omega )}
&\leq C(\mathscr P,p)T^{1/2-1/p}\sup_{s\in[0,T]}\big(\Theta_R(s)|w_2(s)|^2_{L^p}\big)
\\
&\quad \quad \quad 
\times\frac{|\theta '|_{L^\infty}}{R}(\interleave w_1-w_2\interleave_{q,T,\L^p}
+\interleave y_1-y_2\interleave_{q,T,\WW^{1,p}})
\\
&\leq CRT^{1/2-1/p}|\theta' |_{L^\infty}\interleave (w_1-w_2,y_1-y_2)\interleave_{\X_{q,T}}\,.
\end{aligned}
\]
On the other hand, we have
\begin{multline*}
\mathrm{II}
\leq C(\mathscr P,p)\int_0^{\tau _2}\Theta^1_R\frac{\left|\sum\nolimits_{j=1}^2\partial _j\big[(w_1^j(s)-w_2^j(s))w_1(s)\big]\right|_{W^{-1,p/2}}}{(t-s)^{\frac12 +\frac1p}}\d s
\\
+C(\mathscr P,p)\int_0^{\tau _2}\Theta^1_R\frac{\left|\sum\nolimits_{j=1}^2\partial _j[w_2^j(w_1(s)-w_2(s))]\right|_{W^{-1,p/2}}}{(t-s)^{\frac12+\frac1p}}\d s\,,
\end{multline*}
which yields the estimate
$|\mathrm{II}|_{L^q(\Omega )}
\leq CT^{1/2-1/p}R\interleave w_1-w_2\interleave_{q,T,\L^p}.$
Hence, we obtain that
\begin{equation}
 \label{contrib:D1}
 \interleave \varDelta_1 \interleave_{q,T,\L^p}\leq CT^{1/2-1/p}R\interleave(w_1-w_2,y_1-y_2)\interleave_{\X_{q,T}}
\end{equation} 

The second term is similar:
assuming without loss of generality that $\tau_1\leq \tau_2,$ we have
\[
\begin{aligned}
|\varDelta_2|_{L^p}
&\leq \int_0^{\tau_2}|\Theta^1_R-\Theta^2_R| |e^{(s-t)A}\mathscr P\div(\nabla y_2\odot\nabla y_2)|_{L^p}\d s
\\
&\quad \quad \quad 
+\int_0^{\tau_1}\Theta^1_R|e^{(s-t)A}\mathscr P\div(\nabla (y_1-y_2)\odot\nabla y_1)|_{L^p}\d s
\\
&\quad \quad \quad \quad 
+\int_0^{\tau_1}\Theta^1_R|e^{(s-t)A}\mathscr P\div(\nabla y_2\odot\nabla (y_1-y_2))|_{L^p}\d s
\\
&=: \mathrm{I}+\mathrm{II}+\mathrm{III}
\,.
\end{aligned}
\]
By a similar argument as before, we have for the first term:
\[
|\mathrm{I}|_{L^q(\Omega)}\leq C(\mathscr P,p)|\theta'|_{L^\infty}T^{1/2-1/p}R
\interleave (w_1-w_2,y_1-y_2)\interleave_{\X_{q,T}}\,.
\]
Next, proceeding as for $\varDelta_1,$ it is easy to see that:
\[\E \left[\sup_{t\in[0,T]}(\mathrm{II}(t)+\mathrm{III}(t))^q\right]^{1/q}\leq CRT^{1/2-1/p}\interleave y_1-y_2\interleave_{q,T,\WW^{1,p}}.
\]
Hence, we end up with the estimate:
\begin{equation}
\label{contrib:D2}
\interleave \varDelta_2\interleave_{q,T,\L^{p}}
\leq 
CRT^{1/2-1/p}\interleave (w_1-w_2,y_1-y_2)\interleave_{\X_{q,T}}
\end{equation}

We now proceed to the evaluation of the second component. We write
\[
\begin{aligned}
(u_1-
&u_2)(t)
=\int_0^te^{(t-s)\Delta}\big(
\Theta _R^{1}y_1|\nabla y_1|^2- \Theta _R^2y_2|\nabla y_2|^2
\big)\d s
\\
&\quad \quad \quad \quad 
-\int_0^te^{(t-s)\Delta}\big(\Theta ^{1}_R w_1\cdot\nabla  y_1-\Theta  _R^{2}w_2\cdot \nabla y_2\big)\d s
\\
&-\int_0^te^{(t-s)\Delta}\big(F_\C(y_1-y_2)\big)\d s
+\int_0^te^{(t-s)\Delta}(y_1-y_2)\times\d  W
\\
&\quad \quad \quad 
=\varDelta_3+\varDelta_4+\varDelta_5+\varDelta_6
\,.
\end{aligned}
\]
Now, linearity and Step 1 provide the bound
\begin{equation}
\label{contrib:D5_6}
\interleave\varDelta_5+\varDelta_6\interleave_{q,T,\WW^{1,p}}
\leq \interleave y_1-y_2\interleave_{q,T,\WW^{1,p}}
C\Big(T^{1/2-1/p}|\C|^2_{\gamma (L^p)} +T|\C|_{\gamma (W^{\alpha ,p})}\Big)\,.
\end{equation}

Next, recalling \eqref{nota:tau1_tau2} and assuming without loss of generality that $\tau _1\leq \tau _2,$ we have for the third term:
\[
\begin{aligned}
|\varDelta_3|_{W^{1,p}}
&\leq \int_0^{\tau _2}
\Big|e^{(t-s)\Delta}\Big(
\Theta _R^{1}y_1|\nabla y_1|^2- \Theta _R^2y_2|\nabla y_2|^2
\Big)\Big|_{W^{1,p}}\d s
\\
&\leq  \int_0^{\tau _2}(\Theta_R^1-\Theta_R^2)\Big|e^{(t-s)\Delta }(y_2|\nabla y_2|^2)\Big|_{W^{1,p}}\d s
\\
&\quad \quad 
+\int_0^{\tau _2}\Theta_R^1\Big|e^{(t-s)\Delta }(y_1|\nabla y_1|^2-y_2|\nabla y_2|^2)\Big|_{W^{1,p}}\d s
\\
&=\mathrm{I}+\mathrm{II}\,.
\end{aligned}
\]
Proceeding as for the velocity component, we can estimate the first term as
\[
|\mathrm{I}|_{L^q(\Omega )}\leq C'R^2T^{1/2-1/p}|\theta' |_{L^\infty}\interleave( w_1-w_2, y_1-y_2)\interleave_{\X_{q,T}}.
\]
For the second term, using again \eqref{hypercontractivity}, we have
\[
\begin{aligned}
\mathrm{II}
&\leq C\int_0^{\tau _2}\Theta _R^1\frac{\Big|(y_1-y_2)|\nabla y_1|^2\Big|_{L^{p/2}}}{(t-s)^{1/2+1/p}}\d s 
\\
&\quad \quad \quad \quad \quad 
+C\int_0^{\tau _2}\Theta ^1_R\frac{\Big|y_2(|\nabla y_1|^2-|\nabla y_2|^2)\Big|_{L^{p/2}}}{(t-s)^{1/2+1/p}}\d s
\\
&=\mathrm{II}'+\mathrm{III}\,.
\end{aligned}
\]
Since $\WW^{1,p}\hookrightarrow \LL^\infty,$ it holds immediately
\[
|\mathrm{II}'|_{L^q(\Omega )}\leq CR^2T^{1/2-1/p} \interleave y_1-y_2\interleave_{q,T,\WW^{1,p}}.
\]
Similarly, we have
\[
\begin{aligned}
|\mathrm{III}|_{L^q(\Omega )}
&\leq \E\Big[\Big(\int_0^{\tau _2}\Theta ^1_R \frac{\big|y_2\nabla (y_1-y_2)\cdot \nabla (y_1+y_2)\big|_{L^{p/2}}}{(t-s)^{1/2+1/p}}\d s\Big)^q\Big]^{1/q}
\\
&\leq 
C R^2 T^{1/2-1/p} \interleave y_1-y_2\interleave_{q,T,\WW^{1,p}}\,.
\end{aligned}
\]
Finally, we see that there is a universal constant $C>0$ such that:
\begin{equation}
\label{contrib:D3}
\interleave\varDelta_3\interleave_{q,T,\WW^{1,p}} 
\leq C R^2T^{1/2-1/p}\interleave (w_1-w_2,y_1-y_2)\interleave _{\X_{q,T}}
\,.
\end{equation}

For the fourth term, it is again sufficient to assume that $\tau _1\leq \tau _2.$
But in this case it holds
\[
\begin{aligned}
|\varDelta_4|_{W^{1,p}}
&\leq\int_0^{\tau _2}\Big|e^{(t-s)\Delta }\big(\Theta _R^1w_1\cdot \nabla y_1 - \Theta _R^2w_2\cdot \nabla y_2\big)\Big|_{W^{1,p}}\d s
\\
&\leq \int_0^{\tau _2}|\Theta ^1_R-\Theta ^2_R|\Big|e^{(t-s)\Delta }w_2\cdot \nabla y_2\Big|_{W^{1,p}}\d s
\\
&\quad \quad \quad 
+\int_0^{\tau _2}\Theta ^1_R\Big|e^{(t-s)\Delta }(w_1\cdot \nabla y_1 -w_2\cdot \nabla y_2)\Big|_{W^{1,p}}\d s\,.
\end{aligned}
\]
Whence, we obtain as before:
\[
\begin{aligned}
\interleave \varDelta_4\interleave_{q,T,W^{1,p}}
&\leq 
C R^2|\theta '|_{L^\infty}T^{1/2-1/p}\interleave (w_1-w_2,y_1-y_2)\interleave_{\X_{q,T}}
\\
&\quad \quad \quad 
+\E\Big[\Big(\int_0^{T}\Theta _R^1 \big|e^{(t-s)\Delta }\big[(w_1-w_2)\cdot \nabla y_2\big]\big|_{W^{1,p}}\d s\Big)^q\Big]^{1/q}
\\
&\quad \quad \quad \quad \quad
+\E\Big[\Big(\int_0^{T}\Theta _R^1 \big|e^{(t-s)\Delta }\big[w_1\cdot \nabla(y_1-y_2)\big]\big|_{W^{1,p}}\d s\Big)^q\Big]^{1/q}
\\
&\leq 
C'R^2 T^{1/2-1/p}\interleave (w_1-w_2,y_1-y_2)\interleave_{\X_{q,T}}\,.
\end{aligned}
\]
Hence, it follows that for some universal constant $C>0:$
\begin{equation}
\label{contrib:D4}
\interleave\varDelta_4\interleave_{q,T,\WW^{1,p}}
\leq CR^2T^{1/2-1/p}\interleave (w_1-w_2,y_1-y_2)\interleave_{\X_{q,T}}\,.
\end{equation} 

Summing the estimates \eqref{contrib:D1}, \eqref{contrib:D2}, \eqref{contrib:D3}, \eqref{contrib:D4} and \eqref{contrib:D5_6}, we see that there exists a constant $C(p,|\C|_{\gamma (W^{\alpha ,p})})>0$ such that
\begin{multline*}
\interleave\Gamma (w_1,y_1)-\Gamma (w_2,y_2)\interleave_{\X_{q,T}}
\\
\leq C(p,|\C|_{\gamma (W^{\alpha ,p})})(1+R^2)\big(T^{1/2-1/p} + T \big)\interleave(w_1,y_1)-(w_2,y_2)\interleave_{\X_{q,T}}.
\end{multline*}
If we choose $T$ such that
\[
0<T\leq T^*(R):=\min\left(1,4C(p,|\C|_{\gamma (W^{\alpha ,p})})(1+R^2)\right)^{-\frac{2p}{p-2}},
\]
then $\Gamma _{X_0,R}:\X_{q,T}\to\X_{q,T}$ is a contraction. This proves \ref{point_2}.

~\paragraph{\itshape Step 3: definition of the maximal solution}
Using a localization procedure (see \cite[Theorem 4.1]{de2002effect}) we can build a maximal solution as follows:
we define the stopping times
\begin{equation}\label{nota:tau_m}
\tau_m=\inf\left\{t\in[0,T],\enskip \max\big(|v_m(t)|_{L^p},|u_m(t)|_{W^{1,p}}\big)\geq m\right\}\, ,
\end{equation}
and show that
the sequence $\{\tau _m\}_{m\in\N}$ is non-decreasing and that 
\[
(v_{m+1}(t),u_{m+1}(t))=(v_m(t),u_{m}(t))\quad \text{for} \enskip t\in[0,\tau_m],\enskip  \text{a.s.}
\]
Since the proof of these properties does not involve any new idea compared to steps 1 and 2,
we leave the details to the reader.
The maximal solution $(v;u,\tau \equiv\sup_{m\geq 0}\tau _m)$ is then defined by ``gluing together'' $(v_m,u_m)$ for each $m\geq 1.$

~\paragraph{\itshape Step 4: General case}
It is sufficient to let $s=3$ (the general case follows by interpolation).

Let $(v,u)\in L^q\big(\Omega ;C([0,\tau ^{p,1});\L^p\times\WW^{1,p})\big)$ be the solution provided by the above fixed point, and assume that $X_0=(v_0,u_0)$ belongs to $\W^{2,p}\times \WW^{3,p}.$
For $(\tilde w,\tilde y)\in \X_{q,T},$ define $(\tilde v, \tilde u):=\Gamma _{X_0,R}^{\Delta }(\tilde w, \tilde y)$ via the relation 
\[
\begin{aligned}
\tilde v(t)&
=e^{-tA}\Delta v_0+\int_0^t\theta _R(\max(|\tilde w(s)|_{L^p},|\tilde y(s)|_{W^{1,p}}))
\\
&\quad \quad \quad \quad \quad \quad 
\times e^{A(s-t)}\mathscr P[- \tilde w\cdot \nabla v - 2\nabla v\cdot \nabla ^2v - v\cdot \nabla \tilde w
\\
&\quad \quad \quad \quad \quad \quad \quad 
-\div(\nabla \tilde y\odot\nabla u +2\nabla ^2u\odot \nabla ^2u +\nabla u\odot\nabla \tilde y)]\d s
\\
\tilde u(t)&
=e^{t\Delta }\Delta u_0+\int_0^t\theta _R(\max(|\tilde w(s)|_{L^p},|\tilde y(s)|_{W^{1,p}}))
\\
&\quad \quad \quad \quad \quad \quad 
\times e^{\Delta (t-s)}\big(\tilde y|\nabla u|^2 + 4\nabla u\cdot \nabla ^2u\nabla u + 2u (\nabla \tilde y + \nabla ^2u\cdot \nabla ^2u)
\\
&\quad \quad \quad \quad \quad \quad 
-\tilde w\cdot \nabla u - 2\nabla v\cdot \nabla ^2u - v \cdot \nabla \tilde y
-\Delta F_\C u - 2\nabla F_\C \cdot \nabla u -F_\C \tilde y\big)\d s
\\
&\quad \quad \quad \quad \quad 
+\int_0^t e^{\Delta (t-s)}\left[\tilde y\times \d  W+2\nabla u\times \d \nabla W+u\times \d \Delta W\right],
\end{aligned}
\]
Using similar arguments as that of Step 1 and Step 2, it can be shown that the map $\Gamma _{X_0,R}^{\Delta }$ admits a unique fixed point $(\tilde v,\tilde u)$, for every $R>0.$ Moreover, observing that $\Delta $ commutes with $e^{\cdot A}$ and with $\mathscr P$ (see \eqref{fourier}), it is immediately checked that on $\{\max(|\tilde v(t)|_{L^p},|\tilde u(t)|_{W^{1,p}})\leq R\},$
we have
$\tilde v(t)=\Delta v(t)$ while $\tilde u(t)= \Delta u(t),$ which shows the claimed regularity on $(v,u).$

The details of the proof being similar to step 1 and step 2 (the difficulties are merely notational), we leave them to the reader.

~\paragraph{\itshape Step 5: spherical constraint}
For $p>2$ and $s\geq 2,$ if $(v,u)$ is defined as above, we can see it as a continuous path with values in the Hilbert space $H:=\H^{-1}\times \LL^2$.
Because $(v,u)$ takes values in the domain of $L:=(A,-\Delta ),$ it is in fact a \emph{strong solution} in the sense of \cite[p.\ 160]{DPZ} (up to the stopping time $\tau ^{p,s}$).
Hence, it is possible to apply It\^o Formula to the functional $F(v,u):=\frac12|1-|u|^2|_{L^2}^2$ (which is clearly of class $C^2$ from $D(L)\equiv \H^{1}\times\HH^2$ to $\R$).
Using the identity $\Delta (|u|^2-1)=2u\cdot \Delta u+ 2|\nabla u|^2,$ this gives
\begin{equation}
 \begin{aligned}
\label{pre_gronwall}
\frac12
&\int_{\O}(|u(t)|^2-1)^2\d x
\\
&=\iint_{[0,t]\times \O} \big[2|u|^2(|u|^2-1)|\nabla u|^2 + 2u\cdot \Delta u(|u|^2-1) - 2u\cdot \partial _iuv^i(|u|^2-1)\big]\d x\d s
\\
&=\iint_{[0,t]\times \O} \big[2(|u|^2-1)^2|\nabla u|^2 -2|\nabla (|u|^2-1)|^2 - 2u\cdot \partial _iuv^i(|u|^2-1)\big]\d x\d s
\\
&\leq 
4\int_0^t |\nabla u(s)|^2_{L^\infty}F(v(s),u(s))\d s - 2\iint_{[0,t]\times\O}u\cdot \partial _iuv^i(|u|^2-1)\d x\d s
\\
&
=\mathrm{I}+\mathrm{II}\,.
\end{aligned}
\end{equation}

As seen by integration by parts, using that $\div v=0$ we have
\[
\mathrm{II}:=-\iint_{[0,t]\times \O}2u\cdot \partial _iuv^i(|u|^2-1)\d x\d s
=-\iint_{[0,t]\times \O}\partial _i(|u|^2-1)v^i(|u|^2-1)\d x\d s
=-\mathrm{II}\,,
\]
and therefore $\mathrm{II}=0.$
Next, by Gronwall Lemma applied to \eqref{pre_gronwall}, we see that provided $0\leq t<\tau ^{s,p}$ is such that
\begin{equation}
\label{gradient_u}
\int_0^t |\nabla u(s)|_{L^\infty}\d s<\infty,
\end{equation}
(which is always true by the Sobolev embedding $W^{p,s}\hookrightarrow W^{1,\infty}$ for $p>2$ and $s\geq 2$)
then one has $F(v(t),u(t))=0.$
This proves the norm constraint \eqref{sphere_constraint}.
\hfill\qed

\section{A priori estimates}
\label{sec:apriori}

Our purpose in this section is to obtain suitable a priori estimates associated to \eqref{SEL}, assuming that the solution at hand is regular enough.
We divide them into two categories, namely the ``global estimates'', i.e.\ estimates on the whole time interval, and the ``local estimates'', i.e.\ estimates on small balls $B(x,\rho )\subset\O$. Estimates that are local in the time-like variable will be obtained in Section \ref{sec:bootstrap} below.

We start by stating an important byproduct of \eqref{ineq:interp}.
Note that in \cite{St1}, the left hand side below is a gradient, which is unnecessary (as can be immediately seen in the details of the proof).

\begin{proposition}[Struwe \cite{St1}, Lemma 3.1]
\label{pro:interp2}
There exists a constant $\mu _1>0$, such that for all $v\in C(0,T;H^1(\O))$, for all $\rho>0$:
\begin{equation}
\label{ineq:interp_loc}
\iint\limits_{[0,T]\times\O}|v|^4\d y\d t \leq \mu _1\left(\sup_{\substack{t\in [0,T]\\ x\in\O}}\int\limits_{y\in B(x,\rho )}|v(t,y)|^2\d y\right)
\iint\limits_{[0,T]\times\O}\Bigg(|\nabla v|^2+\frac{|v|^2}{\rho^2}\Bigg)\d y\d t\,.
\end{equation}
\end{proposition}

\subsection{Global estimates}
\label{subsec:global}

Let us first introduce some notation.

\begin{notation}
Since it plays a specific role along the proofs, we will denote the ``tension'' by 
\begin{equation}\label{nota:tension}
\tt_u:=\Delta u+u|\nabla u|^2\,.
\end{equation}
The latter terminology is borrowed from \cite{eells1964harmonic}, in the context of harmonic map flows. Roughly speaking, the $L^2$-norm of $\tt_u$ measures ``how far'' $u$ is from being harmonic. Note indeed that $\tt_u=0$ if and only if $u:\O\to\mathbb{S}^2$ is weakly harmonic.

An important observation is that the tension equals the orthogonal projection of $\Delta u$ onto $u^\perp,$ that is:
\[
\tt_u=\Delta u - (\Delta u\cdot u)u\,,
\]
whenever the latter makes sense as a distribution.
\end{notation}

\begin{proposition}
 \label{pro:global_estimates}
 Let $(W,\C)$ be as in Assumption \ref{ass:W}.
Assume that $(v,u)$ is a strong solution to \eqref{SEL} on $[0,T],$ with respect to $(\mathfrak P,W),$ in the sense of Definition \ref{def:strong_sol}.
Denoting by 
\[
E(t):=\frac{1}{2}\left(|v(t)|_{L^2}^2+|\nabla u(t) |^2_{L^2}\right),
\]
then it holds true that for any $0\leq t\leq T,$ a.s.: 
\begin{equation}
\label{estim:energy}
E(t)-E(0) +  \int_0^t\left(|\nabla v(s)|_{L^2}^2+|\tt_u(s)|^2_{L^2}\right)\d s = t|\nabla \C |_{\mathscr L_2(L^2,L^2)}^2+X(t),
\end{equation}
where $(X(t))$ is the martingale defined by the It\^o integral
$X(t):=\int_0^t\langle u\times \d \nabla  W,\nabla u \rangle,$ $t\in [0,T].$

Moreover, for any $m\geq 1$, the following estimate holds:
\begin{multline}
\label{estim:moments}
\E\left[ \sup_{t\in[0,T]} E(t)^{m}\right] + \E\left[\left(\int_0^T|\nabla v(t)|_{L^2}^2 +|\tt_u(t)|_{L^2}^2\d t \right)^m \right]
\\
\leq C\left(T,m,|\C|_{\hs{1}},E(0)\right)
\end{multline}
where the above constant depends on the indicated quantities, but not on the individual element $(v,u).$
\end{proposition}

\begin{proof}
\textit{Proof of \eqref{estim:energy}.}
We can apply It\^o Formula in the form given e.g.\ in \cite[Chapter 4]{prevot2007concise} to the functional
$(v,u)\in \L^2\times \HH^{1}\mapsto \frac12(|v|_{L^2}^2 + |\nabla u|_{L^2}^2).$
It yields:
\begin{equation}
\label{first_ito}
\begin{aligned}
E(t)-E(0)
&= \int_0^t\left\langle\d v,v\right\rangle+ \int_0^t\langle\nabla \d u ,\circ\nabla u \rangle
\\
&=\int_0^t\Big\langle-Av-(v\cdot \nabla )v-\div(\nabla u \odot \nabla u ) ,v\Big\rangle\d t
\\
&\quad 
+\int_0^t\Big\langle\nabla (-v\cdot \nabla u +\Delta u +|\nabla u |^2u) ,\nabla u\Big\rangle\d t
+ \int_0^t\big\langle\nabla (u \times\circ \d  W),\nabla u \big\rangle,
\end{aligned}
\end{equation} 
where the above Stratonovitch integral makes sense as 
\begin{equation}
\label{above_strato}
\begin{aligned}
\int_0^t
\big\langle\nabla (u \times\circ \d  W),\nabla u \big\rangle
&=\int_0^t\langle u \times\d \nabla  W,\nabla u \rangle
\\
+\frac12\sum_{l\in\N}\sum_{j=1}^3 
&\int _0^t\Big [\langle (u\times \CC_l^j)\times\nabla \CC_l^j,\nabla u\rangle 
+\langle u\times\nabla \CC_l^j, \nabla (u\times \CC_l^j)\rangle \Big ]\d s
\\
&=:X(t) + \int _0^t(A^1(s)+A^2(s))\d s.
\end{aligned}
\end{equation}
It is easily seen that the above trace terms are finite a.s.\ , thanks to our assumptions on $u$ and $\C$.
The computation of the trace terms $A^i$ is detailed in Appendix \ref{app:trace} and leads to
\begin{equation}
\label{trace:1}
A^1(t)=0\quad \text{and}\quad A^2(t)=2|\nabla \C|_{\mathscr L_2(L^2,L^2)}^2,\quad t\in[0,T].
\end{equation}
Therefore, we can rewrite \eqref{first_ito} as:
\[
\begin{aligned}
E(t)-E_0
&- X(t) -t|\nabla \C|_{\mathscr L_2(L^2,L^2)}^2
+\int_{0}^t|\nabla v(s)|^2_{L^2}\d s
\\
&=
\int_0^t\Big[-\langle v\cdot \nabla v, v\rangle
-\langle\div(\nabla u \odot \nabla u ) , v\rangle
-\langle\nabla (v\cdot \nabla u),\nabla u\rangle
\\
&\quad \quad \quad \quad \quad 
\quad \quad \quad \quad \quad 
-\langle\Delta u, \Delta u +|\nabla u |^2u \rangle\Big]\d s
\\
&=:\int _0^t\big [\T_1+\T_2+\T_3+\T_4\big] \d s,
\end{aligned}
\]
and we now treat each term separately.

For the first term,
integration by parts gives
\[
\T_1\equiv\int_{\O} v^i\partial _iv^jv^j\d x
=-\int_{\O} \partial _i(v^iv^j)v^j\d x=-\int_\O|v|^2\div v\d x -\int_\O v^i\partial _iv^jv^j\d x\,,
\]
where we use a summation convention over repeated indices.
Hence:
\[
\T_1=-\T_1=0\,.
\]

Concerning the second and third terms, we can proceed as in \cite{LL1}: we have
\[
\T_2\equiv-\int_{\O}v^i\partial _j(\partial _j u)\cdot \partial _iu\d x
=-\int_{\O}\Delta u\cdot (v\cdot \nabla u)\d x
=\int_{\O}\nabla u\cdot \nabla (v\cdot \nabla u)\d x\,,
\]
so that:
\[
\T_2+\T_3=0.
\]

By the fact that $\tt_u\perp u$ (see \eqref{nota:tension}), we have the pointwise identity:
\begin{equation}
\label{id_tension}
\big(\Delta u +u |\nabla u |^2\big)\cdot \Delta u
\equiv\tt_u\cdot \Delta u =|\tt_u|_{\R^3}^2.
\end{equation} 
Therefore,
\[
\T_4= -\int_{\O}|\tt_u|^2\,.
\]

Summing every term above eventually yields \eqref{estim:energy}.

~\paragraph{\itshape Proof of \eqref{estim:moments}}
Let $m>1.$ From the Burkholder-Davis-Gundy inequality, one can bound $X$ as follows:
\[
\E\left[\sup_{s\in[0,t]}|X(s)|^m\right] \leq C(m) \E \left[\left(\int_0^t|\C^* \div(u\times\nabla u)|^2_{L^2}\d s\right)^{m/2}\right].
\]
Hence, using the embedding $\hs{1}\hookrightarrow \mathscr L(L^2,H^1),$ we obtain
\begin{multline}
\label{estim:mart}
 \E\left[\sup_{s\in [0,t]}|X(s)|^m\right]
 \leq C(|\C|_{\hs{1}})\E\left[\int_0^t|u\times \nabla u|_{L^2}^2\d s\right]
\\
 = C(|\C|_{\hs{1}})\E\left[\int_0^t|\nabla u|_{L^2}^2 \d s\right].
\end{multline}
Hence, going back to \eqref{estim:energy}, we can use \eqref{estim:mart}, yielding
for any $t\in[0,T]:$
\[
\Phi (t) \leq C\left(1 +\int_0^t\Phi (s)\d s\right),
\]
where we denote by $\Phi (t):=\E[\sup_{s\in[0,t]}E(s)^m],$
and where the above constant depends only on $E_0,m,T$ and $|\C|_{\hs{1}}.$
The first part of the conclusion hence follows by Gronwall Lemma.

Using again \eqref{estim:energy}, together with the bound obtained on $\Phi ,$ 
we can now bound the term $\E[(\int_0^T[|\nabla v|_{L^2}^2+|\tt _u|^2]\d t)^m]$ in the same way.
\end{proof}
\subsection{Local estimates}
\label{subsec:local}
In this paragraph, we estimate the local energy $\frac12\int_{B(x,\rho )}|v(t,y)|^2+|\nabla u(t,y)|^2\d y,$ where $\rho >0$ and $x\in\O.$
This requires to test \eqref{SEL} against the localization function $\mathds 1_{B(x,\rho)}$ (or a suitable regularized version thereof), which is definitely not divergence-free in general.
To proceed further, we therefore need to introduce the pressure term associated to a solution.

Let $(v,u)$ be a strong solution in the sense of Definition \ref{def:strong_sol} on $[\tau _1,\tau _2).$
Using \eqref{variational_identity_1} against test functions of the form $-\nabla \theta $ where $\theta \in L^2(\O;\R),$
and denoting by $\tilde\pi (\omega ,t):=\pi-\int_{\O}\pi(\omega ,t,x)\d x,$ 
we see from the original equation \eqref{SEL} that $\tilde\pi $ must be given by the equation
\begin{equation}
\label{variational_pi}
\Delta \tilde\pi = -\partial _{ij}(v^iv^j)
-\partial _{ij}(\partial _iu\cdot \partial _ju),\quad \langle\tilde\pi ,1\rangle=0\,,
\end{equation}
for a.e.\ $(\omega ,t)\in\Omega \times[0,T]$ such that $\tau _1(\omega )\leq t<\tau _2(\omega ).$
From \eqref{ineq:interp} we see that 
\[
v\in L^0(\Omega ;L^4([\tau _1,\tau _2);\L^4)),
\]
while
\[
\nabla u\in L^0(\Omega ;L^4([\tau _1,\tau _2);\L^4)).
\]

Therefore, using the fact that the Laplace operator is an isomorphism, from $L^2(\O)$ to $H^{-2}(\O)$, when restricted to the orthogonal of constant functions, we infer from H\"older Inequality that
\[
\tilde\pi \in L^0(\Omega ;L^2([\tau _1,\tau _2);L^2)),
\]
together with the pathwise estimate:
\begin{equation}
\label{estim:pi_0}
\|\tilde\pi (\omega )\|_{L^2(\tau _1;\tau _2';L^2)}
\leq C\left(\|v(\omega )\|_{L^4(\tau _1;\tau _2';L^4)}^2 + \|\nabla u(\omega )\|_{L^4(\tau _1;\tau _2';L^4)}^2\right)
\end{equation} 
for $\P$-a.e.\ $\omega \in\Omega $ and for every stopping time $\tau _2'$ such that $\tau _1\leq \tau _2'<\tau_2.$

On the other hand, choosing any constant $c\in\R,$ one can let $\pi:=\tilde\pi+c,$ and with this definition it is immediately checked that the equation
\begin{equation}
\label{eq:v_pi}
\d v+(-\Delta v+\nabla \pi +v\cdot \nabla v+\div\nabla u\odot\nabla u)\d t=0
\end{equation}
holds for almost every $\omega\in\Omega ,$ in the sense of Bochner in $H^{-1}(\O;\R^2)$.

\begin{proposition}
\label{pro:local_estimates}
Let $(W,\C)$ be as in Assumption \ref{ass:W}.
Consider $\varphi \in C^\infty(\O),$ $0\leq \varphi \leq 1,$ supported in the ball $B(x,\rho)$ for some $x\in\O$ and some $\rho>0,$ and assume that there exists 
$K>0$ with 
\[
\sup_{y\in\O}|\nabla \varphi (y)|\leq \frac{C}{\rho}\,.
\]
Assume that $(v,u)$ is a strong solution to \eqref{SEL} on $[0,T],$ with respect to $(\mathfrak P,W),$
and for any $t\in[0,T],$ define $E^\varphi (t):=\frac12\int_{\O}\varphi ^2(|v(t,y)|^2+|\nabla u(t,y)|^2)\d y.$

Then, there is a universal constant $C_0>0$ such that for every $t\in[0,T],$ it holds
\begin{multline}
\label{ineq:local}
 E^\varphi (t)-E^{\varphi }(0) 
 +\frac12\iint_{[0,t]\times \O}\varphi ^2\left(|\nabla v|^2+|\tt_u|^2\right)\d x\d s
 \\
 \leq C_0\left(1+\frac{1}{\rho^2}\right)\iint_{[0,t]\times \O}\Big(|v|^3+(|v|+1)|\nabla u|^2 +|v||\tilde \pi|\Big)\d x\d s
+ X^\varphi (t)  +C_\C ^\varphi t\,,
\end{multline}
where
$X^\varphi(t) :=\int_0^t\langle \varphi\nabla u,\varphi u\times\d  W\rangle,$
$C^\varphi _\C :=\sum_{l\in\N}\int_\O\varphi ^2 (x)|\nabla \C f_l(x)|^2\d x,$
and $\tilde \pi \equiv \pi-\langle \pi,1\rangle$ is the pressure term corrected by its mean value -- see the above discussion.
\end{proposition}
\begin{proof}
 Apply It\^o Formula to $E^\varphi .$
It holds
\[
\begin{aligned}
E^\varphi (t)-E^\varphi (s)
&=\int_0^t\langle\varphi ^2v ,\d v\rangle
+\int_0^t\langle\varphi ^2\nabla  u,\circ\nabla \d u \rangle
\\
&=\iint_{[0,t]\times \O}\varphi ^2v\cdot \big(\Delta v-v\cdot \nabla v-\div(\nabla u\odot\nabla u)-\nabla \pi\big)\d x\d s
\\
&\quad +\iint_{[0,t]\times\O}\varphi ^2\nabla u\cdot \nabla (\tt_u-v\cdot \nabla u)\d x\d s
+\int_0^t\langle \varphi ^2\nabla u, \nabla (u\times \circ\d  W)\rangle.
\end{aligned}
\]
Integrating by parts and using that a.e.\ $\Delta u\cdot \tt_u =|\tt_u|^2_{\R^3},$
we get:
\[\begin{aligned}
E^\varphi (t)-E^\varphi (0) 
&+ \iint_{[0,t]\times\O}\varphi ^2\left(|\nabla v|^2+|\tt_u|^2\right)\d x\d s
\\
&=-\iint_{[0,t]\times\O}\varphi ^2v^jv^i\partial _i v^j\d x\d s
-\iint_{[0,t]\times\O}\varphi^2 v\div(\nabla u\odot\nabla u)\d x\d s
\\
&\quad -\iint_{[0,t]\times\O}\varphi^2 v\cdot \nabla \pi\d x\d s
-2\iint_{[0,t]\times\O}\varphi (\partial _j\varphi ) \partial _ju^i\tt_u^i\d x\d s
\\
&\quad +\iint_{[0,t]\times\O}\partial _j(\varphi^2 \partial _ju^i)v^\ell \partial _\ell u^i\d x\d s
+\int_0^t\langle \varphi ^2\nabla u,u\times\nabla\circ \d  W\rangle
\\
&=:\int_0^t\sum\nolimits_{m=1}^6\T_m\d s,
\end{aligned}
\]
where we make use of a summation convention whenever it is convenient.
We can now fix the time variable $s\in[0,t]$ and evaluate each term separately.

First, using $\div v=0,$ we have
\[
\T_1=\int_{\O} \varphi (\partial _i\varphi )(v^j)^2v^i\d x.
\]
Therefore, from the assumptions on $\varphi ,$
and since $1/\rho\leq 1+1/\rho^2,$
it holds
\[
\begin{aligned}
\T_1
\leq \int_\O\varphi |\nabla \varphi ||v|^3\d x
\leq C\Big(1+\frac{1}{\rho^2}\Big)\int_\O|v|^3\d x.
\end{aligned}
\]

Similarly for the fourth term, we have by Young Inequality:
\[
\T_4\leq \frac{C}{\rho^2}\int_\O |\nabla u|^2\d x\d s + \frac12\int_\O\varphi ^2|\tt_u|^2\d x,
\]
the second term of which will be eventually absorbed to the left.

For the third term,
using that $\div v=0,$ one can write 
\[
\begin{aligned}
\T_3
&= \int_\O2\varphi (\partial _i\varphi) v^i\tilde\pi\d x
\\
&\leq C\Big(1+\frac{1}{\rho^2}\Big)\int_\O |v||\tilde \pi|\d x\,.
\end{aligned}
\]

Taking into account similar compensations as for the global estimates, we estimate the
second and fifth terms together, which yields:
\[
\begin{aligned}
\T_2+\T_5
&=
-\int_\O \varphi ^2v^\ell \partial_j (\partial _ju^i\partial _\ell u^i) \d x
+\int_{\O}2\varphi (\partial _j\varphi) \partial _ju^iv^\ell \partial _\ell u^i\d x
+\int_{\O}\varphi^2\partial _{jj}u^iv^\ell \partial _\ell u^i\d x
\\
&=-\int_\O \varphi ^2v^\ell \partial _ju^i\partial _{j\ell }u^i\d x +2\int_\O\varphi (\partial _j\varphi) \partial _ju^iv^\ell \partial _\ell u^i\d x
\\
&=:\mathrm{I}+\mathrm{II}
\end{aligned}
\]
Integrating by parts and using again that $\div v=0,$ we have on the one hand
\[
\mathrm{I}=\int_\O\varphi (\partial _\ell \varphi )v^\ell (\partial _j u^i)^2\d x
\leq C\Big(1+\frac{1}{\rho^2}\Big)\int_\O|v||\nabla u|^2\d x.
\]
The evaluation of $\mathrm{II}$ yields a similar bound, namely one has:
\[
\mathrm{II}\leq C\Big(1+\frac{1}{\rho^2}\Big)\int_\O|v||\nabla u|^2\d x,
\]
for another such constant $C>0.$

Finally,
computing the Stratonovitch integral as in \eqref{above_strato}, we have
\[
\begin{aligned}
\T_6
&=\int_0^t\langle \varphi \nabla u,\varphi u\times \nabla \d  W\rangle
\\
&\quad \quad \quad 
+\frac12\sum_{l\in\N}\sum_{j=1}^3 \int _0^t\Big [\langle\varphi ^2 (u\times \CC_l^j)\times\nabla \CC_l^j,\nabla u\rangle 
+\langle \varphi ^2u\times\nabla \CC_l^j, \nabla (u\times \CC_l^j)\rangle \Big ]\d s.
\end{aligned}
\]
Hence, similar computations as that of \eqref{computations_A1} and \eqref{computations_A2}
yield that
\[
\begin{aligned}
\T_6 
&= \int_0^t\langle \varphi \nabla u,\varphi u\times \d  W\rangle
+t\sum_{l\in\N} \int_\O \varphi (x)^2|\nabla \C_l(x)|^2\d x
\\
&=X^\varphi (t) +C^\varphi _\C t\,.
\end{aligned}
\]

Now, gathering all the above bounds, integrating in time, and absorbing to the left when needed, we end up with \eqref{ineq:local}.
\end{proof}
\section{Bootstrap}
\label{sec:bootstrap}
In this section, we will see that provided there is no energy concentration at some $t\geq 0,$ then the solution $u$ can be extended continuously after $t$ (in a suitable space), for a positive time.

Nore precisely, we aim to show the following.

\begin{theorem}
\label{thm:bootstrap}
Consider $(W,\C)$ as in Assumption \ref{ass:W} with $\C\in \hs{3}.$
Let $(v,u)$ be a strong solution to \eqref{SEL} on $[0,\tau )$ with respect to $(\mathfrak P,W),$ such that furthermore $(v,u)$ belongs to $\V^2_{\loc,[0,\tau )} \times \U^3_{\loc,[0,\tau )}$ and where the stopping time $\tau $ is defined by the property that 
\[
\tau <T \quad \text{if and only if} \quad \limsup_{t\nearrow \tau}|v(t)|_{H^2}+|u(t)|_{H^3}=\infty.
\]

	For arbitrary $\rho >0$ and $\varepsilon _1\in (0,\mu_1^{-1}),$ denote by 	
	\begin{equation}\label{nota:zeta}
	\zeta (v,u;\rho,\varepsilon_1):=\inf\left\{t\in[0,\tau):\sup_{s\leq t,\,x\in\O}\int_{B(x,\rho)}\left(|v(s,y)|^2+|\nabla u(s,y)|^2\right)\d y\geq \varepsilon_1\right\}\,.
	\end{equation}

Then, we have for each $m\in \N,$
\[
(v(\cdot \wedge\zeta ),u(\cdot \wedge\zeta ))\in L^m\left(\Omega ;\V^2_{0,T}\times \U^3_{0,T}\right)\,,
\]
and its corresponding norm is bounded in terms of the quantities
$m$, $\rho$, $T$, $|\C|_{\hs{3}}$, $|u_0|_{H^3}$, $|v_0|_{H^2}$ only.

In particular, for any $\rho,\varepsilon _1,$ as above, we have
$\P(\zeta(v,u;\rho,\varepsilon_1)<\tau)=1.$
\end{theorem}

\subsection{Estimates of $u$ in $L^2_{\mathrm{loc}}(\HH^2)$}
The first step is to show that for the concentration time $\zeta$ defined above,
the quantity
$\E[\|u(\cdot \wedge\zeta )\|^m_{L^2(0,T;\HH^2)}]$
is bounded, in terms of $m,T,|\C|_{\hs{1}},$ $|v_0|_{L^2}$ and $|u_0|_{H^1}$ only. 
\begin{lemma}[$L^2_{\mathrm{loc}}(\HH^2)$ estimate for $u$]
\label{lem:L2H2}
Let $v,u,\varepsilon_1,\rho$ be as in Theorem \ref{thm:bootstrap}, and consider $(W,\C)$ as in Assumption \ref{ass:W}.
For each $\rho>0,\varepsilon_1>0,$ define the stopping time $\zeta(v,u;\rho,\varepsilon_1)$ as in \eqref{nota:zeta}.

For any choice of $\varepsilon_1\in(0,\frac{1}{\mu _1}),$
where $\mu _1>0$ is the constant appearing in Proposition \ref{pro:interp2},
and for every $m\geq 1,$ it holds for any $t\in[0,T]:$
\begin{multline}
\label{bound:L2H2_equi}
\E\left[\left(\int_0^{t\wedge \zeta (v,u;\rho,\varepsilon_1)}|\Delta u(s)|^2_{L^2}\d s\right)^m\right]
\\
\leq C(\rho,m,|\C|_{\hs{1}},T)\E \left[\left(\int_0^{t\wedge\zeta (v,u;\rho ,\varepsilon_1)} |\nabla u|^2_{L^2}\d s\right)\right].
\end{multline}
In particular, combining with Proposition \ref{pro:global_estimates}, we have an estimate
\begin{equation}
\label{bound:L2H2}
\E\left[\left(\int_0^{T\wedge \zeta (v,u;\rho,\varepsilon_1)}|\Delta u(s)|^2_{L^2}\d s\right)^m\right]
\leq C(\rho,m,|\C|_{\hs{1}},T,E(0))\,.
\end{equation}
\end{lemma}
\begin{proof}
For a solution $(v,u)$ as above we have $u\perp \tt_u,$
so that proceeding as in \eqref{id_tension}, and using that, in the weak sense,
\[
0\equiv\Delta \big(\frac{|u|^2}{2}\big) = (\Delta u\cdot u) + |\nabla u|^2,
\]
then it holds true that $|\tt_u|^2_{\R^3}=|\Delta u|^2_{\R^3} -|\nabla u|_{\R^{2\times3}}^4.$
Applying now Proposition \ref{pro:global_estimates}, and expanding the tension,
we obtain
\[
\begin{aligned}
E(t)-E(0)
&+\iint_{[0,t]\times\O}(|\Delta u|^2 +|\nabla v|^2)\d x\d s -X(t)-tC_\C 
\\
&=\iint_{[0,t]\times\O}|\nabla u|^4\d x\d s.
\end{aligned}
\]
Next, using Proposition \ref{pro:interp2}, we obtain for times $t\leq \zeta (v,u,\rho ,\varepsilon _1)$:
\begin{equation}
\label{pre:L2H2}
\begin{aligned}
E(t)
&-E(0) +\iint_{[0,t]\times\O}(|\Delta u|^2+|\nabla v|^2)\d x\d s -X(t)-tC_\C 
\\
&\leq \mu _1 \left(\sup_{s\in[0,T],\,x\in\O}\int_{B(x,\rho)}|\nabla u(s,y)|^2\d y\right)
\iint_{[0,t]\times\O}\Big(|\Delta u|^2 +\frac{C}{\rho^2}|\nabla u|^2\Big)\d x\d s
\\
&\leq \mu _1\varepsilon_1\iint_{[0,t]\times\O}|\Delta u|^2\d x\d s +\frac{\mu _1\varepsilon_1C}{\rho^2}\iint_{[0,t]\times\O}|\nabla u|^2\d x\d s\,,
\end{aligned}
\end{equation}
In the case where $\varepsilon_1<\mu _1^{-1},$ the first term above can be absorbed to the left, and the conclusion follows by estimating the martingale part as in Proposition \ref{pro:global_estimates}.
\end{proof}

Note that, in the proof of Lemma \ref{lem:L2H2}, we did not need the local estimate for the part involving $|v|^2.$
However, it will play a role in the following estimates.

\begin{lemma}
\label{pro:exp}
Under the assumptions of Lemma \ref{lem:L2H2}, for any $m\geq 1, $ there hold the exponential bounds:
\begin{align}
 \label{bound:exp1}
 \E \left[\exp \left\{m\, \sup_{t\in[0,\zeta ] }(|v(t)|_{L^2}^2+|\nabla u(t)|^2_{L^2})\right\}\right]
 &\leq C(m,\rho)
\\
\label{bound:exp2}
 \E \left[\exp \left\{m\int_0^\zeta (|\nabla v|_{L^2}^2+|\Delta u|^2_{L^2})\d t\right\}\right]
 &\leq C(m,\rho,T,E_0,|\C|_{\hs{1}})
 \\
 \label{bound:exp3}
 \E \left[\exp \left\{m\int_0^\zeta (|v|^4_{L^4}+|\nabla u|^4_{L^4})\d t\right\}\right]
 &\leq C(m,\rho,T,E_0,|\C|_{\hs{1}}).
\end{align}
\end{lemma}

\begin{proof}
\textit{Proof of \eqref{bound:exp1}.}
Fix $m\geq 1.$
To prove the bound, it is sufficient to observe that the random variable
$\sup_{t\in[0,\zeta]}|v(t)|^2_{L^2}+|\nabla u|_{L^2}^2$
belongs to $L^\infty(\Omega ).$
Indeed, as a consequence of the definition of $\zeta $ in \eqref{nota:zeta} we have
\[
\sup_{t\in[0,\zeta ]}(|v(t)|^2_{L^2}+|\nabla u(t)|^2_{L^2} )
\leq \sum_{1\leq i\leq N_\rho}\sup_{t\in[0,\zeta ]}\int_{B(x_i,\rho)}(|v(t,y)|^2+|\nabla u(t,y)|^2_{L^2})\d y
\leq N_\rho\varepsilon_1,
\]
for any finite sequence $\{x_1,\dots x_{N_\rho}\}$ such that $\cup_{i\leq N_\rho}B(x_i,\rho)=\O$
(note that asymptotically $N_\rho\sim C/\rho^2,$ but this has no importance here).
This proves our claim, and hence \eqref{bound:exp1}.

~\paragraph{\itshape Proof of \eqref{bound:exp2}}
Going back to the proof of Lemma \ref{lem:L2H2}, we see in particular from \eqref{pre:L2H2} and \eqref{estim:energy} that $\P$-a.s.,
for any $t\in[0,T]:$
\begin{equation}
 \label{ineq:Mt}
m\int_0^{t\wedge\zeta }(|\Delta u|_{L^2}^2 +|\nabla v|_{L^2}^2)\d s
\leq C(1+X(t\wedge\zeta )),
\end{equation} 
where the constant $C>0$ depends only on the quantities $\varepsilon_1,m,\rho,E_0,T$ and $|\C|_{\hs{1}}.$

Next, the process $S(t):=\exp \left\{m X(t)\right\},t\in[0,T],$ is a submartingale, and It\^o Formula yields,
for $t\le \zeta$,
\[
\begin{aligned}
S(t)-S(0) -m\int_0^t S(s)\langle \nabla u,u\times\nabla \d  W\rangle
&=m^2\int_0^t S(s)|\C^* \div (u\times \nabla u)|^2_{L^2}\d s
\\
&\leq C(|\C|_{\hs{1}},\rho,\varepsilon_1,m)\int_0^t S(s)\d s.
\end{aligned}
\]
Hence, taking the expectation, and then applying Gronwall Lemma yields the estimate
\[
\sup_{0\leq t\leq T}\E\left[ S(t\wedge\zeta )\right]\leq C(T,|\C|_{\hs{1}},\rho,\varepsilon_1,m)
\]
where $C$ is as above. Applying Doob's inequality for submartingales, we obtain that
\[
\E\left[\sup_{t\leq \zeta }\exp\left\{mX(t)\right\}\right]
\leq C \sup_{t\leq T } \E [S(t\wedge\zeta )]
\]
for another such constant.
Taking the exponential and expectation in \eqref{ineq:Mt},
we end up with the desired bound.

~\paragraph{\itshape Proof of \eqref{bound:exp3}}
By Proposition \ref{pro:interp2}, one observes that
\begin{multline*}
 \E\left[\exp m\int_0^\zeta (|v|_{L^4}^4+|\nabla u|^4_{L^4})\d t\right]
\\
 \leq \E\left[\exp\left\{C\varepsilon_1m \int_0^\zeta( |\nabla v|_{L^2}^2+|\Delta u|^2_{L^2})\d t\right\}
 \exp\left\{\frac{C\varepsilon_1m}{\rho^2}\int_0^\zeta (|v|_{L^2}^2+|\nabla u|_{L^2}^2)\d t\right\}\right],
\end{multline*}
and the conclusion follows by H\"older Inequality, together with \eqref{bound:exp1} and \eqref{bound:exp2}.
\end{proof}

\subsection{Higher regularity}
Fix $\rho >0$ and denote $\zeta _\rho :=\zeta (v,u;\rho ,\varepsilon_1).$
The previous paragraph shows that, if $(v,u)$ is a strong solution restricted to the time interval $[0,\zeta (\rho )],$
then a uniform bound holds for $(v,u)$ in $L^2_t(\H^1_x)\times L^2_t(\HH^2_x)$, and consequently for the pressure 
$\pi$ as well, in $L^2_t(L^2_x)$ (as may be seen using \eqref{estim:pi_0}).
However, this is not enough to conclude immediately that $\big(u(\zeta _\rho ),v(\zeta _\rho )\big)$ is in $\H^1\times \HH^2,$ $\P$-a.s., as needed.
As shown below, this property is true for smooth enough data $(v_0,u_0,\C)$ (we even have $(u(\zeta _\rho ),v(\zeta _\rho ))\in\H^2\times \HH^3$) but the proof requires a bootstrap argument in the spirit of \cite[Section 3]{St1}, which is made difficult by the presence of noise in the equation.

The first step is to show the following.
\begin{proposition}[Higher order estimates]
\label{pro:higher}
Under the assumptions of Lemma \ref{lem:L2H2},
assume in addition that there exists $\rho>0,$ such that a.s.:
\begin{equation}
\sup_{x\in\O,\,t\in[0,T]}\int_{B(x,\rho )}(|v(t,y)|^2+|\nabla u(t,y)|^2)\d y<\varepsilon_1\,,
\end{equation}
and also that a.s.\ :
\begin{align}
v\in \V_T^{1},\quad 
u\in \U_T^{2},\quad 
\text{and}\quad 
\C\in\hs{3}\,.
\end{align}

Then, for every $m\in[1,\infty)$ we have the estimate:
\begin{multline}
\label{pathwise_bound}
\E\Bigg[\sup_{t\in[0,T]} \left(
|\ttt(t)|_{L^2}^2 +|\Delta u(t)|^2_{L^2} +|\nabla v(t)|^2_{L^2} \right)^m
\\
+\Big(\int_0^T\big[|\nabla \Delta u(t)|^2_{L^2}+|\nabla \ttt(t)|_{L^2}^2 +|\nabla^2 v(t)|^2_{L^2}
+|\partial _tv(t)|_{L^2}^2\big]\d t\Big)^m\Bigg]
\\
\leq C(m,|\C|_{\hs{3}},|u(0)|_{H^2},|v(0)|_{H^1})\,,
\end{multline}
for a.e.\ $\omega \in\Omega ,$ where $\ttt$ is the ``corrected tension''
\begin{equation}
\label{nota:corrected_tension}
\ttt:=\Delta u+u|\nabla u|^2 -v\cdot \nabla u\,.
\end{equation}
\end{proposition}
\begin{proof}
{\itshape Step 1. Pathwise Gronwall.}
In order to prove \eqref{pathwise_bound}, we are going to apply It\^o Formula to the quantity
\begin{equation}
\label{nota:G}
G(t):= |\Delta u(t)|^2_{L^2}+|\ttt(t)|_{L^2}^2 +|\nabla v(t)|^2_{L^2}
,\quad t\in [0,T],
\end{equation}
and then conclude by applying Gronwall Lemma pathwise.

As a first step, we will show the following.

\bigskip

\textbf{Claim.}
{\itshape
We have for any $t\in[0,T]:$
\begin{multline}
\label{bound:increment}
G(t)-G(0)
+\int_0^t\Big(|\nabla \Delta u(s)|^2_{L^2}+|\nabla \ttt(s)|_{L^2}^2 +|\nabla^2 v(s)|^2_{L^2}
+|\partial _tv(s)|_{L^2}^2\Big)\d s
\\
\leq \chi(t)\left(\sup_{s\in[0,t]} G(s)\right)+ Y(t)\,,
\end{multline}
where we let
\[
\chi(t):=C\int _0^t(|v(s)|_{L^4}^4+|\nabla u(s)|^4_{L^4})\d s,
\]
for a sufficiently large but universal constant $C,$
while $Y$ denotes the semi-martingale 
\begin{multline}
\label{nota:Y}
Y(t)= 
2\int_0^t\langle \ttt,\partial _\beta u\times\circ \d  \partial _\beta W\rangle
+\int_0^t\langle \ttt,u\times\circ\d  \Delta W\rangle
-\int_0^t\langle \ttt, v^i(u\times\circ\d  \partial _iW)\rangle
\\
+2\int_0^t\langle \Delta u,\partial _\beta u\times\circ\d \partial _\beta W\rangle
+\int_0^t\langle \Delta u,u\times\circ\d \Delta W\rangle\,,
\end{multline}
for $t\in[0,T].$
}

\bigskip

To prove the claim, we proceed as in \cite[Lemma 6]{Ho1},
an essential difference being that here the (ill-defined) time derivative of $u$ must be replaced by the corrected tension $\ttt$ defined in \eqref{nota:corrected_tension}. 

As a first observation note that, similarly as for $\tt$, the corrected tension is a.e.\ orthogonal to $u.$ Indeed, we have by definition
\[
\ttt=\tt_u-v\cdot \nabla u
\]
but $\tt_u \perp u$ and
\begin{equation}
\label{ttt_perp_u}
u\cdot (-v\cdot \nabla u)=-\sum_{i=1}^3u^i\sum_{j=1}^2v^j\partial _ju^i
=-\frac12\sum_{j=1}^2v^j\partial _j\left(|u|^2\right)\equiv0\,.
\end{equation}
We will now apply It\^o Formula to each of the quantities appearing in \eqref{nota:G}.

First, using the semi-martingale decomposition of $u$ and the definition \eqref{nota:corrected_tension}, then It\^o Formula applied to the quantity $\frac12|\ttt|_{L^2}^2$ gives,
using $\div v=0$:
\begin{equation}\label{ito_tt}
\begin{aligned}
\frac12|\ttt(t)|^2_{L^2}-\frac12|\ttt(0)|_{L^2}^2 
&=\int_0^t\left\langle\ttt,\Delta \ttt+\ttt|\nabla u|^2 + u\left(\nabla u\cdot \nabla \ttt \right)\right\rangle\d s
\\
&\quad \quad + M(t),
\end{aligned}
\end{equation}
for all $0\leq s\leq t\leq \zeta ,$ a.s.\ ,
where $M$ denotes the semi-martingale
\begin{equation}
\label{nota:M}
\begin{aligned}
M(t)
&:=\int_0^t\langle \ttt,\Delta (u\times\circ\d  W) +|\nabla u|^2u\times\circ \d  W+2u\nabla (u\times\circ\d  W)\cdot \nabla u -v\cdot \nabla (u\times\circ\d  W)\rangle,
\\
&:=\int_0^t\langle \ttt,\Delta (u\times\circ\d  W)\rangle
+\int_0^t\langle \ttt,|\nabla u|^2u\times\circ\d W \rangle
-\int_0^t\langle \ttt,v^i\partial _i(u\times\circ\d  W)\rangle
\\
&=2\int_0^t\langle \ttt,\partial _\beta u\times\circ \d  \partial _\beta W\rangle
+\int_0^t\langle \ttt,u\times\circ\d  \Delta W\rangle
-\int_0^t\langle \ttt, v^i(u\times\circ\d  \partial _iW)\rangle\,,
\end{aligned}
\end{equation} 
where we have used the fact that the term $\int_0^t\langle \ttt,\ttt\times \circ \d W\rangle$ vanishes thanks to orthogonality, together with
\[
\int_0^t\langle \ttt,u\nabla (u\times\circ\d  W)\cdot \nabla u\rangle =0\,.
\]

Recalling that $\ttt\perp u,$ integration by parts yields the relation
\begin{equation}
\label{ito_tt2}
\frac12\left(|\ttt(t)|^2_{L^2}-|\ttt(0)|_{L^2}^2\right)
-M(t)+\int_0^t|\nabla \ttt|_{L^2}^2\d s
=\int_0^t\langle|\ttt|^2,|\nabla u|^2\rangle\d s.
\end{equation} 
Using \eqref{ineq:interp}, the above r.h.s.\ is estimated as:
\[
\begin{aligned}
\iint\limits_{[0,t]\times\O}
&|\ttt|^2|\nabla u|^2\d x\d s
\\
&\leq\left(\int_0^t|\ttt|_{L^4}^4\d s\int_0^t|\nabla u|^4_{L^4}\d s\right)^{1/2}
\\
&\leq\sqrt{\mu _0}\left(\sup_{s\in[0,t]}|\ttt(s)|_{L^2}^2\right)^{1/2}\left(\int_0^t|\nabla \ttt|^2_{L^2}\d s\right)^{1/2}\left(\int_0^t|\nabla u|^4_{L^4}\d s\right)^{1/2}
\\
&\leq C\left(\sup_{s\in[0,t]}|\ttt(s)|_{L^2}^2\right)\chi(t)
+\frac12\int_0^t|\nabla \ttt|^2_{L^2}\d s
\,,
\end{aligned}
\]
by definition of $\chi(t).$
Hence, from \eqref{ito_tt2} we obtain the bound
\begin{equation}
\label{bootstrap_1}
\frac12\left(|\ttt(t)|^2_{L^2}-|\ttt(0)|_{L^2}^2\right)
+\frac12\int_0^t|\nabla \ttt|_{L^2}^2\d s -M(t)
\leq\chi(t) \left(\sup_{[0,t]}G(s)\right)
\,.
\end{equation}

Let us now apply It\^o Formula on the second term in \eqref{nota:G}.
Using again the semi-martingale decomposition of $u,$ together with the second-order Leibniz Formula for the laplacian, it holds
\begin{multline}
\label{ito_laplace}
\frac12\big(|\Delta u(t)|^2_{L^2}-|\Delta u(0)|_{L^2}^2\big)
+\int_0^t|\nabla \Delta u|_{L^2}^2\d s
\\
=\int_0^t\langle |\Delta u|^2,|\nabla u|^2\rangle\d s
+4\int_0^t\langle\Delta u,\nabla u\nabla ^2u\nabla u \rangle\d s
+2\int_0^t\langle \Delta u,u(\nabla\Delta  u\cdot\nabla u)\rangle\d s
\\
+2\int_0^t\langle \Delta u,u|\nabla^2u|^2\rangle\d s
+\int_0^t\langle\nabla\Delta  u ,\nabla v\nabla u+v\nabla ^2u\rangle\d s
+N(t)\,,
\end{multline}
where $N(t)$ is defined as
\begin{equation}
\begin{aligned}
\label{nota:N}
N(t)
&=\int_0^t\langle\Delta u,\Delta [u\times\circ \d W]\rangle
\\
&:=\int_0^t\langle \Delta u,\Delta u\times\circ\d W\rangle
+2\int_0^t\langle \Delta u,\partial _\beta u\times\circ\d \partial _\beta W\rangle
+\int_0^t\langle \Delta u,u\times\circ\d \Delta W\rangle
\\
&=2\int_0^t\langle \Delta u,\partial _\beta u\times\circ\d \partial _\beta W\rangle
+\int_0^t\langle \Delta u,u\times\circ\d \Delta W\rangle\,,
\end{aligned}
\end{equation}
where we have made use of Leibniz Formula and orthogonality.

With the exception of the fourth term in the r.h.s.\ of \eqref{ito_laplace}, each term can be evaluated by \eqref{ineq:interp}
 and H\"older Inequality, in order to yield an estimate of the form \eqref{bound:increment}.
 For the fourth term, we write that
\[
2\int_0^t\langle \Delta u,u|\nabla^2u|^2\rangle\d s
\equiv2\iint_{[0,t]\times\O}(\Delta u\cdot u)|\nabla ^2u|^2\d x\d s
=-2\iint_{[0,t]\times\O}|\nabla u|^2|\nabla ^2u|^2\d x\d s\,,
\]
using once more the identity $0=\Delta (\frac{|u^2|}{2}) = u\cdot \Delta u +|\nabla u|^2. $

Hence, using Young Inequality several times, it is easily seen that for any $\delta _1>0$ fixed, there is a constant $C_1>0$ such that
\begin{multline}
\label{bootstrap_2}
\frac12\big(|\Delta u(t)|^2_{L^2}-|\Delta u(0)|_{L^2}^2\big)
+\int_0^t|\nabla \Delta u|_{L^2}^2\d s -N(t)
\\
\leq C_1\iint_{[0,t]\times\O}|\Delta u|^2|\nabla u|^2\d x\d s
+C_1\iint_{[0,t]\times\O}|\nabla^2 u|^2|\nabla u|^ 2\d x\d s
\\
\quad +\delta _1\int_0^t|\nabla \Delta u|^2_{L^2}\d s
+\iint_{[0,t]\times\O}|\nabla v|^2|\nabla u|^2\d x\d s
+\iint_{[0,t]\times\O}|v|^2|\nabla^2 u|^2\d x\d s
\\
=:\mathrm{I}+\mathrm{II}+\mathrm{III}+\mathrm{IV}+\mathrm{V}\,.
\end{multline}
To estimate the first term, we use again \eqref{ineq:interp}, which gives:
\begin{multline*}
\iint_{[0,t]\times\O}|\Delta u|^2|\nabla u|^2\d x\d s \leq 
\left(\iint_{[0,t]\times\O}|\Delta u|^4\d x\d s\right)^{1/2}\left(\iint_{[0,t]\times\O}|\nabla u|^4\d x\d s\right)^{1/2}
\\
\leq \sqrt{\mu _0}\left(\sup_{s\in[0,t]}|\Delta u|_{L^2}^2\iint_{[0,t]\times\O}|\nabla \Delta u|^2\d x\d s\right)^{1/2}\chi(t)^{1/2}
\end{multline*}
Thus, we have for any $\delta _2>0$:
\begin{equation}
\label{delta_u_nabla_u}
\mathrm{I}\leq C_2\chi(t)\sup_{s\in[0,t]}G(s) + \delta _2\int_0^t|\nabla \Delta u|^2_{L^2}\d s,
\end{equation} 
where $C_2>0$ only depends on $\delta _2.$

Similarly, estimating the $H^2$ norm by the Laplacian, we have for the second term
\begin{equation}
\label{delta_u_nabla_u_2}
\mathrm{II}\equiv C_1\iint_{[0,t]\times\O}|\nabla^2 u|^2|\nabla u|^ 2\d x\d s
\leq 
C_3\chi(t)\sup_{[0,t]}G(s) + \delta _3\int_0^t|\nabla \Delta u|^2_{L^2}\d s\,.
\end{equation}

Concerning the fourth term in \eqref{bootstrap_2}, we proceed similarly to obtain that
\begin{equation}
\label{term4}
\mathrm{IV}
\leq C_4\chi(t)\sup_{s\in[0,t]}G(s) + \delta _4\int_0^t|\nabla ^2v|_{L^2}^2\d s\,.
\end{equation} 

Similarly,
\begin{multline}
\nonumber
 \iint_{[0,t]\times\O} |v|^2|\nabla ^2u|^2\d x\d s
\leq \sqrt{\mu _0}\left(\iint_{[0,t]\times\O}|v|^4\d x\d s\right)^{1/2}
\\
\times\left(\sup_{s\in[0,t]}|\nabla ^2 u |^2 \iint_{[0,t]\times\O}|\nabla ^3u|^2\d x\d s\right)^{1/2}
\end{multline}
and thus:
\begin{equation}
\mathrm{V}
\leq C_5\chi(t)\sup_{s\in[0,t]}G(s) + \delta _5\int_0^t|\nabla \Delta u|^2_{L^2}\d s,
\end{equation}

It remains to evaluate the velocity terms.
First, multiply the equation on $v$ by $\Delta v$ (which is also divergence free), and integrate on $[0,t]\times\O$ to get:
\begin{multline}
\label{ito_nabla_v}
|\nabla v(t)|_{L^2}^2-|\nabla v(0)|^2_{L^2} +\int_0^t|\Delta v|_{L^2}^2\d s
\\
=\int_0^t\langle\Delta v, \mathscr P[ -v\cdot \nabla v -\div(\nabla u\odot\nabla u) ]\rangle\d s\,.
\end{multline}
Note that, integrating by parts
\[
-\iint_{[0,t]\times\O} \Delta v\cdot (v\cdot \nabla v)\d x\d s
=\iint_{[0,t]\times\O} \Delta v\cdot (v\cdot \nabla v)\d x\d s
=0\,.
\]
Moreover, it holds for $i=1,2:$
\[
[\div(\nabla u\odot\nabla u)]^i
=\partial _{jj}u \cdot \partial _iu + \partial _{j}u\cdot \partial _{ij}u
=
[\Delta u\cdot \nabla u + \frac12 \nabla (|\nabla u|^2)]^i\,,
\]
and therefore
\[
\mathscr P[\div(\nabla u\odot\nabla u)]=\mathscr P[\Delta u\cdot \nabla u]\,.
\]
The last term in \eqref{ito_nabla_v} is then estimated as:
\[
\int_0^t\langle\Delta v,-\div(\nabla u\odot\nabla u) \rangle\d s
\leq \delta _6 \iint_{[0,t]\times\O}|\Delta v|^2\d x\d s+ C_6\iint_{[0,t]\times\O}|\mathscr P[\Delta u\cdot \nabla u]|^2\d x\d s\,.
\]
Using \eqref{ineq:interp} and the continuity of $\mathscr P,$ we have on the other hand
\begin{multline}
\label{bootstrap_3}
\iint_{[0,t]\times\O}|\mathscr P[\Delta u\cdot \nabla u]|^2\d x\d s
\leq C(\mathscr P)\int_0^t|\Delta u|_{L^4}^2|\nabla u|_{L^4}^2\d s
\\
\leq C_7\chi(t)\sup_{s\in[0,t]}G(s) + \delta _7\int_0^t|\nabla \Delta u|^2_{L^2}\d s\,,
\end{multline}
where we have estimated the right hand side as in \eqref{delta_u_nabla_u}.

Now, multiply the equation on $v$ by $\partial _tv$.
Integrating over $[0,t]\times\O,$ it holds
\[
\begin{aligned}
|\nabla v(t)|^2_{L^2}&
-|\nabla v(0)|^2_{L^2}+\int_0^t|\partial _tv|^2_{L^2}\d s
\\
&\leq -\iint_{[0,t]\times\O}v\cdot \nabla v\cdot \partial _tv \d x\d s
+\iint_{[0,t]\times\O}|\partial _tv||\Delta u||\nabla u|\d x\d s
\\
&\leq \delta _8\iint_{[0,t]\times\O}|\partial _tv|^2\d x\d s
\\
&\quad \quad \quad \quad 
+C_8\Big(\iint_{[0,t]\times\O} |\Delta u|^2|\nabla u|^2\d x\d s+\iint_{[0,t]\times\O}|\nabla v|^2|v|^2\d x\d s\Big)\,.
\\
&=: \delta _8\iint_{[0,t]\times\O}|\partial _tv|^2 + C_8(\mathrm{VI} + \mathrm{VII})\,.
\end{aligned}
\]
Using again \eqref{delta_u_nabla_u}, we obtain
\begin{equation}
\label{bootstrap_4}
\mathrm{VI}
\leq C_{9}\chi(t)\sup_{s\in[0,t]}G(s) + \delta _{9}\int_0^t|\nabla\Delta  u|^2_{L^2}\d s\,,
\end{equation} 
while
\begin{equation}
\label{bootstrap_5} 
\mathrm{VII}
\leq C_{10}\chi(t)\sup_{s\in[0,t]}G(s)+\delta _{10}\int_0^t|\nabla ^2v|_{L^2}^2\d s\,.
\end{equation} 

Finally, 
summing the inequalities \eqref{nota:M}, \eqref{bootstrap_1} and \eqref{nota:N}--\eqref{bootstrap_5}, we obtain \eqref{bound:increment}
by choosing the constants $\delta _i,i=1,\dots,10$ small enough.

~\paragraph{\itshape Step 2: the main stochastic estimate}
Using Burkholder-Davies-Gundy Inequality,
we have from Claim \ref{clm:M}:
\begin{multline}
\E\left[\sup_{s\in[0,t]}|\hat Y(s)|^{2m}\right]
\leq C(m)\E\left[\langle Y\rangle(t)^{m}\right]
\\
\leq
\E\Bigg[\Big(\int_0^t\Big\{
|\C^*\div(u\times\nabla \ttt)|_{L^2}^2
+|\C^*\Delta (u\times\ttt)|_{L^2}^2
+|\C^*\partial _i (v^iu\times\ttt)|_{L^2}^2
\\
+|\C^*\Delta (-u\times \Delta u)|_{L^2}^2
+|\C^*\div (u\times \nabla \Delta u)|_{L^2}^2
\Big\}\d s\Big)^{m}\Bigg]
\end{multline}

Thanks to the fact that $\hs{k}\hookrightarrow \mathscr L(L^2,H^k),$ for each $k=1,2,3,$ 
we deduce from the above estimate the existence of a constant $C(m,|\C|_{\hs{3}})>0$ so that
\begin{multline}
\label{sto:est}
 \E\left[\sup_{s\in[0,t]}|\hat Y(s)|^{2m}\right]
\leq 
C(m,|\C|_{\hs{3}})\E\Big[\Big\{\int_0^t\big(|\nabla \ttt(s)|_{L^2}^2 +|\ttt(s)|_{L^2}^2
\\
+|v(s)|^2_{L^4}|\ttt(t)|_{L^4}^2
+|\Delta u(s)|_{L^2}^2
+|\nabla \Delta u(s)|_{L^2}^2\big)\d s\Big\}^m\Big]\,.
\end{multline}

\bigskip

\paragraph{\itshape Step 3. Conclusion}
By Claim \ref{clm:M}, we have $Y(t) = A(t)+\hat Y(t)$
where $\hat Y$ is a martingale and $A$ has finite variation.
Moreover, we infer from \eqref{ito_cor} together with Sobolev embeddings and the fact that $|u|_{\R^3}= 1$, that:
\begin{multline}
\label{est:A}
|A(t)|
\leq C\sum_{l\in\mathbb{N}}|\nabla \psi |_{L^\infty}^2\int_0^t(|\nabla u(s)|^2 + |v(s)|^2+|\Delta \C_l|^2_{L^2})\d s
\\
\leq C(|\C|_{\hs{3}})\int_0^t(1+|\nabla u(s)|^2 + |v(s)|^2)\d s
\end{multline}
From \eqref{bound:increment} together with \eqref{est:A}, we infer
by a well-known generalization of Gronwall inequality that $\P$-a.s.:
\begin{multline}
\label{pre_pathwise_bound}
\mathscr G:=\sup_{t\in[0,T]} G(t) +
\int_0^T\big(|\nabla \Delta u(t)|^2_{L^2}+|\nabla \ttt(t)|_{L^2}^2 +|\nabla^2 v(t)|^2_{L^2}
+|\partial _tv(t)|_{L^2}^2\big)\d t
\\
\leq C\Big[G(0) + \sup\nolimits_{t\in[0,T]}|\hat Y(t)|\Big]\exp\left\{\int _0^T|v(t)|^4_{L^4}+|\nabla u(t)|^4_{L^4}\d t\right\}\,.
\end{multline}

Therefore, using \eqref{sto:est} we have for any $m\geq 1$ in \eqref{pre_pathwise_bound}, for arbitrary $\delta >0:$
\begin{multline*}
\E \left[\mathscr G^m \right]
\leq \frac{C}{2\delta }\E\left[\exp \left\{2m\int _0^T|v(t)|^4_{L^4}+|\nabla u(t)|^4_{L^4}\d t\right\}\right]
+\frac{C\delta }{2}\E\left[\sup|\hat Y(t)|^{2m}\right]
\\
\leq \frac{C(m,\rho ,T,E_0,|\C|_{\hs{1}})}{2\delta }
+\frac{C\delta }{2}\E\Big[
\Big\{\int_0^T\big(|\Delta u|^2_{L^2} +|\nabla \Delta u|_{L^2}^2
\\
+|\nabla \ttt|_{L^2}^2 +|\ttt|_{L^2}^2 +|v|^2_{L^4}|\ttt|_{L^4}^2
\big)\d t\Big\}^m\Big]\,,
\end{multline*}
and the conclusion follows by \eqref{ineq:interp} and absorption to the left, provided $\delta >0$ is taken sufficiently small.
\end{proof}

\subsection{Proof of Theorem \ref{thm:bootstrap}}

We now have all in hand to prove our main Theorem.
The proof is based on an idea of Debussche, De Moor and Hofmanov\'a \cite{debussche2015regularity}.

\bigskip

{\itshape Step 1: increasing the regularity of the stochastic convolution.}
	We define the stochastic convolution
	\begin{equation}
	Z(t):=\int _0^te^{t\Delta}u\times \d  W \equiv\int _0^te^{t\Delta}u\times \C \d \xi ,
	\end{equation}
  $\xi $ being the cylindrical Wiener process given formally by the infinite sum $\sum_{l\geq0}\mathbf B_l(\cdot )f_l.$
Note that the latter series is a well-defined element of $C (H^{-1-\epsilon })$ for any $\epsilon >0,$ where for simplicity we henceforth denote by $C(H^{s}):=C(0,\zeta ;H^s),$ $L^m (H^s):=L^m(0,\zeta ;H^s)$ and so on.
	
	By Proposition \ref{pro:higher}, for any $m\in (4,\infty)$ we have 
	\begin{equation}
	\label{bound:u_H2}
	\E\left[\|u\|_{L^\infty(H^2)}^m\right]<\infty,
	\end{equation}
	hence making use of the parabolic estimate (Prop.\ \ref{pro:parabolic_sto}) with 
	\[
	\delta :=\frac12 \in \big(0,1-\frac2m\big),
	\] 
	we obtain the bound
	\begin{multline}
	 \label{bound:X}
	\E\left[\|Z\|_{C(W^{2,4})}^m\right]
	\leq
	C\E\left[\|Z\|_{C(H^{5/2})}^m\right]
	\\
	\leq
	C'\E\left[\|u\times \CC \|_{L^m(\hs{2})}^m\right]
	\leq
	C''(|\C|_{\hs{4}},T)\E\left[\|u\|_{L^\infty(H^2)}^m\right]\,,
	\end{multline}
	where we have used the Sobolev embeddings $H^{5/2}\hookrightarrow W^{2,4}$ and $H^4\hookrightarrow W^{2,\infty}.$
	
	Next, we define $y:=u-Z$ which solves the following PDE with random coefficients:
	\begin{equation}
	\label{eq:y}
	\partial _t y -\Delta y=-v\cdot \nabla u+u|\nabla u|^2:=f+g.
	\end{equation}
	From \eqref{bound:u_H2} and the Sobolev embedding Theorem, we have 
	\begin{equation}
	\label{bound:f}
	g\in L^m(\Omega; C(L^p)),\quad 
	\text{for any}\quad p\in[1,\infty).
	\end{equation}
	Similarly, from Proposition \ref{pro:higher}, we get
	\begin{equation}
	\label{bound:v_H1}
	\E\left[\|v\|_{C(H^1)}^m\right]<\infty,
	\end{equation}
	and therefore, it holds as well
	\begin{equation}
	\label{bound:g}
	f\in L^m(\Omega; C(L^p)),\quad 
	\text{for any}\quad p\in[1,\infty).
	\end{equation}
	Let $p=4.$
	As a consequence of \eqref{eq:y},\eqref{bound:f},\eqref{bound:g}, one concludes from the standard parabolic estimates (we refer, e.g.\ to \cite{lions1968problemes}, see also \cite{grisvard1969equations}) that $y$ belongs to $L^m(\Omega;C(W^{2,4})),$ together with the bound
	\begin{multline}
	\label{bound:y}
	\E\left[\|y\|^m_{C(W^{2,4})}\right]
	\leq C\E\left[\|v\|^m_{C(L^8)}\|\nabla u\|^m_{C(L^8)}+ \|\nabla u\|^{2m}_{C(L^8)}\right]
	\\
	\leq C'\E\left[\|v\|_{C(H^1)}^{2m} +\|u\|_{C(H^2)}^{2m}\right].
	\end{multline}
	Using \eqref{bound:X}, then \eqref{bound:y} leads to better regularity for $u\equiv y +Z,$
	namely:
	\begin{equation}
	\label{bound:u_W_24}
	\E\left[\|u\|^m_{C(W^{2,4})}\right]
	\leq C(|\C|_{\hs{4}},T)\E\left[\|v\|_{C(H^1)}^{2m}+\|u\|^{2m}_{C(H^2)}\right].
	\end{equation}
	
	\bigskip
	
	\paragraph{\itshape Step 2: increasing the regularity of the solution}
	From the previous step, we also infer better regularity for the velocity term.
	Indeed, using the equation on $v$ and Sobolev embeddings, we obtain an estimate
	\[
	\begin{aligned}
	\|v\|_{C(W^{1,4})}
	&\leq \|\mathscr P (v\cdot \nabla v)\|_{C(W^{-1,4})} + \|\mathscr P\div\nabla u\odot\nabla u\|_{C(W^{-1,4})} 
	\\
	&\leq C\left(\|v\cdot \nabla v\|_{C(L^{4/3})} + \|\nabla u\odot \nabla u\|_{C(L^4)}\right)
	\\
	&\leq C\left(\|v\|_{C(L^4)}\|\nabla v\|_{C(L^2)} + \|\nabla u\|^{2}_{C(L^8)}\right).
	\end{aligned}
	\]
	Hence
	\begin{equation}
	\label{bound:v_W_14}
	\E\left[ \|v\|_{C(W^{1,4})}^m\right]
	\leq C\E\left[\|v\|_{C(H^1)}^{2m} +\|u\|_{C(H^2)}^{2m}\right].	
	\end{equation}
	
	Now, because of \eqref{bound:u_W_24} we have $\nabla g\equiv \nabla u|\nabla u|^2 +2u\nabla ^2u\nabla u\in L^m(\Omega;C(L^{2})),$ and because of \eqref{bound:v_W_14} it holds as well
	$\nabla f\equiv -(\nabla v)^T\cdot \nabla u -v\cdot \nabla^2 u\in L^m(\Omega;C(L^{2})).$ 
	Therefore, we obtain that $y\in L^m(\Omega;C(H^3)),$ and there holds
	\begin{equation}
	\E\|y\|_{C(H^3)}^m \leq C_m\E \left[1+\|v\|_{C(H^1)}^{2m}+\|u\|_{C(H^2)}^{2m}\right]\,.
	\end{equation}

	On the other hand, since $u=y+Z$ has gained $1/2$ degree of regularity, one can repeat Step 1 to obtain 
	\[
	\E \left[\|Z\|_{C(H^3)}^m\right]
	\leq C(m,|\C|_{\hs{5}},T)\E \left[1+\|v\|_{C(H^1)}^{2m}+\|u\|_{C(H^2)}^{2m}\right].
	\]
    We eventually obtain the needed bound on $u,$ namely
	\[
	\E\left[ \|u\|^{m}_{C(H^3)}\right]
	\leq C(|\C|_{\hs{5}},T)\E \left[1+\|v\|_{C(H^1)}^{2m}+\|u\|_{C(H^2)}^{2m}\right].
	\]
	Finally, the bounds on $v$ follow by reusing the parabolic estimates together with appropriate Sobolev embeddings.
	This finishes the proof of Theorem \ref{thm:bootstrap}.
\hfill\qed

\section{Convergence of regular approximations}
\label{sec:convergence}
\subsection{Tightness and passage to the limit}
We now define a sequence $\{ W _n,n\in\N\}$ of Wiener processes in $L^2(\O;\R^3)$
whose covariance belongs to $\hs{s}$ for every $s\in\R.$
More precisely, for each $n\in\N$ we let
\begin{equation}
\label{nota:Wn}
 W_n:=\sum_{0\leq l \leq n} \mathbf B_l (\cdot )\C f_l \, ,
\end{equation}
for $(f_l )_{l\in\N}$ and $(\mathbf B_l )_{l\in\N}$ as in \eqref{nota:eta}.
Next, consider a sequence $(w_n,y_n;\tau_n)$ of strong solutions of \eqref{SEL} such that
\begin{multline}
\label{Hyp1}
w_n(0)\in C^\infty(\O;\R^2),\enskip \div w_n=0,
\quad y_n(0)\in C^\infty(\O;\R^3),\enskip |y_n(0)|_{\R^3}=1\enskip \text{a.e.\ }\quad 
\\
\text{and}\quad (w_n(0),y_n(0))\rightarrow (v_0,u_0) \enskip \text{in}\enskip \L^2\times\HH^1\enskip \text{strong},
\end{multline}
solving the following regularized problem on $[0,\tau _n)$:
\begin{equation}\label{approx}
\left\{
\begin{aligned}
&\d w_n + \big(Aw_n + \mathscr P[w_n\cdot \nabla w_n +\div\nabla y_n\odot\nabla y_n]\big)\d t=0
\\
&\d y_n-\big(\Delta y_n+y_n|\nabla y_n|^2-w_n\cdot \nabla y_n-F_{\C _n}y_n\big)\d t=y_n\times \d  W_n\,.
\end{aligned}
\right.
\end{equation}
Moreover we define the stopping time $\tau_n$ by the property:
\begin{equation}
\label{deftaun}
\P\text{-a.s., }\quad 
\tau _n=T\quad \quad \text{or}\quad \limsup_{t\nearrow\tau _n}[|w_n(t)|_{H^2}+|y_n(t)|_{H^3}]=\infty\enskip .
\end{equation}
Note that the existence of such a sequence is guaranteed by Theorem \ref{thm:localSolv}.
Now, fix $\epsilon_1\in(0,1/\mu _1)$ and choose a positive, non-increasing sequence 
\[
\rho_k\to0,\quad k\to\infty.
\]
For $n,k\in\N$,
define the stopping times:
\begin{equation}\label{zeta_tightness}
\zeta _{n,k}:=
\inf\Big\{0\leq t<\tau_n,\enskip\sup\limits_{x\in \O}\int\nolimits_{B(x,\rho_k)}\big(|w_n(t,y)|^2+|\nabla y_n(t,y)|^2\big)\d y\geq\epsilon_1\Big\},
\end{equation}
with the convention that the infimum is equal to $\tau _n$ whenever the above set is empty.
In the sequel, we will denote by $(v_{n,k},u_{n,k})$, $k\in\N$, the process:
\begin{equation}\label{nota:v_nk}
(v_{n,k}(t),u_{n,k}(t))=
\left\{
\begin{aligned}
&(w_n(t),y_n(t))\qquad\text{if}\enskip0\leq t\leq \zeta_{n,k}\,,
 \\[1em]
&\Big(e^{(\zeta _{n,k}-t)A^2}w_n(\zeta_{n,k}),e^{-(t-\zeta _{n,k})\Delta ^2}y_n(\zeta_{n,k})\Big)\quad\text{if}\enskip\zeta_{n,k}<t\leq T\,.
\end{aligned}
\right.
\end{equation}

\begin{notation}
For each sequence of random variables $X_k:\Omega \to E,k\in\N$ where $E$ is some Polish space,
we will henceforth denote by $\mathbf X$ the corresponding random variable in the product space $\prod _{k\in\N}E.$
For instance, we define for each $n\in\N:$
\[
\mathbf{v}_{n}:=\{v_{n,k},k\in\N\},\quad 
\mathbf{u}_{n}:=\{u_{n,k},k\in\N\},\quad 
\boldsymbol{\zeta }_{n}:=\{\zeta _{n,k},k\in\N\},
\]
and so on.
\end{notation}
\begin{claim}
\label{clm:tightness}
There exists $\alpha \in(0,\frac12)$ such that
for every $\delta<1$
the sequence 
\[
(\mathbf{v}_n , \mathbf{u}_n,\boldsymbol{\zeta }_n, W_n),\quad n\in\N
\]
is tight in the space 
\[
\mathfrak U^\delta :=\Big(\prod_{k\in\N}\V_T^{\delta -1}\Big)\times \Big(\prod_{k\in\N}\U_T^{\delta}\Big) \times \Big(\prod_{k\in\N}[0,T]\Big)\times C^\alpha ([0,T];H^1).
\]
\end{claim}

\begin{proof}
The proof relies in particular on the bound
\begin{multline*}
\E\left[\|v_{n,k}\|_{\V_{T}^0}^2\right]-1
\leq 
\E\left[\sup_{t\in[0,\zeta _{n,k}]}|v_{n,k}(t)|^2_{L^2} +\int_0^{\zeta _{n,k}}|\nabla v_{n,k}|_{L^2}^2\d s\right]
\\
+\E\left[\sup_{t\in[\zeta _{n,k},T]}\left|e^{A^2 (\zeta -t)}v_{n,k}(\zeta _{n,k})\right|^2_{L^2}
+\int_{\zeta _{n,k}}^T|e^{A ^2(\zeta _{n,k}-s)}\nabla v_{n,k}|_{L^2}^2\d s\right]
\\
\leq 
C\left(k,E_0,T,|\C|_{\hs{1}}\right)
+C\E\left[\int_{\zeta _{n,k}}^T\frac{|v_{n,k}|_{L^2}^2}{(s-\zeta _{n,k})^{1/2}}\d s\right]
\leq C''\left(k,E_0,T,|\C|_{\hs{1}}\right),
\end{multline*}
by \eqref{hypercontractivity} particularized for $\Lambda := (A^2,\H^4),$ and  $p=q=2.$

Proceeding similarly for $u_n$, one obtains as well the estimate 
\[
\E\left[\|u_n\|_{\U_T^1}^2\right]\leq C(k,E_0,T,|\C|_{\hs{1}}).
\]
In order to apply Aubin-Lions' lemma, one needs first higher order estimates with respect to the time-like variable, but in an arbitrary Sobolev space, typically of negative order. Such estimates follow directly from the equation \eqref{approx}, and hence the proof is omitted. The conclusion follows from Markov Inequality.
\end{proof}

Thanks to \ref{clm:tightness} and Prokhorov Theorem, we infer the existence of an extraction $n_\ell ,\ell \in\N$, and a law $\mu $ supported in $\mathfrak U^1$ such that
\[
\P_{(\mathbf v_{n_\ell },\mathbf u _{n_\ell},\boldsymbol\zeta _{n_\ell },W_{n_\ell })}\to \mu\quad   \text{weakly.}
\]
Furthermore, using Skorohod Embedding Theorem (see \cite{watanabe1981stochastic})
we see that there exist
\begin{itemize}
 \item a stochastic basis $\mathfrak {\Hat P}=(\Hat\Omega,\mathcal{\Hat A},\mathbb{\Hat P},(\mathcal{\Hat F}_t)_{t\in[0,T]},\Ws)$, where $\Ws$ is a Wiener process in $L^2(\O;\R^3)$ such that \eqref{covariance_law} holds;
 \item a sequence of random variables 
 \[
\big(\Bus_\ell ,\Bzs_\ell ,\Bzetas _\ell ,\Ws_\ell \big) 
\in L^0(\Hat\Omega,\mathfrak U^3)
,\quad \ell \in\N,
 \]
such that for each $\ell \in\N$, identifying $\Bus_\ell $ (resp.\ $\Bzs_\ell $, $\Ws_\ell $) as a mapping from $\Hat\Omega\times[0,T]$ to $\prod_{k\in\N}\L^2$ (resp.\ to $\prod_{k\in\N}\HH^1$, $\HH^1$), then the corresponding stochastic process is predictible with respect to the filtration 
\[
\mathcal{\Hat F}_t^\ell := \sigma \{\Bus_\ell (s),\Bzs_\ell (s), \Ws_\ell (s), \enskip s\leq t\},\quad  t\in[0,T];
\]
the element $\Ws_\ell$ is an $\LL^2$-Wiener process with respect to the above filtration,
and its covariance is given by \eqref{covariance_law} with $\C_{n_\ell }$ instead of $\C$;
for each $k\in\N,$ the random variable $\zetas_{\ell ,k}$ is a positive stopping time;
\item limits $(\Bzs,\Bus,\Bzetas,\Ws)\in \mathfrak U^1$, for every $k\in\N$,
\end{itemize}
such that the following convergences hold for each $k\in\N$:
\begin{align}
\label{cv1}
&(\vs_{\ell ,k},\us_{\ell ,k})\to(\vs_k,\us_k)\quad\mathbb{\Hat P}-\text{a.s.}
\quad \text{in}\enskip \V_T^{\delta -1}\times\U_T^{\delta} \enskip \text{for every}\enskip \delta <1,
\\
\label{cv2}
&\zetas_{\ell ,k}\to\zetas _k\quad \mathbb{\Hat P}-\text{a.s.},
\\
\label{cv3}
&\Ws_\ell \to \Ws \quad \mathbb{\Hat P}-\text{a.s.\ in}\enskip C^\alpha ([0,T];H^1)\enskip \text{for}\enskip \alpha\enskip \text{as above.}
\end{align}
Furthermore, by the Sobolev Embedding Theorem, together with Proposition \ref{pro:global_estimates} and dominated convergence, we immediately obtain 
\begin{equation}
\label{cv4}
\begin{aligned}
&(\vs_{\ell,k},\nabla \hat u_{\ell ,k })\to (\vs_k,\nabla \us_k),\quad \text{in}\enskip L^2(\Omega\times[0,T];L^p),\quad \text{for every}\enskip p\in[1,\infty)\enskip .
\\
&\hat u_{\ell ,k}\to \us_k\quad \text{in}\enskip L^2(\Omega \times[0,T];L^\infty)\,,
\end{aligned}
\end{equation} 
for each $k\in\mathbb{N}.$

In particular, making use of \eqref{cv1}, \eqref{cv2}, \eqref{cv3}, \eqref{covariance_law} and dominated convergence,
it is easily seen that for each $k\in\N,$
\begin{equation}
\label{solution_sko}
(\mathfrak {\Hat P},\Ws;\vs_k,\us_k) \enskip \text{is a martingale solution of \eqref{SEL}}
\end{equation}
up to the stopping time $t=\zetas _k,$ in the sense of Definition \ref{def:mart_sol}.
Details are left to the reader (see, e.g., \cite{alouges2014semi}).

\subsection{Improvement of the convergence}
\label{rem:improvement}
We are now going to improve the convergence of $(\vs_{\ell ,k},\us_{\ell ,k})$ in $\V^{\delta -1}_\loc\times \U^\delta _\loc $ by showing that it remains true for $\delta =1.$ 

\begin{claim}
\label{pro:improvement}
Let $(\vs_{\ell ,k},\us_{\ell ,k})$ be as in the previous paragraph.

The following convergence holds
\begin{equation}
\label{cvgence}
(\vs_n,\us_n)\to (\vs,\us)\quad \text{in}\enskip L^2(\Omega ;\V^0\times \U^1).
\end{equation}
\end{claim}
Before we proceed to the proof of Claim \ref{pro:improvement}, let us collect some inequalities.
Consider two martingale solutions $(\mathfrak P,W_1,v_1,u_1),$ and $(\mathfrak P,W_2,v_2,u_2)$ defined on the same probability space, and up to a common stopping time $\kappa >0,$ 
Denote further by  $g=v_1-v_2$ and by $f=u_1-u_2.$

Computing It\^o Formula on $\frac12|g(t)|^2$ we find that
\begin{equation}
\frac12|g(t)|_{L^2}^2 + \int_0^t|\nabla g|^2\d s
=\iint_{[0,t]\times\O}\big( -g^jg^i\partial _iv_1^j - g^jv_2^i\partial _i g^j \big)\d x\d s
+\mathfrak t(v_1,v_2;u_1,u_2)\,,
\end{equation}
where we denote by $\mathfrak t:(\V^0_{[0,\kappa ]})^2\times(\U^1_{[0,\kappa ]})^2\to \R$ the operation
\[
\mathfrak t:(v_1,v_2;u_1,u_2)\mapsto \iint_{[0,t]\times\O} (v_1^j-v_2^j)(\partial _{ii}u_2\partial _ju_2 - \partial _{ii}u_1\partial _ju_1)\d x\d s\,.
\]

Proceeding as in the proof of Theorem \ref{thm:uniqueness}, we have on the one's hand:
\begin{equation}
\label{stg_1}
|\iint_{[0,t]\times\O}-g^jg^i\partial _iv_1^j\d x\d s|
\leq  C(\epsilon )\int_0^t|g|^2_{L^2}( 1+|\nabla v_1|_{L^2}^2)\d s + \epsilon \int_0^t|\nabla g|_{L^2}^2\d s\,.
\end{equation} 
and similarly
\begin{equation}
 \label{stg_2}
|\iint_{[0,t]\times\O}- g^jv_2^i\partial _i g^j\d x\d s|
\leq C'(\epsilon )\int_0^t|g|_{L^2}(1+|v_2|_{L^4}^4)
+\epsilon \int_0^t|\nabla g|^2_{L^2}\d s\,.
\end{equation} 
Next, applying It\^o Formula to $\frac12|\nabla f|^2_{L^2},$ we obtain
\begin{multline}
\frac12|\nabla f(t)|^2_{L^2} + \int_0^t|\Delta f|^2_{L^2}\d s - X(t)
\\
=\iint_{[0,t]\times\O}\Big[(-\Delta f\cdot f)|\nabla u_1|^2
-(\Delta f\cdot u_2)(\nabla f\cdot \nabla (u_1+u_2)) \Big]\d x\d s
+\mathfrak q(v_1,v_2;u_1,u_2)\,,
\end{multline}
where $\mathfrak q:(\V^0_{[0,\kappa ]})^2\times(\U^1_{[0,\kappa ]})^2\to \R$ is defined as
\[
\mathfrak q:(v_1,v_2; u_1,u_2)\mapsto \iint_{[0,t]\times\O} v_1^j\partial _{ii}u_2\partial _ju_2 -v_1^j\partial _{ii}u_1\partial _ju_1 - v_2^j\partial _{ii}u_2\partial _ju_2 + v_2^j\partial _{ii}u_1\partial _ju_1\d x\d s\,.
\]
Now, a crucial observation is that
\[
\mathfrak t(v_1,v_2;u_1,u_2)+\mathfrak q(v_1,v_2;u_1,u_2)
=\iint_{[0,t]\times\O} [-g^j\partial _{ii}u_1\partial _jf + v_1^j\partial _{ii}f\partial _jf]\d x\d s,
\]
and hence, similar computations as above show that
\begin{multline}
\label{stg_3}
|\mathfrak t(v_1,v_2;u_1,u_2)+\mathfrak q(v_1,v_2;u_1,u_2)|
\\
\leq \epsilon \int_0^t(|\nabla g|^2_{L^2} + |\Delta f|^2_{L^2})\d s 
+ C(\epsilon )\int_0^t(|\Delta u_1|_{L^2}^2 + |v_1|_{L^4}^4 + 1)(|g|^2_{L^2} + |\nabla f|^2_{L^2})\d s\,.
\end{multline}

The remaining term can be estimated as follows:
\begin{equation}
\label{stg_4}
|\iint_{[0,t]\times\O} (-\Delta f\cdot f)|\nabla u_1|^2\d x\d s|
\leq\epsilon \int |\Delta f|^2_{L^2}\d s +C(\epsilon ) \int|\nabla u_1|_{L^4}^4|f|_{L^\infty}^2\d s\,.
\end{equation} 

Concerning the stochastic integral, we have for any $m>1:$
\begin{equation}
 \label{stg_5}
\E\left[\sup_{s\in[0,t]}|X(s)|^m \right]
\leq C(m,|\nabla \C|_{\mathscr L_2(L^2,L^2)})\E\left[\int_0^t|\nabla f|_{L^2}^2\d s\right]^{m/2}.
\end{equation}

Now, fix $k\geq 1$ and for simplicity in the notations let 
\[
(\vs_\ell ,\us_\ell ):= (\vs_{\ell ,k},\us_{\ell ,k})\,.
\]
Denote by 
\[
g_\ell:=\vs-\vs_\ell,\quad \text{and}\quad f_\ell:=\us-\us_\ell\,.
\]
From \eqref{cv4}, we infer that
\begin{align}
\label{almost_ev_1}
&(\vs_\ell,\nabla \us_\ell)\to (\vs,\nabla \us),\quad \P\otimes \d t\enskip \text{almost everywhere in}\enskip L^p,\enskip  p\in[1,\infty)\,.
\\
\label{almost_ev_2}
&f_\ell\to 0\,,\quad \P\otimes \d t\enskip\text{almost everywhere in}\enskip L^\infty.
\end{align}
Next, upon choosing $\epsilon >0$ sufficiently small,
we can sum up the inequalities \eqref{stg_1}, \eqref{stg_2}, \eqref{stg_3}, \eqref{stg_4} and \eqref{stg_5} to obtain
\begin{multline}
\label{pre-vitali}
\E\left[\esssup_{t\in[0,T]}(|g_\ell(t)|_{L^2}^2 + |\nabla f_\ell(t)|^2)
+\int_0^T(|\nabla g_\ell|^2_{L^2}+|\Delta f_\ell|_{L^2}^2)\d t
\right]
\\
\leq 
C\E\Big[\int _0^T \Big(\left|\nabla (\vs_\ell ,\vs)\right|_{L^2}^2 + \left|(\vs_\ell,\vs)\right|_{L^4}^4 + \left|\Delta (\us_\ell,\us )\right|_{L^2}^2 
+  \left|\nabla (\us_\ell,\us )\right|_{L^4}^4 +1\Big)
\\
\times(|g_\ell|_{L^2}^2+ |\nabla f_\ell|_{L^2}^2 + |f_\ell|_{L^\infty}^2)\d t
\Big]\,.
\end{multline}

We can now proceed to the proof of the main result of this paragraph.

\begin{proof}[Proof of Claim \ref{pro:improvement}]
We first need to show the following uniform integrability Property.\\

\hfill\begin{minipage}{\dimexpr\textwidth-2em}
\hspace{1em}{\em There is a full measure set $\Omega_1\in\mathcal A,$ such that for any $\omega \in\Omega _1$
	the family $(t\mapsto|\nabla \vs_\ell (\omega ,t)|^2_{L^2}+|\Delta  \us_\ell (\omega ,t)|^2_{L^2})_{n \in \N}$ is uniformly integrable with respect to the Lebesgue measure $\d t.$
	More precisely, for every $\omega \in\Omega _1$ and  $\epsilon >0,$ there corresponds a $\delta (\epsilon,\omega ) >0$ such that:
	\begin{equation}
	\label{uniform_int} 
	\int_s^t (|\nabla \vs_\ell (\omega ,r)|^2_{L^2}+|\Delta \us_\ell (\omega ,r)|_{L^2}^2)\d r
	\leq \epsilon \enskip ,
	\end{equation} 
	for every $(s,t)\in[0,T]^2$ such that $0\leq |t-s|\leq \delta (\epsilon,\omega ).$
}
\end{minipage}\\

Prior to show \eqref{uniform_int}, note that $\C_{n_\ell }$ is the covariance associated with $\Ws_\ell $ through \eqref{covariance_law}. Thus
	\begin{equation}
	\label{ass:covariance}
	\mathrm{cov}(\Ws_\ell (1),\Ws (1))= \C_{n_\ell } \C^* \,.
	\end{equation}
	
Next, for each $\ell  \in \N,$ and for every $(\omega ,t)\in\Omega \times[0,T],$  by estimating as in \eqref{pre:L2H2}, we see that there exists a constant $C(\rho ,|\C|_{\hs{1}})>0$ such that $\P$-a.s.,
	for any $0\leq s \leq t\leq T:$
	\begin{multline}
	 \label{pre_unif}
	\int_s^t|\nabla \vs_\ell (r)|^2_{L^2}+|\Delta \us_\ell (r)|_{L^2}^2\d r
	\\
	\leq 
	C\sup_{\theta_1,\theta _2 \in[s,t]}|X_\ell (\theta _1)-X_\ell (\theta _2)|
	+C(t-s)\left(1+\sup_{\theta \in[0,T]}|\nabla \us_\ell (\theta )|^2_{L^2}\right),
	\end{multline}
	where we denote by 
	\[
	X_\ell (\theta ):=\int_0^\theta \langle \nabla \us_\ell ,\us_\ell \times \d \nabla \Ws_\ell \rangle.
	\]
	From \eqref{pre_unif} and Proposition \ref{pro:global_estimates} we see that, in order to show the claimed uniform integrability, it is sufficient to show the existence of $\Omega_1 \in\mathcal A$ with full measure such that for each $\omega \in\Omega_1 :$ 
	\begin{equation}
	\label{converse_ascoli}
	X_\ell (\omega ,\cdot )\enskip \text{is uniformly equi-continuous}\,.
	\end{equation}
	
	But from \cite[Thm.\ 4.27 p.\ 103 \& Thm.\ 3.15 p.\ 77]{DPZ} we see that, denoting by $X(\theta ):=\int_0^\theta \langle \nabla \us ,\us \times \d \nabla \Ws \rangle,$ we have by \eqref{ass:covariance}:
	\begin{multline*}
	\E\left[\sup_{0\leq t\leq T}|X_\ell (t)-X (t)|\right]
	\leq C \E\Big[(\int_0^T|\C_{n_\ell }^* \div(\us \times\nabla (\us_\ell - \us )|_{L^2}^2\d r)^{1/2}
	\\
	+(\int_0^T|\C_{n_\ell }^* \div((\us_\ell -\us )\times\nabla \us )|^2_{L^2}\d r)^{1/2}
	+(\int_0^T|(\C_{n_\ell }^* -\C^* )\div(\us \times\nabla \us )|^2_{L^2}\d r)^{1/2}\Big].
	\end{multline*}
	By H\"older Inequality together with \eqref{cv4}, we infer that the above right hand side converges to $0$ as $n ,p \to\infty.$ 
	This shows in particular the existence of $\tilde \Omega \in \mathcal A$ with full probability such that 
	\[
	\omega \in\tilde \Omega
	\quad \Rightarrow\quad 
	X_\ell (\omega ,\cdot )-X(\omega ,\cdot ) \to 0\quad \text{strongly in}\enskip C(0,T;\R).
	\]
	From the converse of Ascoli-Arzela Theorem, we see that \eqref{converse_ascoli} holds if one lets $\Omega _1:=\tilde \Omega.$ This proves uniform integrability and \eqref{uniform_int}.

\bigskip

	Now, making use of Proposition \ref{pro:interp2}, we infer immediately that a similar uniform integrability as that of \eqref{uniform_int} holds for the family
$t\mapsto |\vs_\ell(t)|_{L^4}^4+|\nabla \us_\ell (t) |^4_{L^4}.$
Therefore, using Vitali's convergence theorem, we see thanks to \eqref{almost_ev_1}, \eqref{almost_ev_2} and \eqref{uniform_int} that
\begin{multline}
\label{as:vitali}
\int _0^T (|\nabla (\vs_\ell ,\vs)|_{L^2}^2 + |(\vs_\ell ,\vs)|_{L^4}^4 + |\Delta (\us_\ell ,\us)|_{L^2}^2 +  |\nabla (\us_\ell ,\us)|_{L^4}^4 +1)
\\
\times(|g_\ell|_{L^2}^2+ |\nabla f_\ell|_{L^2}^2 + |f_\ell|_{L^\infty}^2)\d t
\longrightarrow 0 \,,\quad 
\P\text{-a.s.:}
\end{multline}
By dominated convergence, we further see that \eqref{as:vitali} also holds in $L^1(\Omega ).$

Finally, we deduce \eqref{cvgence} from the inequality \eqref{pre-vitali}. Claim \ref{pro:improvement} is now proved.
\end{proof}

As a consequence of the above analysis, we observe the following.
\begin{corollary}
\label{cor:improvement}
The whole sequence $(u_n,v_n;\zeta _n)_{n\in \N}$ converges in probability to a local strong solution $(u,v;\zeta ),$ in the space $\V^0\times\U^1\times[0,T].$
Namely, for any $\epsilon >0$ it holds
\[
\P\Big(\|v_{n,k}-v_k\|_{\V^0(0,\zeta _k)} + \|u_{n,k}-u_k\|_{\U^1(0,\zeta _k)} + |\zeta _{n,k}-\zeta _n|>\epsilon \Big)\to_{n\to\infty} 0\,.
\]
\end{corollary}

\begin{proof}
{\itshape Step 1: convergence of the processes.}
Using the same arguments as above, it is immediately checked that given two sequences $m_\ell$ and $n_\ell$,
the sequence of random variable defined by
\[
(X_\ell ,Y_\ell ):=((\mathbf v_{n_\ell },\mathbf u_{n_\ell },\boldsymbol\zeta _{n_\ell }),
(\mathbf v_{m_\ell },\mathbf u_{m_\ell },\boldsymbol\zeta _{m_\ell }))
\]
contains a subsequence converging in distribution to a law $\nu$ supported in $E\times E$ where $E$ denotes the Polish space $\prod_{k\in\N}\V_{T}^0\times\prod_{k\in\N}\U_{T}^1\times\prod_{k\in\N}[0,T].$

Reasoning as in the above paragraph, we see that there exists a filtered probability space $\mathfrak{\tilde P}=(\tilde\Omega,\mathcal{\tilde A},\mathbb{\tilde P};\tilde W)$ endowed with a Wiener process satisfying \eqref{covariance_law}, and a sequence of random variables $(\tilde X_\ell ,\tilde Y_\ell )\sim (X_\ell ,Y_\ell ),$ converging $\mathbb{\tilde P}$-almost surely to an element
\[
(\tilde X,\tilde Y):=\left(\left(\mathbf {\vs},\mathbf{\us},\boldsymbol{\Hat\zeta}\right),
\left(\vvs,\uus,\boldsymbol{\Hat{\Hat \zeta }}\right)\right)\in L^0(\tilde\Omega ;E\times E)
\]
Furthermore, for any $k\in\N,$ the processes $(\vs_k,\us_k)$ and $(\vvs_{k},\uus_k)$ are both martingale solutions defined on the same stochastic basis, up to the stopping time $\kappa _k:=\Hat \zeta _k\wedge \Hat{\Hat{\zeta }}_k.$ By Theorem \ref{thm:uniqueness}, we infer that $(\vs_k,\us_k)|_{[0,\kappa _k]}=(\vvs_k,\uus_k)|_{[0,\kappa _k]}.$ 

~\paragraph{\itshape Step 2: convergence of the stopping times}
We now claim that 
\begin{equation}
\label{equality:zeta}
\kappa _k=\Hat\zeta _k=\Hat{\Hat{\zeta}}_k= \inf\left\{t\in [0,T]:\enskip \sup_{x\in\O}\int_{B(x,\rho _k)}|\nabla \us(t,y)|^2\d y\geq \varepsilon_1\right\}.
\end{equation}
This is indeed a consequence of Claim \ref{pro:improvement}: for any $k\in\N,$ since by \eqref{cvgence} the convergence $\us_{\ell ,k}\to _{\ell \to\infty}\us_k$ is strong in $C(H^1),$ we can write
\begin{multline*}
\sup_{t\in[0,\kappa _k]}\left|\sup_{x\in\O}\int_{B(x,\rho _k)}|\nabla \us_{\ell ,k}(t,y)|^2\d y-\sup_{x\in\O}\int_{B(x,\rho _k)}|\nabla \us_{k}(t,y)|^2\d y
\right|
\\
\leq\sup_{t\in[0,\kappa _k]}\sup_{x\in\O}\int_{B(x,\rho )}\left||\nabla \us_{\ell ,k}(t,y)|^2-|\nabla \us_{k}(t,y)|^2\right|\d y
\to0\,,
\end{multline*}
showing that 
\[
\Hat \zeta _k\equiv\mathrm{a.s.-}\lim_{\ell \to\infty}\Hat \zeta _{\ell ,k}=\zeta (\vs_k,\us_k),
\]
where the above r.h.s.\ denotes the stopping time defined from the solution $(\vs_k,\us_k)$ by \eqref{nota:zeta}. This shows the claimed property.

~\paragraph{Conclusion}
As a consequence, one sees that $(X,Y)$ is supported on the diagonal of $E\times E.$ By the well known Gy\"ony-Krylov Lemma \cite{gyongy1996existence}, this implies that the original sequence $(\mathbf v_{n},\mathbf u_{n},\boldsymbol \zeta _{n})_{n\in\N}$ converges in probability to an element $(\mathbf v,\mathbf u,\boldsymbol \zeta ),$ in the space 
$\prod_{k\in\N}\V^{\delta -1}_{T}\times\prod_{k\in\N}\U_T^\delta \times \prod_{k\in\N}[0,T].$
\end{proof}

\subsection{Controlling the bubbling time}
\label{sec:positivity}

We start with a lemma.

\begin{lemma}
\label{lem:uniform}
Consider a local strong solution $(v,u)$ of \eqref{SEL} on $[0,\tau )\subset(0,T],$ where $\tau $ is a stopping time.
For all $\rho >0,$ and $\delta \in[0,T],$ define the random variable
\[
\varepsilon(\delta ,\rho ):=\sup_{t\in[0,\delta\wedge \tau ),\,x\in\O}\int_{B(x,\rho )}\left(|v (t,y)|^2+|\nabla u (t,y)|^2\right)\d y,
\]
and assume that for some constant $\varepsilon_1>0,$ it holds
\begin{equation}
\label{hyp:epsilon}
\P(\varepsilon  (0,3\rho ) <\varepsilon_1)=1\enskip .
\end{equation} 
Furthermore, denote by $\tilde \pi $ the zero-mean pressure term obtained from equation \eqref{variational_pi}.

Then, for all $\delta \in[0,T]$ and every $\lambda >0,$ we have the estimate
\begin{multline}
\label{est:P_eps}
\P\big(\left\{\varepsilon  (\delta ,\rho ) \geq \varepsilon_1+\lambda \right\}\cap\left\{\delta <\tau  \right\}\big)
\\
\leq \frac{C(\rho ,T,E_0,|\C|_{\hs{1}})}{\lambda }\E \Big[\delta  + \int_0^\delta \big(|v |_{L^4}^4 + |\nabla u |_{L^4}^4+|\tilde\pi |^{4/3}_{L^{4/3}}\big)\d s \Big]
\end{multline}
where the above constant depends on the indicated quantities, but not on the individual element $(v ,u )$ such that $\E[\sup_{t\in[0,T]}(|v(t)|_{L^2}^2 + |\nabla u(t)|^2_{L^2})]\leq E_0.$
\end{lemma}

The proof of the above Lemma relies partially on the following argument, whose proof can be found in \cite{H1}:\\

\hfill\begin{minipage}{\dimexpr\textwidth-2em}
\hspace{1em}{\em For every $\rho >0,$ there exist 
a positive integer $N_\rho$, and points $\{x_\rho ^{1},x_\rho ^2,\dots,x_\rho ^{N_\rho }\}\subset \O$ fulfilling the property:
\begin{equation}
\label{coveringArg}
``\text{for all}\enskip x\in\O,\enskip\text{there exists}\enskip i\in\{1,\dots, N_\rho \}\enskip \text{with}\enskip  B(x,\rho) \subset B(x_\rho ^i,2 \rho)\,."
\end{equation}
}
\end{minipage}

\begin{proof}[Proof of Lemma \ref{lem:uniform}.]
Let  $x^i_\rho ,$ $1\leq i\leq N_\rho $ be as in \eqref{coveringArg}, and
consider test functions $\varphi=\varphi _{\rho,i}\in C^\infty_0(\O)$ with $\spt\varphi\subset B(x^i_\rho ,3 \rho)$
such that
\begin{equation}
\label{nota:varphi_spt}
\mathds{1}_{B(x_\rho ^i, 2\rho)}\leq \varphi\leq \mathds{1}_{B(x_\rho ^i, 3\rho)}\, ,
\quad 
\sup\limits_{x\in B(x^i_\rho,3\rho)}|\nabla \varphi(x)|\leq \frac{C}{\rho}\, ,
\end{equation}
for some $C>0$ independent of $i,\rho $.

By \eqref{hyp:epsilon} together with \eqref{nota:varphi_spt}, we have for every $i\in\{1,\dots,N_\rho \}:$
\[
E ^{i,\rho } (0):=
\E [|\varphi _{\rho ,i}v (0)|^2_{L^2}+ |\varphi _{\rho ,i}\nabla u (0)|_{L^2}^2]<\varepsilon_1
\]
By \eqref{coveringArg},
we see that for $\delta \in[0,\tau),$ it holds
\begin{equation}
\begin{aligned}
\label{series_inclusions}
\big\{\varepsilon  (\delta ,\rho )\geq \varepsilon_1+\lambda \big\}
&\subset
\big\{\exists i, \;  1\leq i \leq N_\rho,\enskip \sup_{t\in [0, \delta]} E ^{i,\rho } (t )\geq \varepsilon_1+\lambda\big\}
\\
&\subset
\cup_{i=1}^{N_\rho }
\big\{\sup_{t\in [0, \delta]}E ^{i,\rho }(t )-E^{i,\rho } (0)\geq \lambda \big\}.
\end{aligned} 
\end{equation}
But from Proposition \ref{pro:local_estimates},
we have:
\begin{equation}
\begin{aligned}
\label{local_dissipation}
\frac{1}{2}\big(E^{i,\rho } (t )-
&E^{i,\rho }(0) \big)
\\
&\leq C (1+ \frac{1}{\rho^2})\iint_{[0,t ]\times\mathbb T^2}\big(|v|^3+(|v|+1)|\nabla u|^2
+|v||\pi-\bar \pi|\big)\d s 
\\
&\quad \quad \quad 
+ X^{i,\rho } (t ) 
+ t|\varphi\nabla \C|_{\mathscr L_2(L^2,L^2)}^2
\\
&=:V^{i,\rho } (t),
\end{aligned}
\end{equation} 
for $t\in[0,\tau ],$ a.s.,
where we let $X^{i,\rho } (t ):=\int_0^t\langle\varphi\nabla u,\varphi u\times\nabla \d W\rangle.$

Next, from \eqref{series_inclusions}, we have the estimate
\begin{equation}
\label{bd:markov}
\begin{aligned}
\P(\varepsilon  (\delta ,\rho )\geq \varepsilon_1\ + \lambda )
&\leq \P\Big(\exists i\in\{1,\dots,N_\rho \},\,\sup_{t\in [0, \delta]}E^{i,\rho }(\delta ) -E(0)^{i,\rho } \geq \lambda \Big)
\\
&\leq\sum\nolimits_{i=1}^{N_\rho }\P\Big(\sup_{t\in[0,\delta]} V^{i,\rho}(t)\geq \frac{\lambda }{2}\Big)
\\
&\leq \sum\nolimits_{i=1}^{N_\rho }\frac{2}{\lambda }\E\Big[\sup_{t\in[0,\delta]} V^{i,\rho}(t)\Big].
\end{aligned}
\end{equation} 
Using H\"older Inequality, and then Burkholder-Davis-Gundy Inequality to estimate the martingale term, 
we end up with the bound
\begin{multline}
\label{ineq:localControlEn}
\E\left[\sup_{0\leq t\leq \delta }V^{i,\rho}(t)\right]
\\
\leq
C (1+\frac{1}{\rho ^2})
\E\Bigg[\delta \sup_{t\in[0,\inf(\delta,\tau))}(|v(t)|_{L^2}^2+|\nabla u(t)|_{L^2}^2)
+\int_0^\delta (|v|_{L^3}^3+|v|^2_{L^2})\d s
\\
+\int_0^\delta |\nabla u|_{L^4}^4\d s 
+\left(\int_0^\delta |v|^4_{L^4}\d s\right)^{1/4}
\left(\int_0^\delta |\tilde\pi|_{L^{4/3}}^{4/3}\d s\right)^{3/4}
\Bigg]
\\
+C(|\C|_{\hs{1}})\E\left[\iint_{[0,\delta ]\times\O}(\varphi^{i,\rho })^4 |\nabla u|^2\d s\right]^{1/2}
+\delta |\varphi^{i,\rho }\nabla \C|_{\mathscr L_2(L^2,L^2)}^2
\\
\leq 
C(\rho ,T,E_0,|\C|_{\hs{1}})\E \Big[\delta  + \int_0^\delta \big(|v|_{L^4}^4 + |\nabla u |_{L^4}^4+|\tilde \pi|^{4/3}_{L^{4/3}}\big)\d s \Big]
\end{multline}
Inserting this estimate into \eqref{bd:markov}, we obtain \eqref{est:P_eps}.
\end{proof}

\subsection{End of the proof of Theorem \ref{thm:main}}
\label{sec:end_proof}

\begin{trivlist}
\item[\em Step 1.] We first prove the following assertion:\\

\hfill\begin{minipage}{\dimexpr\textwidth-2em}
{\em there exists an integer $k_0$ such that for each $k\geq k_0,$
the limit point $\zeta _k$ of the sequence $\{\zeta_{\ell ,k} ,\ell \in\N\}$ defined in \eqref{zeta_tightness} verifies
\begin{equation}
\label{positivity}
\P(\zeta _k>0)=1\, .
\end{equation}
}
\end{minipage}

To prove \eqref{positivity}, observe that since the sequence $\rho _k$ converges to $0$ and since the sequence
$(v_{n,k}(0),u_{n,k}(0))$ converges to $(v_k,u_k)(0)$  strongly in $\L^2\times \HH^1$,
there is an integer $k_0$ such that for each $k\geq k_0:$
\begin{equation}\label{k_geq_k0}
\sup_{x\in\O}\int_{B(x,3 \rho_k) }\Big(|v_{n ,k}(0)|^2+|\nabla u_{n ,k}(0)|^2\Big)\d y\leq \frac{\mu _1^{-1}}{2}\, ,\enskip \text{uniformly in}\enskip n \in\N\,.
\end{equation}
For such $k\geq k_0$ we have in particular
\begin{equation}\label{eval:P_kappa_lk}
\P(\zeta_{k} >0)=1-\P(\zeta _{k} =0)=1-\lim_{\delta \to 0}\P(\zeta _k\leq\delta )
\end{equation}
and, according to the definition of $\zeta _{n,k}$ in \eqref{zeta_tightness}, and $\tau_n$ in \eqref{deftaun},
\begin{multline}
\label{strict}
\left\{\zeta_{n,k} \leq \delta\right\}\subset\left\{\enskip \zeta_{n,k} <\tau_n\enskip\text{and}\enskip \zeta_{n,k} \leq\delta\right\}
\cup\left\{\zeta_{n,k}=\tau_n\enskip \text{and}\enskip \tau_n\leq \delta \right\}
=:\Omega_1\cup\Omega_2\,.
\end{multline}
However, by Theorem \ref{thm:bootstrap}, we have
$\P(\Omega_2)=0$ and therefore, Lemma \ref{lem:uniform} with $\varepsilon_1=\frac12 \mu_1^{-1}$ yields:
\begin{multline*}
\P\big(\{\zeta _{n,k}\leq \delta \}\big)=
\P\big(\{\varepsilon(\delta ,\rho _k)\geq \mu_1^{-1}\}\cap \{  \zeta_{n,k} <\tau_n\}\big)
\\
\leq 2\mu _1C(k,T,E_0,|\C|_{\hs{1}})\E\Big[
\delta +\int_0^\delta \big(|v_{n,k}|_{L^4}^4 + |\nabla u_{n,k}|_{L^4}^4+|\tilde\pi _{n,k}|^{4/3}_{L^{4/3}}\big)\d s 
\Big]\,.
\end{multline*}

Estimating the $L^{4/3}_t(L^{4/3}_x)$-norm of $\tilde \pi$ by its $L^2_t(L^2_x)$-norm, and making use of the estimate \eqref{estim:pi_0}, we further obtain
\[
\P\big(\{\zeta _{n,k}\leq \delta \}\big)
\leq C(\mu _1,k,T,E_0,|\C|_{\hs{1}})\E\Big[
\delta +\int_0^\delta \big(|v_{n,k}|_{L^4}^4 + |\nabla u_{n,k}|_{L^4}^4\big)\d s 
\Big]\,,
\]
which by \eqref{estim:moments}, \eqref{bound:L2H2} and \eqref{ineq:interp} goes to $0$ as $\delta \to0,$ uniformly in $n\in\N.$

Now, writing that for each $n \in\N$:
\[\left\{\zeta _k\leq \delta\right\}\subset\left\{|\zeta _{n ,k}-|\zeta _{n ,k}-\zeta _k||\leq\delta\right\}\, ,\]
it is seen that
\[
 \P\left(\zeta _k\leq \delta\right)\leq \P\left(\zeta _{n ,k}\leq2\delta \right)+\P\left(|\zeta _{n ,k}-\zeta _k|\geq \delta\right)\, .
\]
The conclusion follows by the fact that
$|\zeta _{n ,k}-\zeta _k|\overset{\P}{\rightarrow}0$ as $n \to\infty$,
and the uniform convergence of $\P(\zeta _{n ,k}\leq\delta )$ as $\delta \to0$.
This shows \eqref{positivity}.

~\paragraph{\itshape Step 2. Definition of the solution}
We now define
\begin{equation}
\label{nota:struwe}
(v_\star,u_\star)(v_0,u_0;t):=
\begin{cases}
 (v_k(t),u_k(t))\enskip \text{if}\enskip t\in[0,\zeta _k)\enskip \text{for some}\enskip k\in\N.
\\
0\enskip \text{otherwise.}
 \end{cases}
\end{equation}
Taking the limit as $n\to \infty$ in the identity $\zeta _{n,k}\leq\zeta _{n,k+1}$ (which holds because the sequence $\rho _k$ is non-increasing, see \eqref{zeta_tightness}), one obtains that for each $k\in \N$:
\[
\zeta _k\leq \zeta _{k+1},
\]
and hence by Theorem \ref{thm:uniqueness}, one sees that the definition \eqref{nota:struwe} is not ambiguous.
Hence, we have defined a local strong solution $(v_\star(v_0,u_0),u_\star(v_0,u_0);\zeta _\star(v_0,u_0)),$ where
\[
\zeta _\star(v_0,u_0):=\lim_{k\to\infty}\zeta _k=\sup_{k\in\N}\zeta _k\,.
\]
The proof of the property \ref{prop:3}, as well as the fact that the random variable $J$ is finite a.s., follows exactly the same steps as that of \cite{H1}, and therefore we leave it to the reader.

It is sufficient to show \ref{prop:2} for $j=0,$ hence we let for $k\in\N:$
\[
V_k:=v_k(\sigma _k)
\quad \text{and}\quad 
U_k:=u_k(\sigma _k)\,,
\]
where $\sigma _k\nearrow \tau ^1:= \zeta _\star(v_0,u_0),$ $\P$-a.s.
Notice by Proposition \ref{pro:global_estimates} that for any $m\geq 1,$ the sequence $(V_k)_{k\in\N}$ is bounded in $L^m(\Omega ;\L^2)$, while $(U_k)_{k\in\N}$ is bounded in $L^m(\Omega ;\HH^1).$
On the other hand using the equation and the fact that $\zeta _k$ converges a.s. to $\zeta _\star$,
it is easily seen that $(V_k,U_k)$ is a Cauchy sequence in $L^2(\Omega ;\H^{-1})\times L^2(\Omega ;\LL^2).$
In particular, the sequence $(V_k,U_k)$ can have only one limit point in $L^m(\Omega ;\L^2)\times L^m(\Omega ;\HH^1),$
hence showing that the whole sequence converges towards a uniquely determined couple $(U,V).$
This shows \ref{prop:2}.

In order to prove \ref{prop:4}, by Theorem \ref{thm:uniqueness} and the above step, it is sufficient to show that for each $j\in\{0,\dots , J-1\},$
\begin{equation}
\label{belongs_to}
(v,u)\enskip \text{belongs to}\enskip 
L^4(\Omega ;L^4(\tau _j,\tau _{j+1}';\L^4\times \WW^{1,4}))\,,
\end{equation} 
where $\tau _j'$ is an arbitrary stopping time such that $\tau _j\leq \tau _{j+1}'<\tau _{j+1}.$ But \eqref{belongs_to} holds thanks to \eqref{ineq:interp}.
This finishes the proof of Theorem \ref{thm:main}.
\hfill\qed
\end{trivlist}

\appendix
\section{Appendix: computation of some trace terms}
\label{app:trace}
In this appendix, we denote by $(v,u)$ a strong solution of \eqref{SEL}, and we furthermore assume that it is supported in $\V_T^2\times\U_T^3.$
By the expression ``trace term'' we refer to It\^o-Stratonovitch corrections of the form \eqref{ito_strato}.

Recalling the notation \eqref{canonical_R3},
it is easy to see that for any vector $\zeta \in\R^3:$
\begin{equation}
  \label{geometric_fact}
 \sum_{i=1}^3 (\zeta \times \ee_i)\times \ee_i
 = -2\zeta .
\end{equation}

Similarly, the norm constraint yields that for any index $\alpha \in \{1,2\}:$
\begin{equation}
 \label{geometric_fact_2}
\partial _\alpha \frac{|u|^2}{2}=\partial _\alpha u\cdot u \equiv0.
\end{equation} 

We can now proceed to the computations.

\begin{proof}[Proof of \eqref{trace:1}]
Using \eqref{geometric_fact}, and recalling \eqref{nota:psi_k}, \eqref{psi_l_j},
we have for the first term
\begin{equation}
\label{computations_A1}
\begin{aligned}
 A^1(t)
 &=\frac12\sum_{l,j,\alpha }\langle \partial _\alpha u, (u\times \CC^{j}_l)\times\partial _\alpha \CC^j_l\rangle 
\\
 &=-\sum_{l,j,\alpha }\langle \partial _\alpha u, u (\C_l\partial _\alpha \C_l)\rangle 
\\
&=0,
\end{aligned}
\end{equation} 
where we have also used \eqref{geometric_fact_2}.

Now, from the skew-symmetry of the vector product and properties \eqref{geometric_fact} 
and \eqref{geometric_fact_2}, we have for the second term:
\begin{equation}
\label{computations_A2}
\begin{aligned}
 A^2(t)
 &\equiv\frac12\sum_{l,j,\alpha }\big [\langle u\times\partial _\alpha  \CC^j_l,
\partial _\alpha u\times \CC^j_l \rangle
+\langle u\times\partial _\alpha  \CC^j_l,u\times \partial _\alpha \CC^j_l \rangle
\big ]
\\
&=-\frac12\sum_{l,j,\alpha }\big [
\langle (u\times\partial _\alpha  \CC^j_l)\times \CC^j_l,\partial _\alpha u \rangle
+2\langle (u\times\partial _\alpha  \CC^j_l)\times \partial _\alpha \CC^j_l,u \rangle
\big ]
\\
&=\sum_{l,\alpha }\big [
\langle u (\partial _\alpha \C_l)\C_l,\partial _\alpha u\rangle 
+2|\partial _\alpha  \C_l|^2_{L^2}
\big ]
\\
&=2\sum_{l}|\nabla \C_l|_{L^2}^2,
\end{aligned}
\end{equation}
which is the claim.
\end{proof}

We now aim to compute the It\^o-Stratonovitch correction for the semimartingale $Y$ defined in \eqref{nota:Y}.
We have, using coordinates:
\begin{multline}
\label{decomp_M_app}
Y(t)= 2\int_0^t \langle \ttt,\partial _\beta u\times\circ \d \partial _\beta W\rangle
+\int_0^t\langle \ttt,u\times\circ\d  \Delta W\rangle
-\int_0^t\langle \ttt, v^i(u\times\circ\d  \partial _iW)\rangle
\\
+2\int_0^t\langle \Delta u,\partial _\beta u\times\circ\d \partial _\beta W\rangle
+\int_0^t\langle \Delta u,u\times\circ\d \Delta W\rangle
=:\sum\nolimits_{\gamma =1}^5Y^\gamma (t)\,,
\end{multline}
where we recall that
$\ttt=\tt_u-v\cdot \nabla u= \Delta u+u|\nabla u|^2 -v\cdot \nabla u$.

\begin{claim}
\label{clm:M}
Under the condition $\C\in\hs{4},$ the process $Y$ is a well-defined semi-martingale, 
whose decomposition writes as
\begin{multline}
\label{ito_cor}
Y(t)
=
\sum\nolimits_{l\in\N}\iint_{[0,t]\times\O}\Big[
2(u\times\partial _\alpha \CC_l^j\cdot \partial _\alpha u)
(u\cdot \partial _\beta u\times\partial _\beta \CC_l^j)
\\
+4(\partial _\alpha \C_l \partial _\beta \C_l )\partial _\alpha u\cdot \partial _\beta u
+|\Delta \C_l |^2  -(v\cdot \nabla \C_l )^2
\\
+|\nabla u|^2|\nabla \C_l|^2
+2\sum_{\alpha ,\beta }(\partial _\alpha \C_l \partial _\beta  \C_l )\partial _\alpha u\cdot \partial _\beta u
+|\Delta \C_l |^2
\Big]\d x\d s
+\hat Y(t)
\end{multline}
where $\hat Y(t)$ is a martingale.

Moreover, its quadratic variation is estimated above as
\begin{multline}
\label{quad:Y}
\langle Y\rangle(t)
\leq 2
\int_0^t\Big\{
|\C^*\div(u\times\nabla \ttt)|_{L^2}^2
+|\C^*\Delta (u\times\ttt)|_{L^2}^2
+|\C^*\partial _i (v^iu\times\ttt)|_{L^2}^2
\\
+|\C^*\Delta (-u\times \Delta u)|_{L^2}^2
+|\C^*\div (u\times \nabla \Delta u)|_{L^2}^2
\Big\}\d s
\end{multline}
\end{claim}

\begin{proof}
{\itshape Computation of $Y^1 .$}
Similarly as above, we have, recalling \eqref{ito_strato},
\[
\begin{aligned}
&Y^1(t)-2\int_0^t \langle \ttt,\partial _\beta u\times\d  \partial _\beta W\rangle
\\
&=\sum\nolimits_{l ,j}\int_0^t\Big[\langle \ttt,(\partial _\beta u\times\CC_l^j)\times\partial _\beta \CC_l^j\rangle
+\langle \ttt,(u\times\partial _\beta \CC_l^j)\times\partial _\beta \CC_l^j\rangle
\\
&\quad \quad \quad \quad 
+\langle \Delta u\times\CC_l^j ,\partial_\beta u \times\partial_\beta\CC_l^j\rangle
+2\langle \partial _\alpha u\times\partial _\alpha \CC_l^j ,\partial_\beta u \times\partial_\beta\CC_l^j\rangle
\\
&\quad \quad \quad \quad 
+\langle u\times\Delta \CC_l^j,\partial_\beta u \times\partial_\beta\CC_l^j\rangle
+\langle u\times\CC_l^j|\nabla u|^2,\partial_\beta u \times\partial_\beta\CC_l^j\rangle
\\
&\quad \quad 
+2\langle u(\partial _\alpha u\times \CC_l^j\cdot \partial _\alpha u) ,\partial_\beta u \times\partial_\beta\CC_l^j\rangle
+2\langle u(u\times \partial _\alpha \CC_l^j\cdot \partial _\alpha u),\partial_\beta u \times\partial_\beta\CC_l^j\rangle
\\
&\quad \quad \quad 
-\langle v^i \partial _iu\times\CC_l^j,\partial_\beta u \times\partial_\beta\CC_l^j\rangle
-\langle v^iu\times\partial _i\CC_l^j,\partial_\beta u \times\partial_\beta\CC_l^j\rangle
\Big]\d s
\\
&=\sum\nolimits_{l }\int_0^t\Big[\sum_{j}\left\langle \ttt,(\partial _\beta u\times\CC_l^j)\times\partial _\beta \CC_l^j - (\partial_\beta u \times\partial_\beta\CC_l^j)\times\CC_l^j\right\rangle
\\
&\quad \quad \quad 
+4\langle \partial _\alpha u\partial _\alpha \C_l,\partial_\beta u \partial_\beta\C_l\rangle
+2\langle u(u\times \partial _\alpha \CC_l^j\cdot \partial _\alpha u),\partial_\beta u \times\partial_\beta\CC_l^j\rangle
\Big]\d s
\\
&=\sum\nolimits_{l }\int_0^t\Big[
4\langle \partial _\alpha u\partial _\alpha \C_l,\partial_\beta u \partial_\beta\C_l\rangle
+2\langle u(u\times \partial _\alpha \CC_l^j\cdot \partial _\alpha u),\partial_\beta u \times\partial_\beta\CC_l^j\rangle
\Big]\d s
\,,
\end{aligned}
\]
thanks to \eqref{geometric_fact}.

\bigskip

\paragraph{\itshape Computation of $Y^2 $}
We have
\[
\begin{aligned}
&Y^2(t)-\int_0^t \langle \ttt,u\times\d \Delta W\rangle
\\
&=\frac12\sum\nolimits_{l ,j}\int_0^t\Big[\langle \ttt,(u\times\CC_l^j)\times\Delta \CC_l^j\rangle
+\langle \Delta u\times\CC_l^j ,u\times\Delta\CC_l^j\rangle
\\
&\quad \quad \quad \quad \quad 
+2\langle \partial _\alpha u\times\partial _\alpha \CC_l^j ,u\times\Delta\CC_l^j\rangle
+\langle u\times\Delta \CC_l^j,u\times\Delta\CC_l^j\rangle
\\
&\quad \quad \quad 
+\langle u\times\CC_l^j|\nabla u|^2,u\times\Delta\CC_l^j\rangle
+2\langle u(\partial _\alpha u\times \CC_l^j\cdot \partial _\alpha u) ,u\times\Delta\CC_l^j\rangle
\\
&\quad \quad \quad \quad 
+\langle u(u\times \partial _\alpha \CC_l^j\cdot \partial _\alpha u),u\times\Delta\CC_l^j\rangle
-\langle v^i \partial _iu\times\CC_l^j,u\times\Delta\CC_l^j\rangle
\\
&\quad \quad \quad \quad \quad 
\quad \quad \quad \quad \quad 
-\langle v^iu\times\partial _i\CC_l^j,u\times\Delta\CC_l^j\rangle
\Big]\d s
\\
&=\sum_{l }\int_0^t\Big[
\int_\O|\Delta \C_l|^2\d x
-\langle v^i\partial _i\C_l ,\Delta\C_l\rangle
\Big]\d s\,.
\end{aligned}
\]

\paragraph{\itshape Computation of $Y^3 $}
Proceeding similarly, we have
\[
\begin{aligned}
&Y^3(t)-\int_0^t\langle \ttt,v^iu\times\d \partial _iW\rangle
\\
&=
\frac12\sum\nolimits_{l ,j}\int_0^t\Big[\langle \ttt,v^i(u\times\CC_l^j)\times\partial _i\CC_l^j \rangle
+\langle \Delta u\times\CC_l^j,v^iu\times\partial _i\CC_l^j\rangle
\\
&\quad \quad \quad 
+2\langle \partial _\alpha u\times\partial _\alpha \CC_l^j,v^iu\times\partial _i\CC_l^j\rangle
+\langle u\times\Delta \CC_l^j,v^iu\times\partial _i\CC_l^j\rangle
\\
&\quad \quad \quad 
+\langle u\times\CC_l^j |\nabla u|^2,v^iu\times\partial _i\CC_l^j\rangle
+2\langle u(\partial _\alpha u\times\CC_l^j \cdot \partial _\alpha u),v^iu\times\partial _i\CC_l^j\rangle
\\
&\quad \quad \quad \quad 
+\langle u(u\times\partial _\alpha \CC_l^j\cdot \partial _\alpha u),v^iu\times\partial _i\CC_l^j\rangle
-\langle v^k\partial _ku\times\CC_l^j,v^iu\times\partial _i\CC_l^j\rangle
\\
&\quad \quad \quad \quad \quad 
\quad \quad \quad \quad \quad 
-\langle v^ku\times\partial _k\CC_l^j,v^iu\times\partial _i\CC_l^j\rangle
\Big]\d s
\end{aligned}
\]
Simplifications lead finally to
\[
\begin{aligned}
&Y^3(t)-\int_0^t\langle \ttt,v^iu\times\d \partial _iW\rangle
\\
&=
\frac12\sum_{l }\int_0^t\Big[
\langle \Delta u\cdot u,\C_lv^i\partial _i\C_l\rangle
+2\langle \Delta \C_l ,v^i\partial _i\C_l \rangle
+2\langle  |\nabla u|^2,v^i(\partial _i\C_l)\C_l \rangle
\\
&\quad \quad 
-2\langle v^k\partial _k\C_l ,v^i\partial _i\C_l \rangle
\Big]\d s
\\
&=
\sum_{l}\int_0^t\Big[
\langle \Delta \C_l ,v^i\partial _i\C_l \rangle
-\langle v^k\partial _k\C_l ,v^i\partial _i\C_l \rangle
\Big]\d s\,.
\end{aligned}
\]

\bigskip

\paragraph{\itshape Computation of $Y^4.$}
We have
\[
\begin{aligned}
&Y^4(t)-\int_0^t\langle \Delta u,\partial _\beta u\times\d \partial _\beta W\rangle
\\
&=\frac12\sum_{l ,j}\int_0^t\Big[
\langle \Delta u,(\partial _\beta u\times\CC_l^j)\times\partial _\beta \CC_l^j\rangle
+\langle \Delta u,(u\times\partial _\beta \CC_l^j)\times\partial _\beta \CC_l^j\rangle
\\
&\quad \quad \quad \quad 
+\langle \Delta u\times\CC_l^j,\partial _\beta u\times\partial _\beta \CC_l^j\rangle
+2\langle \partial _\alpha u\times\partial _\alpha \CC_l^j,\partial _\beta u\times\partial _\beta \CC_l^j\rangle
\\
&\quad \quad \quad \quad \quad 
\quad \quad \quad \quad \quad 
+\langle u\times\Delta \CC_l^j,\partial _\beta u\times\partial _\beta \CC_l^j\rangle
\Big]\d s,
\\
&=\frac12\sum_{l }\int_0^t\Big[
2\langle |\nabla u|^2,(\partial _\beta \C_l)^2\rangle
+4\langle \partial _\alpha u\partial _\alpha \C_l ,\partial _\beta u\partial _\beta \C_l\rangle
\Big]\d s.
\end{aligned}
\]

\paragraph{\itshape Computation of $Y^5$}
Similarly:
\[
\begin{aligned}
&Y^5(t)-\int_0^t\langle \Delta u,u\times\d \Delta W\rangle
\\
&=\frac12\sum_{l,j}\int_0^t\Big[
\langle \Delta u,(u\times\CC_l^j)\times\Delta \CC_l^j\rangle
\\
&\quad \quad \quad \quad 
+\langle \Delta u\times\CC_l^j ,u\times\Delta \CC_l^j\rangle
+2\langle \partial _\alpha u\times\partial _\alpha \CC_l^j,u\times\Delta \CC_l^j\rangle
\\
&\quad \quad \quad \quad \quad 
\quad \quad \quad \quad \quad 
+|u\times\Delta \CC_l^j|_{L^2}^2
\Big]\d s
\\
&=\frac12\sum_{l}\int_0^t
2|\Delta \C_l |_{L^2}^2
\d s\,.
\end{aligned}
\]

Summing the above contributions, we end up with \eqref{ito_cor}.

\bigskip

\paragraph{\itshape Quadratic variation of $Y$}
The estimate \eqref{quad:Y} is a simple consequence of \cite[Theorem 4.27]{DPZ} applied to each $Y^\gamma ,\gamma =1,\dots,5,$ together with the fact that
$\langle Y\rangle(t)\leq C\sum_{\gamma =1}^5\langle Y^\gamma \rangle(t)\,.$
This finishes the proof.
\end{proof}

\end{document}